\documentclass[generic,preprint]{imsart}
\RequirePackage{natbib}

\startlocaldefs

\usepackage{amssymb,amsmath,amsfonts,amsthm,amscd,amstext,mathtools}
\usepackage[utf8]{inputenc}
\usepackage{xspace,enumerate}
\usepackage{graphicx}
\usepackage{subcaption}
\usepackage[pdftex]{hyperref}
\usepackage[charter]{mathdesign}
\usepackage[linecolor=blue]{todonotes}

\usepackage[a4paper,scale={0.72,0.74},marginratio={1:1},footskip=7mm,headsep=10mm]{geometry}

\numberwithin{equation}{section}

\newtheorem{theorem}{Theorem}[section]
\newtheorem{atheo}{Theorem}

\newtheorem{claim}[theorem]{Claim}
\newtheorem{lemma}[theorem]{Lemma}
\newtheorem{proposition}[theorem]{Proposition}
\newtheorem{corollary}[theorem]{Corollary}
\newtheorem{remark}[theorem]{Remark}
\newtheorem{definition}[theorem]{Definition}

\long\def\xcom#1{}

\newcommand{\cA}{{\ensuremath{\mathcal A}} }

\newcommand{\cC}{{\ensuremath{\mathcal C}} }

\newcommand{\cE}{{\ensuremath{\mathcal E}} }
\newcommand{\cF}{{\ensuremath{\mathcal F}} }

\newcommand{\cL}{{\ensuremath{\mathcal L}} }

\newcommand{\cO}{{\ensuremath{\mathcal O}} }

\newcommand{\cS}{{\ensuremath{\mathcal S}} }

\newcommand{\cV}{{\ensuremath{\mathcal V}} }
\newcommand{\cW}{{\ensuremath{\mathcal W}} }

\newcommand{\gep}{\varepsilon}       

\renewcommand{\tilde}{\widetilde}          
\DeclareMathSymbol{\leqslant}{\mathalpha}{AMSa}{"36} 
\DeclareMathSymbol{\geqslant}{\mathalpha}{AMSa}{"3E} 
\DeclareMathSymbol{\eset}{\mathalpha}{AMSb}{"3F}     

\newcommand{\R}{\mathbb{R}}

\newcommand{\Z}{\mathbb{Z}}
\newcommand{\N}{\mathbb{N}}

\def\bs{\boldsymbol}

\DeclareMathOperator{\sign}{sign}

\newcommand\bP{\ensuremath{\bs{\mathrm{P}}}}
\newcommand\bE{\ensuremath{\bs{\mathrm{E}}}}

\newcommand{\ind}{{\sf 1}}

\renewcommand{\epsilon}{\varepsilon}

\newenvironment{myitemize}{%
\begin{list}{$\bullet$}%
        {%
        \setlength{\itemsep}{0.4em}%
        \setlength{\topsep}{0.5em}%
        \setlength\leftmargin{2.45em}%
        \setlength\labelwidth{2.05em}%
        \setlength{\labelsep}{0.4em}%
        }%
        }%
{\end{list}}

\renewenvironment{itemize}{
\begin{myitemize}}%
{\end{myitemize}}

 \newcommand{\be}[1]{\begin{equation}\label{#1}}
 \newcommand{\ee}{\end{equation}}

 \newcommand{\bl}[1]{\begin{lemma}\label{#1}}
 \newcommand{\el}{\end{lemma}}

 \newcommand{\br}[1]{\begin{remark}\label{#1}}
 \newcommand{\er}{\end{remark}}

 \newcommand{\bt}[1]{\begin{theorem}\label{#1}}
 \newcommand{\et}{\end{theorem}}

 \newcommand{\bd}[1]{\begin{definition}\label{#1}}
 \newcommand{\ed}{\end{definition}}

 \newcommand{\bcl}[1]{\begin{claim}\label{#1}}
 \newcommand{\ecl}{\end{claim}}

 \newcommand{\bp}[1]{\begin{proposition}\label{#1}}
 \newcommand{\ep}{\end{proposition}}

 \newcommand{\bc}[1]{\begin{corollary}\label{#1}}
 \newcommand{\ec}{\end{corollary}}

 \newcommand{\bpr}{\begin{proof}}
 \newcommand{\epr}{\end{proof}}

 \newcommand{\bi}{\begin{itemize}}
 \newcommand{\ei}{\end{itemize}}

\newcommand{\esp}[1]{\mathbb{E}\etc{#1}}
\newcommand{\ens}[1]{\left\{#1\right\}}

\newcommand{\prob}[1]{\mathbb{P}\etp{#1}}
\newcommand{\etc}[1]{\left [#1 \right ]}
\newcommand{\etp}[1]{\left (#1 \right )}
\newcommand{\valabs}[1]{\left|#1 \right|}

\newcommand{\unsur}[1]{\frac{1}{#1}}

\newcommand{\probmubeta}[1]{\bP_{\beta,\mu_\beta}\etp{#1}}

\newcommand{\undemi}{\frac12}
\newcommand{\un}[1]{1_{\etp{#1}}}

\date{\today}

\endlocaldefs

\begin{document}

\begin{frontmatter}
\title{A shape theorem for the scaling limit of the IPDSAW at criticality
}
\runtitle{Shape Theorem for critical IPDSAW}

\begin{abstract} \footnote{\today}

In this paper we give a complete characterization of the 
scaling limit of the critical Interacting Partially Directed Self-Avoiding Walk (IPDSAW) introduced in \cite{ZL68}.  As the system size $L\in \N$ diverges, 
we prove that the set of occupied sites, rescaled horizontally by $L^{2/3}$ and vertically by $L^{1/3}$  converges in law for the Hausdorff distance towards a
non trivial random set. This limiting set is built with a Brownian motion $B$ conditioned to come back at the origin at $a_1$ the time at which its  
geometric area reaches $1$. The modulus of $B$ up to $a_1$ gives the height of the limiting set, while 
its center of mass process is an independent Brownian motion.

Obtaining the shape theorem requires to derive a functional central limit theorem for the excursion of a random walk with Laplace symmetric increments conditioned 
on sweeping a prescribed geometric area.  This result is proven in a companion paper 
\cite{CarPet17b}.
%

%
%
  
\end{abstract}

\author{\fnms{Philippe}
  \snm{Carmona}\corref{}\ead[label=e1]{philippe.carmona@univ-nantes.fr}
\ead[label=u1,url]{http://www.math.sciences.univ-nantes.fr/~carmona/}}


\and
\author{\fnms{Nicolas} \snm{Pétrélis}\ead[label=e2]{nicolas.petrelis@univ-nantes.fr}}
\affiliation{Université de Nantes}
\affiliation{Université de Nantes 
and Universidad de Chile}
\affiliation{Université de Nantes}

\address{Laboratoire de Math\'ematiques Jean Leray UMR 6629\\
Universit\'e de Nantes, 2 Rue de la Houssini\`ere\\
BP 92208, F-44322 Nantes Cedex 03, France\\ \printead{e1}\\\printead{e2}
}

\runauthor{P. Carmona et al.}



\thanks{{\it Acknowledgements.}  The author thanks the Centre Henri Lebesgue ANR-11-LABX-0020-01 for creating an attractive mathematical environment.}

\end{frontmatter}

\tableofcontents

\section{Introduction and results}

Deriving the scaling limit of a polymer model at its critical point is a difficult issue that had been tackled so far in \cite{DGZ05} or in \cite{S13} for wetting models  and in \cite{CD09} for a Laplacian pinning-model.  With the present paper, we display the scaling limit 
of the critical-IPDSAW. It is a Shape Theorem whose limiting object is a truly $2$-dimensional random set.


\subsection{The model}\label{s11}

The interacting partially directed self-avoiding walk (IPDSAW)  is a \emph{self-avoiding} random walk on $\mathbb{Z}^2$ that only takes unitary steps \emph{upwards, downwards and to the right}. Thus, the set of allowed $L$-step paths is
\begin{align}\label{defWL}
\nonumber\mathcal{W}_L=\{w=(w_i)_{i=0}^L\in(\mathbb{N}_0\times\mathbb{Z})^{L+1}:\,&w_0=0,\, w_L-w_{L-1}=\rightarrow,\\
\nonumber&w_{i+1}-w_i\in\{\uparrow,\downarrow,\rightarrow\}\;\, \forall 0\leq i<L-1,\\
\nonumber&w_i\neq w_j\;\,\forall i<j\}.
\end{align}
Any non-consecutive vertices of the walk though adjacent on the lattice are called \textit{self-touchings} and an energetic reward $\beta\geq0$ is assigned to each trajectory for each self-touching. Thus, every random walk trajectory $w=(w_i)_{i=0}^L\in\mathcal{W}_L$ is associated with the Hamiltonian 
\begin{equation}\label{eq:Hal}
H_{L}(w):=\sum_{\substack{i,j=0\\i<j-1}}^L\mathbf{1}_{\{\lVert w_i-w_j\rVert=1\}},
\end{equation}
which allows us to define $P_{L,\beta}$ the polymer law in size $L$ as, 
\begin{equation}\label{polmes}
P_{L,\beta}(w)= \frac{ e^{\beta H_{L}(w)}}{Z_{L,\beta}}, \quad w\in \cW_L, 
 \end{equation}
where $Z_{L,\beta}$ is the normalizing constant known as the partition function of the system. The exponential growth rate of the partition function is captured by the 
free energy of the model, i.e., 
$f(\beta)=\lim_{L\to \infty}\frac{1}{L} \log Z_{L,\beta}$.
\smallskip

 The IPDSAW undergoes a collapse transition at some
$\beta_c$ that is explicitly known  (see e.g. \cite{BGW92} or \cite[Theorem 1.3]{NGP13}) and  the phase diagram is partitioned into an extended phase $\cE=[0,\beta_c)$ inside which the free energy is larger than $\beta$
and a collapsed phase $\cC=[\beta_c,\infty)$ where the free energy equals $\beta$.  The asymptotics of the free energy close to criticality are
analyzed in \cite[Theorem B]{CNGP13} where the phase transition is proven to be second order with a critical exponent $3/2$, i.e., 
$f(\beta_c-\epsilon)=\beta_c-\epsilon+\gamma \epsilon^{3/2}+o(\epsilon^{3/2})$ where the pre-factor $\gamma$  is closely related with a
continuous model built with Brownian trajectories that are penalized energetically depending on their geometric area.  

In \cite{CNGP13} and \cite{CP15}, a rather complete description of the main geometric features of a typical path sampled from $P_{L,\beta}$   is provided 
inside the extended phase $(\beta<\beta_c)$ and inside the collapsed phase $(\beta>\beta_c)$ (see the discussion in Section \ref{discuss} below).  
However, the scaling limit of the model at criticality ($\beta=\beta_c$), which is the most delicate case, was still to be derived and this is the object of the present paper.

\subsection{Main result: the limiting shape of IPDSAW at criticality}

We identify each  $w\in \Omega_L$ with a connected compact subset of $\R^2$ denoted by $S(\omega)$ that extends the sites of $\Z^2$
occupied by $w$ to squares of length $1$, i.e.,  
\be{defS}
S(w)=\Big\{ \cup_{i=0}^{L} w(i)+\big[-\tfrac12,\tfrac12\big]^2\Big\}, \quad w\in \cW_L.
\ee 
For
$(v_1,v_2)\in (0,\infty)^2$, we let $T_{v_1,v_2}$ be a rescaling operator that acts on $\mathfrak S$ the set of closed subsets of $\R^2$ endowed with the Hausdorff distance. For $S\subset \R^2$, the set $T_{v_1,v_2}(S)$ is
obtained after rescaling $S$ by $v_1$ horizontally and by $v_2$ vertically, i.e.,
\be{defscal}
T_{v_1,v_2}(S)=\bigg\{ \bigg (\frac{x}{v_1},\frac{y}{v_2} \bigg)\colon \, (x,y)\in S  \bigg\}.
\ee
For  $(\alpha_{1},\alpha_{2})$ in $[0,1]$, we denote by 
$Q_{L,\beta}^{\alpha_1,\alpha_2}$ the law of $T_{L^{\alpha_{1}},L^{\alpha_{2}}}(S(w))$ seen as a random variable 
on $\mathfrak{S}$ endowed with its Borel $\sigma$-algebra and when $w$ is sampled from $P_{L,\beta}$. 

With Theorem \ref{Theo-shape} 
below, we prove that, at criticality, the IPDSAW rescaled in time by $L^{2/3}$ and in space by $L^{1/3}$ converges in distribution towards a non trivial random set, built 
with the help of two independent Brownian motions.

%
%
\begin{atheo}[Shape Theorem]\label{Theo-shape}
For $\beta=\beta_c$, we have 
\begin{equation}
Q_{L,\beta}^{\frac{2}{3},\frac{1}{3}}\xrightarrow[L\to \infty]{d} \cS_{\text{crit}}(B,D)
\end{equation} 
with $\cS_{\text{crit}}(B,D)$ a random subset of $\R^2$ defined as
\begin{equation}\label{defcS}
\cS_{\text{crit}}(B,D)=\bigg\{(x,y)\in [0,a_1]\times \R\colon\,  D_x-\frac{|B_x|}{2}\leq y \leq D_x+\frac{|B_x|}{2}\bigg\}
\end{equation}
where
$B$ and $D$ are independent Brownian motions of variance $\sigma_\beta^2$ (defined below \eqref{lawP}), where $a_1$ is the time at which the geometric area described by $B$ reaches $1$, that is,  $\int_0^{a_1} |B_u| du=1$ and with $B$ conditioned  on the 
event $B_{a_1}=0$.
\end{atheo}
Let us say a few words about the 4 main challenges that we faced to prove Theorem \ref{Theo-shape}. Thanks to the representation Theorem 
\ref{transf21}, everything boils down to studying a random walk $V$ conditioned on having a prescribed large geometric area. To be more specific, we need to consider the joint convergence of a couple of processes:
the profile $|V|$ (corresponding to $|B|$ in Theorem \ref{Theo-shape}) and the center-of-mass walk $M$ (corresponding to $D$ in Theorem \ref{Theo-shape}).
\begin{enumerate}
\item Proving the convergence of time-changed discrete processes with an implicit time-change to corresponding time-changed continuous processes.
\item Handling the fluctuations of the center-of-mass walk $M$ on the excursions of the profile $|V|$. The main difficulty is that these are  not independent at fixed time horizon, although 
we shall prove that they are asymptotically independent.
\item Extending the pioneering work of \cite{DKW13} to obtain local limit theorems for a $3$ component process, i.e., an excursion of the profile conditioned on having a large extension, the associated center-of-mass walk and the geometric areas. 
This issue is settled in \cite{CarPet17b}.
\item Adapting to our needs  the reconstruction procedure introduced in \cite{DGZ05}.
\end{enumerate}


\begin{figure}[h]
    \centering
    \begin{subfigure}[b]{0.3\textwidth}
        \includegraphics[width=\textwidth]{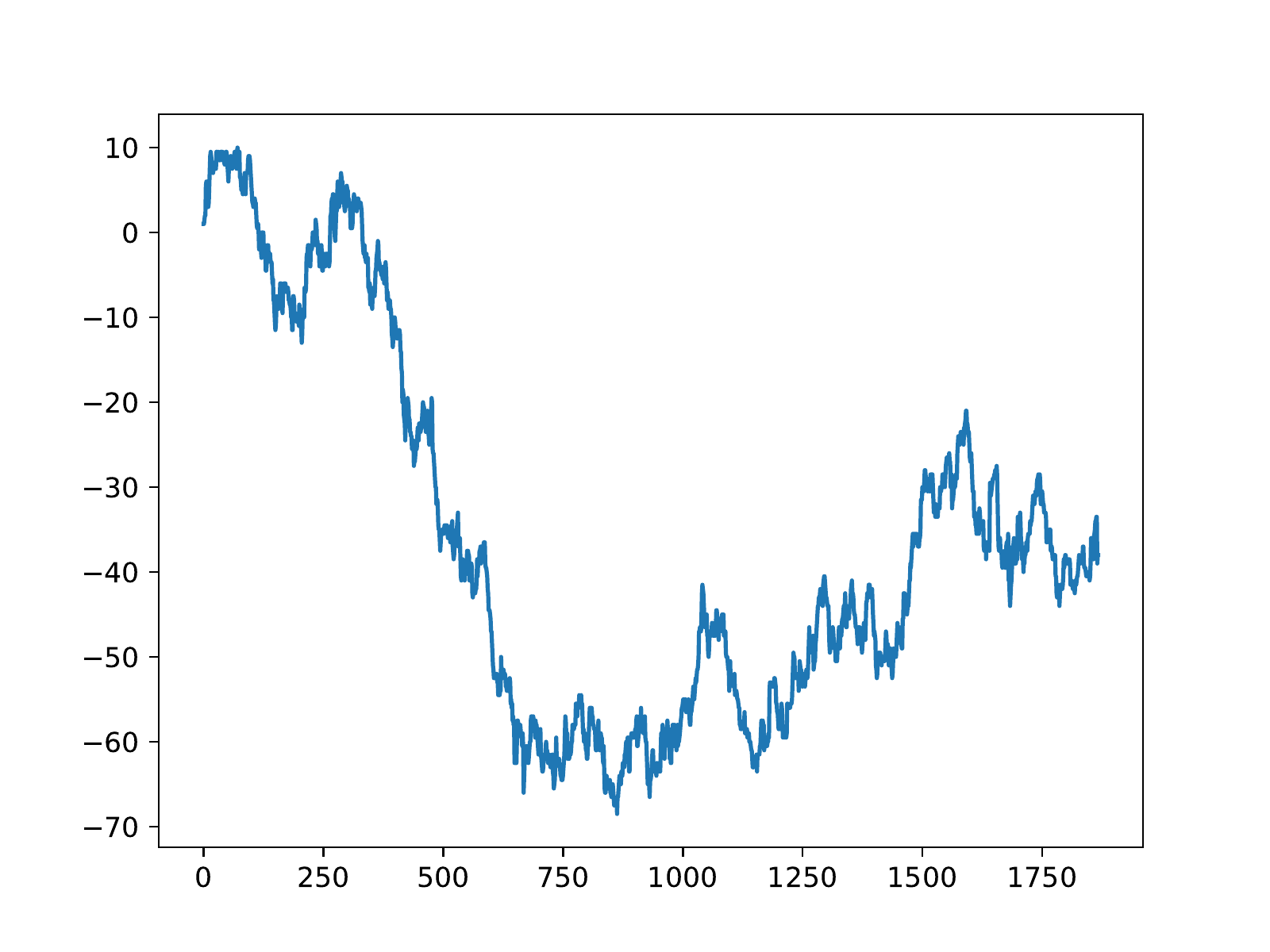}
        \caption{Center of Mass Walk}
        \label{fig:gull}
    \end{subfigure}
    ~ 
    \begin{subfigure}[b]{0.3\textwidth}
        \includegraphics[width=\textwidth]{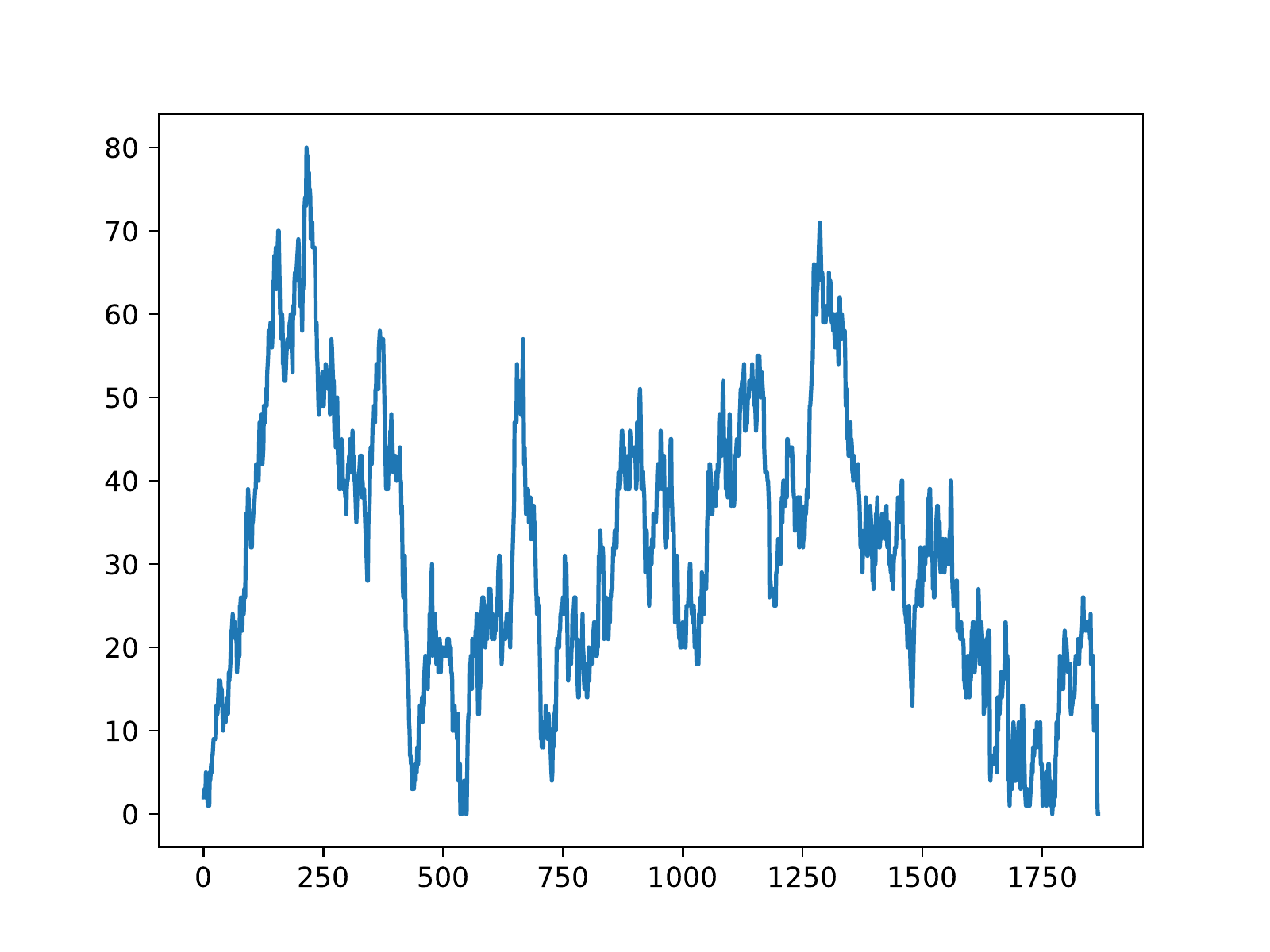}
        \caption{Profile}
        \label{fig:tiger}
    \end{subfigure}
    ~ 
    \begin{subfigure}[b]{0.3\textwidth}
        \includegraphics[width=\textwidth]{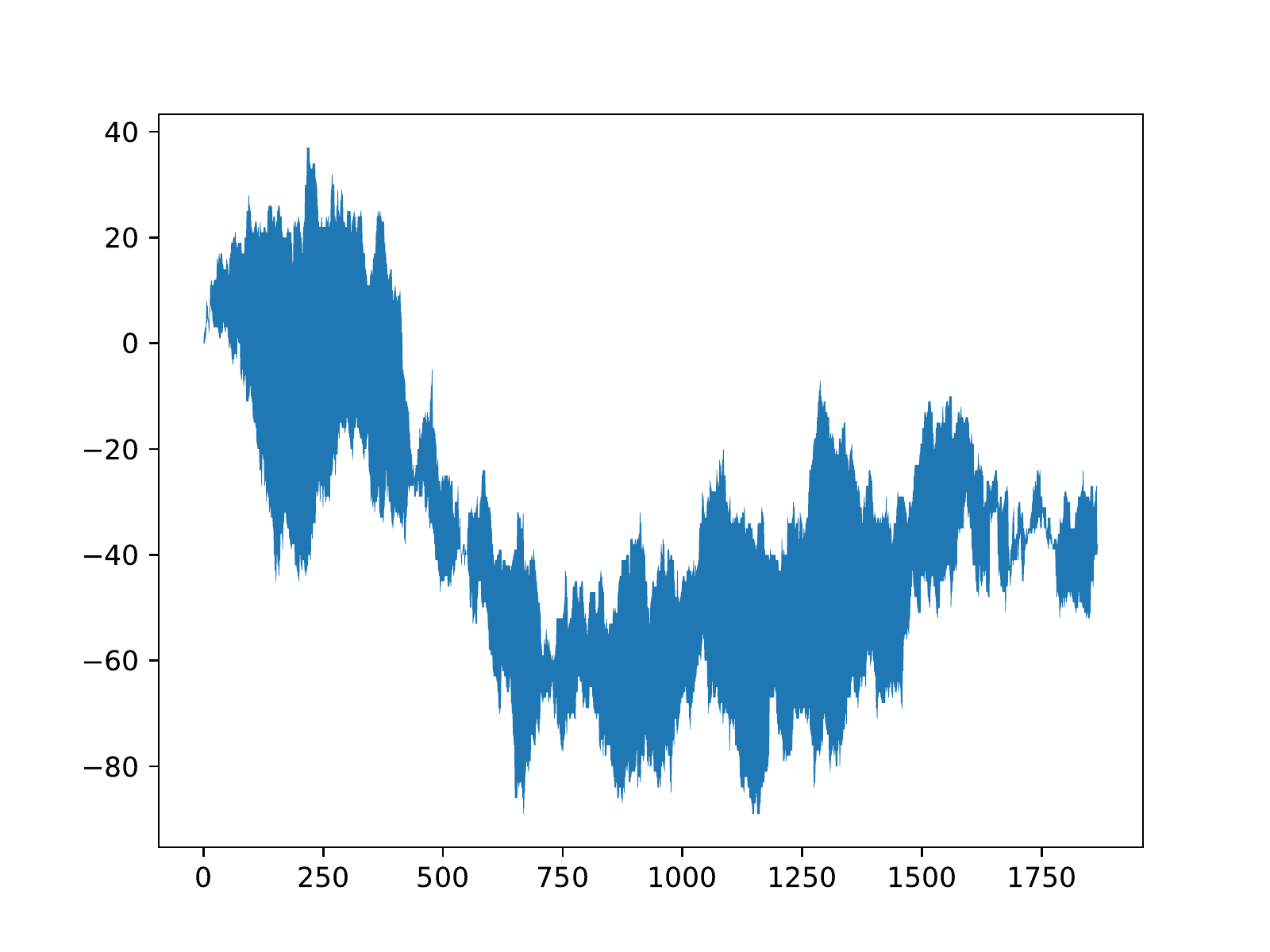}
        \caption{Enveloppe of IPDSAW}
        \label{fig:mouse}
    \end{subfigure}
    \caption{Critical IPDSAW, length $L=60000$, exact simulation}\label{fig:animals}
\end{figure}

  

\subsection{Reminder: scaling limits in the non-critical regimes}\label{discuss}
In the present section, we will explain why the shape Theorem stated above completes the picture of the scaling limit of IPDSAW initiated in  \cite{CNGP13} and \cite{CP15}. To that aim we need first to recall the stretch representation of the model, and then to associate with every configuration its profile and its center-of-mass walk, from which the occupied set in \eqref{defS} can be reconstructed. With these tools in hand we will briefly recall the scaling limits obtained in \cite{CNGP13} and \cite{CP15} concerning the extended and the collapsed regime of IPDSAW. We will terminate this section by explaining why the critical regime (that is the object of the present paper)  is more delicate than the others.  
\smallskip

\noindent {\it Stretch  description of a path.} There is a natural representation of any path in $\cW_L$ as a collection of oriented vertical stretches separated by one horizontal step. Thus, we set 
$\Omega_L:=\bigcup_{N=1}^L\mathcal{L}_{N,L}$, where $\mathcal{L}_{N,L}$ is the set of all possible configurations consisting of $N$ vertical stretches that have a total length $L$, that is
\begin{equation}\label{defLL}
\textstyle\mathcal{L}_{N,L}=\Bigl\{l\in\mathbb{Z}^N:\sum_{n=1}^N|l_n|+N=L\Bigr\}.
\end{equation}
A one to one correspondence between $\Omega_L$ and $\cW_L$ is obtained by 
associating with a given $l\in \Omega_L$ the path $w_l$  of $\cW_L$ that starts at $0$,
takes $|l_1|$ vertical steps north if $l_1>0$ and south if $l_1<0$, then takes one horizontal step, then takes $|l_2|$ vertical steps north if $l_2>0$ and south if $l_2<0$ then takes one horizontal step and so on...

For $N\in \{1,\dots,L\}$ and $l\in \cL_{N,L}$, the Hamiltonian associated with $w_l$ can be rewritten as 
\begin{equation}\label{hamm}
\textstyle H_L(w_l)=H_{L}(l_1,\ldots,l_N)=\sum_{n=1}^{N-1}(l_n\;\tilde{\wedge}\;l_{n+1})
\end{equation}
where
\begin{equation}\label{hammm}
x\;\tilde{\wedge}\;y=\begin{dcases*}
	|x|\wedge|y| & if $xy<0$,\\
  0 & otherwise.
  \end{dcases*}
\end{equation}
Thus,  the polymer measure in \eqref{polmes} becomes
\begin{equation}
P_{L,\beta}(l)= \frac{ e^{\beta H_{L}(w_l)}}{Z_{L,\beta}}, \quad l\in \Omega_L.
 \end{equation}
 We recall \ref{defS} and we denote by    
$S(l)$ the occupied set associated with any $l\in \Omega_L$ (i.e. $S(l)=S(w_l)$). We observe that $S(l)$ can be fully reconstructed with two auxiliary processes, i.e, the \emph{center-of-mass walk} $M_l$
and the \emph{ profile} $|l|$.
To be more specific, we associate with each $l \in \cL_{N,L}$ the profile $|l|=(|l_i|)_{i=0}^{N+1}$ (with $l_{N+1}=0$ by convention) and the center-of-mass walk $M_l=\big(M_{l,i}\big)_{i=0}^{N+1}$  that links the middles of each stretch consecutively, i.e.,   $M_{l,0}=0$  and 
\be{droitmi}
M_{l,i}=l_1+\dots+l_{i-1}+\frac{l_i}{2},\quad i\in \{1,\dots,N\},
\ee
and $M_{l,N+1}=l_1+\dots+l_N$. 

\smallskip

\noindent {\it Particularities of the critical regime.} 
A consequence of the fact that the occupied set $S(l)$ associated with $l\in \Omega_L$ can be recovered from its profile and center of mass walk is that  for every $(\alpha_1,\alpha_2)\in [0,1]^2$ the scaling limit of the rescaled occupied set $T_{L^{\alpha_1},L^{\alpha_2}}(S(l))$ (with $l$ a typical path sampled from $P_{L,\beta}$) can be derived from the scaling limit of  $(|l|, M_l)$ rescaled in time by $L^{\alpha_1}$ and in space by $L^{\alpha_2}$.
This is the strategy adopted in \cite{CNGP13} for the collapsed regime and in \cite{CP15} for the extended regime. In the extended regime (i.e., $\beta<\beta_c$), the horizontal extension of a typical path follows a law of large number of speed $L$ (that is $\alpha_1=1$), the vertical fluctuations of its
 center-of-mass walk are of order $\sqrt{L}$ (i.e., $\alpha_2=1/2$) whereas its vertical stretches are of finite size. Therefore, once rescaled vertically by $\sqrt{L}$ the profile 
vanishes whereas the center-of-mass walk displays a Brownian limit.  In other words,  once rescaled vertically by $\sqrt{L}$ and horizontally by $L$ and in the limit $L\to \infty$
the upper and lower envelopes of $T_{L,\sqrt{L}}(S(l))$ coalesce into a continuous trajectory whose law is that of a Brownian motion, i.e., 
one can straightforwardly deduce from \cite[Theorem 2.8]{CP15} that 
\begin{equation}
Q_{L,\beta}^{1,\frac{1}{2}}\xrightarrow[L\to \infty]{d} \big\{(s,\sigma_\beta B_s)\colon\; s\in [0,e_{\beta}]\big\}
\end{equation}
where $\sigma_\beta$ and $e_\beta$ are explicit constants and $B$ is a standard Brownian motion.
\smallskip

In the collapsed regime (i.e., $\beta>\beta_c$) the fraction of self-touching performed by a typical trajectory equals $1+o(1)$, which forces the vertical stretches to be long and with alternating signs. As a consequence the typical horizontal extension of a path sampled from $P_{L,\beta}$ is much shorter than its counterpart in the extended regime and follows a law of large number of speed $\sqrt{L}$ (i.e., $\alpha_1=1/2$). The typical length of vertical stretches is $\sqrt{L}$ as well (i.e., $\alpha_2=1/2$)  and the profile rescaled in time and space by $\sqrt{L}$ converges towards a deterministic Wulff shape. The center-of-mass walk, in turn, fluctuates with an amplitude $L^{1/4}$ and therefore vanishes when we rescale it in time and space by $\sqrt{L}$. Unlike the extended regime, inside the collapsed phase the scaling limit of $T_{\sqrt{L},\sqrt{L}}(S(l))$ is driven by the profile only and we recall  \cite[Theorem D]{CNGP13} which states that 
\begin{equation}
Q_{L,\beta}^{\frac{1}{2},\frac{1}{2}}\xrightarrow[L\to \infty]{d} \cS_\beta
\end{equation}
where $\cS_\beta$ is a deterministic Wulff-Shape, symmetric with respect to the $x$-axis.
\smallskip

In \cite[Theorem 2.2]{CP15} we proved that, at criticality,  the horizontal extension of a typical path follows a central limit theorem with speed $L^{2/3}$ and a limiting law corresponding to that of the random 
time $a_1$ at which the geometric area swept by a Brownian motion (of variance $\sigma_\beta^2$) reaches $1$ conditioned on the fact that the Brownian touches $0$ at $a_1$.  Thus, the last pending issue concerning the scaling limits of IPDSAW was to derive the scaling limit of the full path at criticality. This is the object of the present paper but let us insist on the fact that this is also the hardest issue. 
The reason is that, unlike the extended regime or the inside of the collapsed regime, at criticality  the profile and the center-of-mass walk display vertical fluctuations of the same order  (i.e. $L^{1/3}$).

\section{Organization of the proof}\label{organi}

The present Section is an extended outline of the proof of Theorem \ref{Theo-shape}. In Section \ref{rwcritp} we settle some notation to state Theorem \ref{transf21} which
sheds light on the fact that the critical-IPDSAW can be studied indirectly with the help of an auxiliary random walk conditioned on sweeping a prescribed geometric area.  
Then, in Section \ref{centerofmass} 
we state Theorem \ref{midl} which provides the scaling limits of the properly rescaled profile and center-of-mass walk for a typical configuration sampled from $P_{L,\beta}$. Theorem \ref{midl} actually implies Theorem \ref{Theo-shape}
but we will not prove Theorem \ref{midl} directly. As exposed carefully in Remark \ref{tzl}, we will rather apply a time change on both profile and center-of-mass walk to state Theorem \ref{midlaux} which implies Theorem \ref{midl}
but turns out to be easier to prove.   


\subsection{Random walk representation of IPDSAW at its critical point}\label{rwcritp}

The stretch representation of IPDSAW (displayed in Section \ref{discuss} above) was initially used in \cite{NGP13} to develop an alternative probabilistic approach of the model. This new approach involves an auxiliary random walk that we describe below before stating Theorem \ref{transf21} which enlightens  the particular relationship  between this random walk
and the model at criticality  (i.e., at $\beta=\beta_c$).
\smallskip

We  let $\mathbf{P}_{\beta}$ be the law of the random walk $V:=(V_n)_{n\in\mathbb{N}}$ starting from the origin and whose increments
$(U_i)_{i\in\mathbb{N}}$ are i.i.d and follow a discrete Laplace law, i.e.,
\begin{equation}\label{lawP}
\mathbf{P}_{\beta}(U_1=k)=\tfrac{e^{-\frac{\beta}{2}|k|}}{c_{\beta}}\quad\forall k\in\mathbb{Z}\quad\text{with}\quad c_{\beta}:=\tfrac{1+e^{-\beta/2}}{1-e^{-\beta/2}}.
\end{equation}  
and we set $\sigma_\beta^2:=\text{Var}_\beta(U_1)$.
For $L\in \N$ and $N\in \{1,\dots,L\}$ we set 
\be{defG}
\cV_{N,L-N}:=\{V\colon G_N(V)=L-N, V_{N+1}=0\} \quad \text{with} \quad  G_N(V)=\textstyle \sum_{i=0}^{N} |V_i|,
\ee
 and
we denote by $T_N$ the  one-to-one correspondence that maps 
$\cV_{N,L-N}$ onto $\cL_{N,L}$ as 
\be{defTN}
T_N(V)_i=(-1)^{i-1} V_i\quad  \text{for all} \quad i\in \{1,\dots N\}.
\ee 
For $n\in \N$ and for $V=(V_i)_{i=0}^\infty\in \Z^\N$  we  define 
 $K_n(V):=n+G_N(V)=\sum_{i=1}^n 1+|V_i|$ and its pseudo-inverse 
\begin{equation}\label{taul}
\xi_s=\inf\{i\geq 0\colon\, K_i(V)\geq s\}, \quad s\in [0,\infty).
\end{equation}  
We note incidentally  that \eqref{taul} implies 
$K_j=\max\{n\geq 1\colon\, \xi_n=j\}$ for $j\in \N_0$. With a slight abuse of notation (and for random walk trajectories $V$ only) we will call $K_n$ the geometric area swept by $V$ up to time $n$ 
although it would be more correct to call it geometric area plus extension.

With these notations in hand, we state the fundamental Theorem \ref{transf21} below.  With this Theorem, we claim that at criticality, studying the model IPDSAW is completely equivalent to 
studying the $V$ random walk conditioned on sweeping a prescribed geometric area.
\begin{atheo}[Random Walk Representation at criticality]\label{transf21}
\begin{equation}\label{transf2}
P_{L,\beta_c}\bigl(l\in\cdot \bigr)=\mathbf{P}_{\beta_c}\bigl(T_{\xi_L}(V)\in\cdot \mid V_{\xi_L+1}=0, \ K_{\xi_L}=L\bigr).
\end{equation}
\end{atheo}
This theorem will be proven in Section \ref{proofRWR}.

%


\subsection{Center-of-mass Walk and Profile} \label{centerofmass}

With every 
$l\in \Omega_L$, we associate $\tilde M_{l,L}$ and $|\tilde l |_L$    the cadlag processes on $[0,\infty)$ obtained by rescaling the center-of-mass walk $M_l$ and the profile $|l|$
by $L^{2/3}$ horizontally and by $L^{1/3}$ vertically, i.e.,
\begin{align}\label{cadlagM}
\tilde M_{l,L}(s)&=  \frac{1}{L^{1/3}} M_{l,\lfloor s L^{2/3}\rfloor\wedge N_l}, \quad s\in [0,\infty),\\
|\tilde l|_L(s)&= \frac{1}{L^{1/3}} |l_{\lfloor s L^{2/3}\rfloor\wedge N_l}|, \quad s\in [0,\infty),
\end{align}
where $N_l$ is the number of vertical stretches composing $l$ (i.e. $l\in \cL_{N_l,L}$).

We  denote by $R_{L,\beta}$ the law of $(|\tilde l|_L,\tilde M_{l,L} )$ with $l$ sampled from $P_{L,\beta}$ and we state Theorem \ref{midl} 
which claims that the rescaled profile and center-of-mass walk of a typical configuration of the critical-IPDSAW converge simultaneously towards
Brownian motions stopped at some particular random time. This Theorem is illustrated with Figure \ref{fig:animals}, where an 
exact simulation of the critical IPDSAW is provided in length $L=60000$. 
  
\smallskip

\begin{atheo} \label{midl}
At criticality ($\beta=\beta_c$)
we have 
\begin{equation}
R_{L,\beta}\xrightarrow[L\to \infty]{d} \big (|B_{s\wedge a_1}|,D_{s\wedge a_1}\big)_{s\in [0,\infty)}
\end{equation} 
where
$B$ and $D$ are independent Brownian motions of variance $\sigma^2_\beta$, where $a_1$ is the time at which the geometric area described by $B$ reaches $1$, that is,  $\int_0^{a_1} |B_u| du=1$ and with $B$ conditioned  on the 
event $B_{a_1}=0$.
\end{atheo}

We will prove in Section \ref{step1} that Theorem \ref{midl} implies Theorem \ref{Theo-shape}.  
For this reason, the target of the present paper will become to prove Theorem \ref{midl}, but let us first
recall Theorem \ref{transf21} which allows us to view $R_{L,\beta}$ as the law of two other 
cadlag processes built with the $V$ random walk. 
In this spirit, for a given random walk trajectory  $V$, we define $M=(M_i)_{i=0}^\infty$ the counterpart of the center-of-mass walk introduced in 
\eqref{droitmi} as 
\be{defM}
M_{i}=\sum_{j=1}^{i} (-1)^{j+1} \frac{U_j}{2},\quad i\in \N,
\ee
and with $M_0=0$. We let  $\widehat V_L$  and  $\widehat M_L$ be the cadlag processes obtained after rescaling $V$ and $M$ by $L^{2/3}$ in time and by $L^{1/3}$ in space and stopped at 
$\xi_L$ (recall \eqref{taul}), that is,

\begin{align}\label{defmchap}
\widehat V_L \colon  [0,\infty)&\mapsto \R \hspace{2.5 cm} \text{and} \quad \quad \quad  \widehat M_L \colon  [0,\infty)\mapsto \R\\
\nonumber s&\to L^{-\frac{1}{3}} \ V_{\lfloor sL^{2/3}\rfloor  \wedge \xi_L} \hspace{3.8 cm} s\to L^{-\frac{1}{3}} \ M_{\lfloor sL^{2/3}\rfloor  \wedge \xi_L}
\end{align}
A consequence of Theorem \ref{transf21} is that, 
\be{equalaw1}
R_{L,\beta_c}=_{\text{Law}} (| \widehat V_L|, \widehat M_L)
\ee
with $V$ sampled from $\mathbf{P}_{\beta_c}\bigl(\cdot \mid V_{\xi_L+1}=0, \ K_{\xi_L}=L\bigr)$. In the proof of Therorem \ref{midl} (see Section \ref{step2} below), we will use the representation of $R_{L,\beta_c}$ in 
\eqref{equalaw1}.

\smallskip

\noindent {\it Renewal structure}  
We introduce a renewal structure which roughly consists of the excursions of the $V$ random walk away from the origin and turns out to be a fundamental tool of our analysis.
To  that purpose, we define a sequence of stopping times $(\tau_k)_{k\in\N}$ similar to ladder times by the prescription $\tau_0=0$ and
\begin{equation}\label{premm}
  \tau_{k+1} = \inf\ens{i> \tau_k : V_{i-1}\neq 0 \text{ and }
    V_{i-1}V_i \le 0},
\end{equation}
so that the length of the $k$-th excursions is given by  
\begin{equation}\label{deuxx}
  \mathfrak{N}_k = \tau_k -\tau_{k-1} \quad(k\ge 1)\,,
\end{equation}
and the area swept
\be{trois}
\mathfrak{A}_k =\valabs{V_{\tau_{k-1}}} + \cdots + \valabs{V_{\tau_k-1}} \quad(k\ge 1)\,.
\ee
For each excursion we consider the sum of its length and its geometric area.   For this reason we define the quantity $X_i=\mathfrak{N}_i+\mathfrak{A}_i$ for $i \in \N$ and, with a slight abuse of notation, we will call $X_i$ the geometric area swept by the $i$-th excursion.  We set $S_0=0$ and $S_n=X_1 + \cdots + X_n$ for $n\geq 1$  so that we can define $\mathfrak{X}$ a random set of points on $\N_0$ as
\be{defX}
\mathfrak{X}=\{S_n, n\in \N_0\}.
\ee
We will also need to consider $v_L$  the number of excursions that have been completed by $V$ when its geometric area reaches $L$, i.e.,    $v_L:=\max\{i\geq 0\colon S_i\leq L\}$.

\begin{remark}\label{tyl}
It turns out that it is sufficient to prove Theorem \ref{midl} with $V$ sampled from ${\bf P}_{\beta_c}(\cdot\,|\, L\in \mathfrak{X})$ instead of  $\mathbf{P}_{\beta_c}(\cdot \mid V_{\xi_L+1}=0, \ K_{\xi_L}=L)$. Understanding this last point requires to define, for every $L\in \N$, the random variable  $y_L$ which, for $V$ sampled from $\mathbf{P}_{\beta_c}(\cdot \mid V_{\xi_L+1}=0, \ K_{\xi_L}=L)$, records the length along which V  sticks to the origin before $\xi_L+1$, i.e., 
$$y_L:=\max\{k\geq 0\colon\, V_{\xi_L-k+1}=\dots=V_{\xi_L+1}=0\}.$$
 We will prove in Section \ref{prtyl} that $(y_L)_{L\in \N}$ forms 
a tight family of random variables and therefore  it is sufficient to consider the events $y_L=k$ for finitely many $k$. Moreover,  for $k\in \N$, the event $y_L=k$ yields that  $L-k+1\in \mathfrak{X}$ and $V_{\xi_{L-k+1}}=0$, so that  the trajectory $(V_i)_{i=0}^{\xi_{L-k+1}}$ has for law ${\bf P}_{\beta_c}(\cdot \, |\, L-k+1\in \mathfrak{X}, V_{\xi_{L-k+1}}=0)$.  We conclude by noticing that 
the symmetric Laplace distribution of the increments of $V$  yields that 
$(V_i)_{i=0}^{\xi_{L-k}}$ and $(M_i)_{i=0}^{\xi_{L-k}}$ have the same law when $V$ is sampled from ${\bf P}_{\beta_c}(\cdot \, |\, L-k+1\in \mathfrak{X}, V_{\xi_{L-k+1}}=0)$ as when $V$ is sampled from ${\bf P}_{\beta_c}(\cdot \, |\, L-k+1\in \mathfrak{X})$.
\end{remark}
\smallskip

\begin{remark}\label{tzl}
Our strategy to prove Theorem \ref{midl} is reminiscent of the strategy used in  
\cite{DGZ05} to derive the scaling limit of a particular polymer model, i.e., the critical wetting model. To be more specific, the authors prove that, at criticality,  a $1+1$-dimensional $L$-step random walk pinned at an horizontal hard-wall,
constrained to start and end at the wall and rescaled in time by $L$ and in space by $\sqrt{L}$  converges in distribution towards 
the modulus of a Brownian bridge. To achieve this result they display a smart reconstruction of the path under the polymer measure using the following $4$ features of their model
\begin{itemize}
\item[i)] the Hamiltonian only depends on $|\cA_L|$ where $\cA_L:=\{0,y_1,\dots,y_{n_L}=L\}$ is the set of pinned sites of a given $L$-step random walk path. 
Moreover, under the critical polymer measure, $\frac{1}{L} \cA_L$ converges in law  towards 
$\cA_\infty:=\{s\in [0,1]\colon\, \beta_s=0\}$ with $(\beta_s)_{s\in [0,1]}$ a Brownian bridge,  
\item[ii)] under the a priori random walk law $P$ and once conditioned on  $\cA_L:=\{0,y_1,\dots,y_{n_L}=L\}$ the excursions 
of the random walk are independent and their respective length are prescribed by the inter-arrivals of  $\cA_L$,
\item [iii)] a random walk excursion of length $N$ rescaled in time by $N$ and in space by $\sqrt{N}$ converges in distribution towards a standard Brownian excursion,
\item [iv)] under the critical polymer measure, the inter-arrivals of $\cA_L$, i.e., $(y_{i+1}-y_i)_{i\geq 0}$ are almost surely finite and heavy-tailed random variables.
\end{itemize}
Their technique consists of using i) in combination with Skohorod's representation theorem to first  sample $\cA_L$ under the polymer measure 
and $\cA_\infty$ such that $\frac{1}{L}  \cA_L$ converges almost surely towards $\cA_\infty$. Then, with iv) they claim that it suffices to consider finitely many inter-arrivals (and therefore excursions) of
$\cA_L$ to reconstruct a fraction of the path arbitrary close to $1$. Finally, they use (ii-iii)  in combination with  Skohorod's representation theorem to sample random walk excursions on the longest inter-arrivals of $\cA_L$ 
 and recover the modulus of a Brownian bridge.
  
Let us briefly describe the 4 features of IPDSAW (a--d) which can be substituted to (i--iv) above in order to adapt the reconstruction technic in our context:
\begin{itemize}
\item[a)] Theorem \ref{transf21}  implies that the Hamiltonian of IPDSAW is somehow absorbed in the auxiliary random walk law ${\bf P}_\beta(\cdot\, |\, L\in \mathfrak{X})$. The   
sequence of cumulated geometric areas $\mathfrak{X}_L:=\mathfrak{X}\cap[0,L]=\{0,x_1,\dots,x_{v_L}\}$ (recall \eqref{defX}) plays the role of $\cA_L$ and $\frac{1}{L}\mathfrak{X}_L$ converges in distribution 
towards $\mathfrak{X}_\infty:=\{s\in [0,1]\colon\, B_{a_s}=0\}$ where $B$ is defined as in the statement of Theorem \ref{midl}, 
\item[b)] under the random walk law ${\bf P}_\beta$ and once conditioned on $\mathfrak{X}_L$,  Proposition \ref{pro:renstruccritical} below guarantees that the excursions (in modulus) are independent and their respective geometric area are prescribed by the inter-arrivals of $\mathfrak{X}_L$, 
\item[c)] with Theorem \ref{convexc}, we claim that, once conditioned on sweeping a geometric area $N$,  a random walk excursion and its associated center-of-mass walk rescaled in time by $N^{2/3}$
and in space by $N^{1/3}$ converge towards a Brownian excursion normalized by its area and an independent Brownian motion,
\item[d)] Under ${\bf P}_\beta$, the inter-arrivals of $\mathfrak{X}_L$, i.e., $(z_{i+1}-z_i)_{i\geq 0}$ are almost surely finite and heavy-tailed random variables.
\end{itemize}
 Although the statements (a--d) constitute the skeleton of our path reconstruction (borrowed from \cite{DGZ05}), adapting this method to the context of critical-IPDSAW raises two major additional challenges that are addressed in the present paper.  The first difficulty comes from the fact that we consider simultaneously the random walk $V_L$ and its associated center-of-mass walk $M_L$. 
If the random walk comes back very close to the origin at the end of its excursions this is not the case of the center-of-mass walk. Therefore, one needs informations on the position of $M_L$ at the beginning of every long excursion of $V_L$. This requires to control the fluctuations of $M_L$ on the "short" excursions of $V_L$ and it will be the object of  Propositions \ref{proposui}
and \ref{propfin} in Section \ref{prmidlaw}.  

 The second difficulty comes from the fact that the system size $L$ provides a conditioning on the  
  geometric area rather than on the length of the paths taken into account. To be more specific, in the wetting model the paths are constrained to 
 complete their last excursion with a total length that equals $L$ whereas in the present model there are no constraint on the total length but the paths must complete their last excursion
 with a total geometric area that equals $L$. 
 This is the reason why we perform below a time change on $\widehat V_L$ and  $\widehat M_L$ so that the resulting random processes 
are defined on $[0,1]$ and that for $s\in [0,1]$ they are observed at the time at which the geometric area swept by $V$ equals $s L$. Operating this time change allows us to state 
Theorem \ref{midlaux} whose proof is simpler although it is equivalent to Theorem \ref{midl}.
 
\end{remark}

%
We recall \eqref{taul} and we define 
the cadlag processes $\widetilde V_L$ and  $\widetilde M_L$ as
\begin{align}
\widetilde  V_L \colon  [0,1]&\mapsto \R \hspace{2.5 cm} \text{and} \hspace{1.5cm}  \widetilde M_L \colon  [0,1]\mapsto \R\\
\nonumber s&\to L^{-\frac{1}{3}} \ V_{\xi_{sL}}\hspace{5.1 cm} s\to L^{-\frac{1}{3}} \ M_{\xi_{sL}}
\end{align}
We will need to perform the same type of time change for the standard Brownian motion $B$, i.e., for $s\geq 0$ we denote by $A_s$ 
the geometric area
swept by $B$ up to time $s$, that is, 
\be{defarea}
A_s=\int_0^s |B_u| du.
\ee 
The continuity and strict monotonicity of $A$ allows us to define 
$a$ as the inverse of  $A$, i.e., $A_{a_s}=s$ for $s\in [0,\infty)$. 
For $B$ and $D$ are two independent Brownian motions, we define the continuous processes $\widehat B$, $\widetilde B$, $\widehat D$ and $\widetilde D$ as the Brownian counterparts of $\widetilde V_L$, $\widehat V_L$, $\widetilde M_L$ and $\widehat M_L$, respectively,  i.e., 
\begin{align}
\widehat  B \colon  [0,\infty)&\mapsto \R \hspace{2.5 cm} \text{and} \hspace{1.5cm}    \widetilde B \colon  [0,1]\mapsto \R\\
\nonumber s&\to B_{s\wedge a_1} \hspace{5.1 cm} s \to B_{a_s}
\end{align}
\begin{align}
\widehat  D \colon  [0,\infty)&\mapsto \R \hspace{2.5 cm} \text{and} \hspace{1.5cm}    \widetilde D \colon  [0,1]\mapsto \R\\
\nonumber s&\to D_{s\wedge a_1} \hspace{5.1 cm} s \to D_{a_s}
\end{align}
\smallskip

 We denote by $\widetilde R_{L,\beta}$
the law of $(|\widetilde  V_L|, \widetilde  M_L)$ when  $V$ is sampled from
$\mathbf{P}_{\beta_c}\bigl(\cdot \mid L\in \mathfrak{X}\bigr)$ and we state Theorem \ref{midlaux} 
which is the counterpart of Theorem \ref{midl} with $\widetilde  V_L, \widetilde  M_L , \widetilde  B, \widetilde D$
instead of  $\widehat  V_L, \widehat  M_L ,  \widehat B, \widehat D$. In Section \ref{step2}, we will display an explicit link between those quantities with 
equations \eqref{rattach} and \eqref{rattach2} and we will prove that  Theorem \ref{midl} is a consequence of Theorem \ref{midlaux}.

\begin{atheo} \label{midlaux}
For $\beta>0$, 
\begin{equation}
\widetilde R_{L,\beta}\xrightarrow[L\to \infty]{d} \big (|B_{a_s}|, D_{a_s}\big)_{s\in [0,1]}
\end{equation} 
where
$B$ and $D$ are independent Brownian motions of variance $\sigma_\beta^2$
and $a$ is the inverse function of the geometric area swept by $B$ and where $B$ is considered under the conditioning $\{B_{a_1}=0\}$.
\end{atheo}
The proof of Theorem \ref{midlaux} is the object of Section \ref{prmidlaw} below.

%

\subsection{Outline of the proof of Theorem \ref{midlaux}}\label{prmidlaw}

Our proof of  Theorem \ref{midlaux} relies on the renewal structure introduced in (\ref{premm}--\ref{defX}) above and it may be divided into three steps.
\begin{enumerate}
\item handling small excursions of random walk with Proposition \ref{proposui}, of Brownian motions with Proposition \ref{propfin},
\item handling large excursions with Theorems \ref{convhaus}, \ref{convexc} and \ref{Bessel},
\item reconstructing the limiting process with Proposition \ref{propded}.
 \end{enumerate}

 Here, the geometric area 
of each excursion will be of particular importance. For $k\in \N$, we will indeed truncate the rescaled profile $|\widetilde V_L|$ and the rescaled center-of-mass walk  $\widetilde M_L$ 
(respectively the time-changed Brownian motions $|\widetilde B|$ and $\widetilde D$)  outside the excursions of $V$  (resp. $B$) sweeping a geometric area
larger than $L/k$ (resp. $1/k$) to obtain  $|\widetilde V_{L,k}|$ and $\widetilde M_{L,k}$ (resp. $|\widetilde B^k|$ and $\widetilde D^k$). Then,  the proof of Theorem \ref{midlaux} will be organized as follows. 
With Proposition \ref{proposui} (proven in  Section \ref{step4}),  we state that provided $k$ and $L$ are  large enough,  
 $||\widetilde M_L-\widetilde M_{L,k}||_\infty+||\widetilde V_L-\widetilde V_{L,k}||_\infty$ is arbitrary small in probability.  With Proposition \ref{propfin} (proven in Section \ref{prpropfin}) we prove that,
 provided $k$ is large enough  
 $||\widetilde B-\widetilde B^k||_\infty+||\widetilde D-\widetilde D^k ||_\infty$ also is arbitrarily small in probability.
 Finally, with Proposition \ref{propded} (proven in Section \ref{step3}), we provide a simplified version of Theorem \ref{midlaux} by substituting the truncated processes $(|\widetilde V_{L,k}|,\widetilde M_{L,k})$ to $(|\widetilde V_L|$, $\widetilde M_{L})$ and $(\widetilde B^k, \widetilde D^k)$ to $(|\widetilde B|$, $\widetilde D)$, respectively.    
 Those three propositions imply Theorem \ref{midlaux}.

 \begin{remark}\label{frombtoub}
For  sake of conciseness,  we will display the proof of Theorem \ref{midlaux} under the law ${\bf P}_{\beta,\mu_\beta}$ instead of ${\bf P}_{\beta}$. The only difference between those two laws is 
that, under ${\bf P}_{\beta,\mu_\beta}$ the law of $V_0$  is $\mu_\beta$ (defined in \eqref{defubeta}) which is symmetric on $\Z$ with an exponential tail. Proving Theorem \ref{midlaux} under ${\bf P}_\beta$ is not more difficult but (and this is explained in Proposition \ref{pro:renstruccritical} below) it would force us to consider separately the very first excursion of each path from all the other excursions. This distinction 
is not necessary anymore under ${\bf P}_{\beta,\mu_\beta}$ and this lightens the proofs a little bit.

 \end{remark}
\smallskip

\noindent {\it Truncation of the profile and center-of-mass walk.} We recall \eqref{defM} and we observe that the center-of-mass walk can be written as 
\be{defM2}
M_{i}=-\frac{V_0}{2}+\sum_{j=1}^{i} (-1)^{j-1} \frac{V_j-V_{j-1}}{2}=\sum_{j=0}^{i-1} (-1)^{j-1} V_j+ (-1)^{i-1} \frac{V_i}{2} ,\quad i\in \N.
\ee
We recall \eqref{premm} and for every $r\in\N$, we let $M^{\text{exc}}(r)$ be  the contribution of the $r$-th excursion to the center-of-mass walk, i.e., 
\be{defexcmid}
M^{\text{exc}}(r)=\sum_{i=\tau_{r-1}}^{\tau_r-1} \, (-1)^{i-1} \, V_i.
\ee 
For  $x\in \N$, we truncate $V$ outside the excursions of geometric area larger than $x$ to obtain $(V_x^+(i))_{i\in \N\cup\{0\}}$. Similarly, with the help of \eqref{defexcmid} we define the discrete process $(M^+_{x}(i))_{i\in \N\cup\{0\}}$
which remains constant outside 
the excursions of geometric area larger than $x$ and follows the center-of-mass walk elsewhere, i.e., for every $t\in \N$ and $i\in \{\tau_{t-1},\dots,\tau_t-1\}$
\begin{align}\label{tronc3}
M_{x}^+(i)&:=\sum_{r=1}^{t-1}\,  M^{\text{exc}}(r) \, \ind_{\{X_r\geq x\}}+ \Big[\sum_{j=\tau_{t-1}}^{i-1} (-1)^{j-1} V_j + 
(-1)^{i-1} \frac{V_i}{2}\Big] \,  \ind_{\{X_t\geq x\}},\\
\nonumber V_x^+(i)&:=V_i \,  \ind_{\{X_t\geq x\}}. 
\end{align}
For $k\in \N$, we let  $\widetilde V_{L,k}$ and $\widetilde M_{L,k}$ be the cadlag processes obtained from  $V_{L/k}$ and $M_{L/k}$ as we obtained $\widetilde M_{L}$ from  $M_{L}$,  i.e., 
\begin{align}\label{tronc}
\widetilde M_{L,k}(s)&:=\frac{1}{L^{1/3}} M^+_{L/k}(\xi_{sL}), \quad s\in [0,1],\\
\nonumber \widetilde V_{L,k}(s)&:=\frac{1}{L^{1/3}} V^+_{L/k}(\xi_{sL}), \quad s\in [0,1].
\end{align}

\noindent {\it Truncation of Brownian motion.} As in the discrete case, we truncate $\widetilde B$ and $\widetilde D$ outside the 
excursions of $B$ sweeping a geometric area larger than $1/k$ to obtain $\widetilde B^{\,k}$ and $\widetilde D^{\,k}$ 
, i.e.,
\begin{align}\label{defDk}
\widetilde D^k(s)&=\int_0^{a_s} \, \ind_{\Gamma_k}(u) \, d D_u,\\
\nonumber \widetilde B^k(s)&=\widetilde B_s  \, \ind_{\Gamma_k}(s),
\end{align}
where $\Gamma_k :=\ens{u >0: A_{d_u} -A_{g_u} \ge \unsur{k}}$ with
$d_u=d_u(B):=\inf\ens{t>u : B_t=0}$ , $g_u = \sup\ens{t< u : B_t =0}$
so that  $d_u - g_u$ (resp. $A_{d_u} -A_{g_u}$)  is
the length (resp. the geometric area) of the excursion straddling $u$.
\smallskip

%
With these truncated processes in hand we can state Propositions \ref{proposui} and \ref{propfin} in order to estimate the time-changed profile $|\widetilde V_L|$ and center-of-mass walk $\widetilde M_L$
and the time-changed  Brownian motions $\tilde B$ and  $\tilde D$ with their truncated versions. 
\begin{proposition}\label{proposui}
 For every $\gep>0$,
\begin{align}\label{trza}
\lim_{k\to \infty} &\limsup_{L\to \infty} \probmubeta{ \sup_{s\in [0,1]} \big|\,  \widetilde V_L(s)-\widetilde V_{L,k}(s)\, \big|\geq \gep\   \Big | \  L\in \mathfrak{X}\ }=0,\\
\lim_{k\to \infty} &\limsup_{L\to \infty} \probmubeta{ \sup_{s\in [0,1]} \big|\widetilde M_L(s)-\widetilde M_{L,k}(s)\big|\geq \gep\   \Big | \  L\in \mathfrak{X}\ }=0,
\end{align}
\end{proposition}

\begin{proposition}\label{propfin}
For every $\gep>0$,
\begin{align}\label{trza2}
\lim_{k\to \infty} \prob{ \sup_{s\in [0,1]} \big| \widetilde B(s)-\widetilde B^k(s)\big|\geq \gep }&=0,\\
\lim_{k\to \infty} \prob{ \sup_{s\in [0,1]} \big|\widetilde D(s)-\widetilde D^k(s)\big|\geq \gep }&=0.
\end{align}
where
$B$ and $D$ are defined as in Theorem \ref{midlaux}.
\end{proposition}

Theorem  \ref{convhaus}, proven for instance in \cite[Proposition A.8]{CSZ}, claims that 
once rescaled by $L$ and when sampled from  $\mathbf{P}_{\beta}(\cdot\ |\ L\in \mathfrak{X})$ the set $\mathfrak{X}\cap [0,L]$ converges in law (in the space of closed subsets of $[0,1]$ endowed with the Hausdorff
distance) towards $\widetilde C_{1/3}:=C_{1/3}\cap[0,1]$ conditioned on $1\in C_{1/3}$ where $C_{1/3}$ is the  $1/3$-stable regenerative set.

\begin{atheo} \label{convhaus}
For $L\in \N$,  we  let $\mathfrak{X}$ be sampled from $\mathbf{P}_{\beta,\mu_\beta}(\cdot\ |\ L\in \mathfrak{X})$, then,
\be{limXX}
\lim_{L\to \infty} \frac{\mathfrak{X}\cap [0,L]}{L}=_{\text{Law}} \widetilde C_{1/3}.
\ee
\end{atheo}

Theorems \ref{convexc} and \ref{Bessel} below are proven in a
companion paper \cite{CarPet17b}. 
 With Theorem \ref{convexc}, we state that, a random walk excursion
 (together with its center-of-mass walk) conditioned to have a prescribed area $L$, properly rescaled and  subject to an adhoc time change converge towards a Brownian excursion normalized by its area  (together with an independent Brownian motion)
also subject to a similar time change.

\begin{atheo} \label{convexc}
We consider $V$ sampled from $\probmubeta{\cdot\ |\ X_1=L}$, then,
\begin{align}\label{vv}
\lim_{L\to \infty} \Big (\valabs{\widetilde V_L(s)}, \widetilde M_L(s)\Big)_{s\in [0,1]}=_{\text{Law}}  \big(\cE_{a_s}, B_{a_s} \big)_{s\in [0,1]}
\end{align}
where $\cE$ is a Brownian excursion normalized by its area and $a$ is the inverse function of this 
area and where $B$ is a standard brownian motion independent of $\cE$.
\end{atheo}

Let  $Y$ be distributed as $(\valabs{B_{a_s}})_{0\le s \le 1}$
conditioned by $B_{a_1}=0$. Then $Y$ is distributed as $\etp{\big(
  \frac{3}{2} \rho_{t}\big)^{2/3},t\in[0,1]}$ where 
 $(\rho_t)_{t\in [0,1]}$ is a Bessel bridge of dimension $\delta=4/3$.  
 We let $Z(Y)$ be the set of zeros of $Y$, i.e.,
$Z(Y)=\{s\in [0,1]\colon\, Y(s)=0\}$. We let also $\pi^Y$ be the law of an excursion of $Y$ renormalized by its extension  and let $\gamma_{\cE}$ be the law of $\cE_{a}$ defined in \eqref{vv} above.

\begin{atheo} \label{Bessel} 
The following equalities in distribution hold true,
\be{equadis}
Z(Y)=_{\text{Law}} \widetilde C_{1/3}\quad \quad \text{and} \quad \quad    \pi^Y= \gamma_{\cE}.
\ee
\end{atheo}

Propositions \ref{proposui} and \ref{propfin} are of key importance, because they reduce significantly the level of complexity of Theorem \ref{midlaux}. It becomes indeed sufficient 
to prove Proposition \ref{propded} below which is a simplified version of Theorem \ref{midlaux} to the extend that the profile and center of mass walk are replaced by their
truncated version.  We let $\widetilde R_{L,\beta}^{\, k}$ be the joint law of 
$(|\widetilde V_{L,k}|, \widetilde M_{L,k})$ under $\probmubeta{\cdot \mid L\in \mathfrak{X}}$ and $B$ and $D$ be independent Brownian motions as defined in the statement of Theorem 
\ref{midlaux}.

\begin{proposition}\label{propded}
 For every $k\in \N$,
\be{convtronca}
\widetilde R_{L,\beta}^{\,k}\xrightarrow[L\to \infty]{d} (\widetilde B^k, \widetilde D^k).
\ee
\end{proposition}

In section \ref{step3} we prove Proposition \ref{propded} subject to Theorems \ref{convhaus}, \ref{convexc} and \ref{Bessel}.


\section{Preparations}\label{prepa}
In Section \ref{mren} below, we give a complete description of the renewal structure introduced in (\ref{premm}--\ref{defX}) 
and consisting of excursions of the path away from the origin. We recall some facts from \cite{CP15} concerning the 
geometric area and extension of those excursions and we go further by  giving a method to reconstruct a trajectory 
$V$ of law ${\bf P}_{\beta,\mu_\beta}$ with the help of independent excursions. In Section \ref{secpassunif}, we justify the use of Skorokhod Lemma for those
cadlag processes considered in the present paper. In Section \ref{proofRWR}, we prove Theorem \ref{transf21}.

\subsection{More about the renewal process}\label{mren}

We recall (\ref{premm}--\ref{defX}), we let $\mu_\beta$ be a probability law on $\Z$ defined as 
\be{defubeta}
 \mu_\beta(k) = \frac{1-e^{-\beta/2}}{2}
e^{-\frac{\beta}{2} \valabs{k}} \un{k\neq 0} +  (1-e^{-\beta/2})
\un{k=0},
\ee
and we let $\bP_{\beta,x}$ be the law of the random walk starting
  from $V_0=x\in \Z$ and $\bP_{\beta,\mu_\beta}$ be the law of the
  random walk when  $V_0$ has
distribution $\mu_\beta$. 
In \cite[Lemma 4.6]{CP15} the tail distribution of $X_1$ is displayed as well as a renewal theorem for the set $\mathfrak{X}$. To be more specific, 
for $x\in \{0,\mu_\beta\}$ there exists a $c_{x,\beta}>0$ and $d_{x,\beta}>0$  such that 
\be{tailest}
{\bf P}_{\beta,x}(X_1=n)=\frac{c_{x,\beta}}{n^{4/3}} (1+o(1)) \quad \text{and}\quad {\bf P}_{\beta,x}(n\in \mathfrak{X})=\frac{d_{x,\beta}}{n^{2/3}} (1+o(1)).
\ee
Note that Lemma 4.6 is stated in \cite{CP15} under ${\bf P}_{\beta,\mu_\beta}$ but holds true under 
${\bf P}_{\beta}$ as well. 

In the present paper we need to go further in the analysis of the renewal.  
With the help of  $(\tau_k)_{k\geq 0}$,  we divide any random walk trajectory $V$ into a sequence of excursions $(\mathfrak{E}_k)_{k\geq 0}$ and we also denote by $(|\mathfrak{E}|_k)_{k\geq 0}$ the same excursions in modulus, i.e.,   for $k\in \N$
\be{defexc}
\mathfrak{E}_k=(i,V_i)_{i\in \{\tau_{k-1},\dots,\tau_k-1\}} \quad \text{and} \quad |\mathfrak{E}|_k=(i,|V_i|)_{i\in \{\tau_{k-1},\dots,\tau_k-1\}}.
\ee
We will consider this sequence under ${\bf P}_\beta$ and ${\bf P}_{\beta,\mu_\beta}$. It is not true that the excursions themselves are
independent because the sign of any excursion depends on the sign of the preceding excursion. However,
when considered in modulus, those excursions are independent.  
\begin{proposition}
\label{pro:renstruccritical}
Under ${\bf P}_{\beta}$ the random processes $(|\mathfrak{E}|_k)_{k\ge 1}$ are
independent and the sequence $(|\mathfrak{E}|_k)_{k\ge 2}$ is IID. 
The law of $|\mathfrak{E}|_1$ is that  of the first excursion (in modulus) of a 
random walk of law $\bP_{\beta,0}$ and for $k\geq 2$ the law of $|\mathfrak{E}|_k$ is that of the first excursion (in modulus) of a 
random walk of law $\bP_{\beta,\mu_\beta}$.

Under ${\bf P}_{\beta,\mu_\beta}$ the random processes $(|\mathfrak{E}|_k)_{k\ge 1}$ are IID. 
The law of $|\mathfrak{E}|_1$ is that  of the first excursion (in modulus) of a 
random walk of law $\bP_{\beta,\mu_\beta}$.
\end{proposition}
\begin{proof}
We note that $V$ is a Markov chain, that $\tau_k$ is a stopping time (for $k\in \N$) and that for every $x\in \Z$ the law of 
$|\mathfrak{E}|_1$ under ${\bf P}_{\beta,x}$ equals the law of  $|\mathfrak{E}|_1$ under ${\bf P}_{\beta,|x|}$. Therefore, the proof of Proposition
\ref{pro:renstruccritical} will be complete once we show (by induction) that for every $k\in \N$, the random variable $|V_{\tau_k}|$ is independent of 
the $\sigma$-algebra $\sigma(\mathfrak{E}_1,\dots,\mathfrak{E}_k, \tau_k)$ and has the same law as $|T|$ with $T$ a random variable of law $\mu_\beta$. 

We pick $t_1<t_2<\dots<t_k\in \N$ and $(v_0,\dots,v_{t_k-1})\in \Z^{t_k}$ that are compatible with the event 
$$M_k:=\{\mathfrak{E}_1=(v_0,\dots,v_{t_1-1}), \dots,\mathfrak{E}_k=(v_{t_{k-1}},\dots,v_{t_k-1})\}.$$
We set $x\in \N_{0}$ and we compute $C:={\bf P}_{\beta,\mu_\beta}(M_k\cap \{|V_{\tau_k}|=x\})$ as 
\begin{align}\label{indepb}
\nonumber C&={\bf P}_{\beta,\mu_\beta}( V_0=v_0,\dots,V_{t_k-1}=v_{t_k-1}, V_{t_k}=-\text{sign} (v_{t_k-1})\,  x)\\
\nonumber &={\bf P}_{\beta,\mu_\beta}( V_0=v_0,\dots,V_{t_k-1}=v_{t_k-1}) \, {\bf P}_{\beta,\mu_\beta}(U_1=x+|v_{t_k-1}|)\\
&={\bf P}_{\beta,\mu_\beta}(M_k) \, \frac{{\bf P}_{\beta,\mu_\beta}(U_1=x+|v_{t_k-1}|)}{{\bf P}_{\beta,\mu_\beta}(U_1\geq |v_{t_k-1}|)}.
\end{align}
The ratio on the r.h.s. in \eqref{indepb} is equal to $(1-e^{-\frac{\beta}{2}}) \, e^{-\frac{\beta}{2} x}$ which is exactly $\mathbb{P}(|T|=x)$ when $T$ has law $\mu_\beta$. 
This completes the proof.
\end{proof}

\noindent {\it Random walk reconstruction.} 
Proposition \ref{pro:renstruccritical} will  allow us to reconstruct a random walk of law ${\bf P}_{\beta,\mu_\beta}$ with a sequence of independent excursions (in modulus).
With Definitions \ref{deftiv}  below we give the details of this construction. 
\begin{definition}\label{deftiv}
 Let $(\gep_i)_{i\in \N}$ be an i.i.d. sequence of symmetric Bernoulli trials taking values $1$ and $-1$.  
Let also 
 $(W_0,W_1,\dots, W_{\tau_1})$ be the first excursion of a trajectory $W$ with law 
${\bf P}_{\beta,\mu_\beta}$. Independently from $(\gep_i)_{i\in \N}$, let $\big\{(V_i^{j})_{i\in \{0,\dots, \tau^j_1\}}, j\geq 1\big\}$ be a sequence of independent  copies of  
$(W_0,W_1,\dots, W_{\tau_1-1})$ and set  $\tau_0=0$ and   $\tau_j=\tau_{j-1}+\tau_1^j$ for $j\geq 1$.  Finally, define $V$ as follows:
\begin{itemize}
\item $V_i=V_i^1,\quad i\in \{0,\dots,\tau_1-1\}$,
\item for $j\geq 2$ \, if\,  $V_0^j\neq 0$\,  then\,  $ V_i=-\text{sign}(V_{\tau_{j-1}-1}) \, \big| V^j_{i-\tau_{j-1}}\big|$ 
\, for every\,  $i\in \{\tau_{j-1},\dots,\tau_{j}-1\}$,
\item for $j\geq 2$\,  if\,  $V_{0}^j= 0$\,  then\,  $V_i=\gep_j \,  \big| V^j_{i-\tau_{j-1}}\big|$\,  for every\,  $i\in \{\tau_{j-1},\dots,\tau_{j}-1\}$.
\end{itemize}
The resulting stochastic process $V$ is a random walk of law ${\bf P}_{\beta,\mu_\beta}$.
\end{definition}

\begin{remark}\label{remindep}
The construction in Definition \ref{deftiv} will be used in Section \ref{step4} and we note that, by construction, the sequence of signs $(\gep_i)_{i\in \N}$ is independent 
of the modulus of the trajectory $(|V_i|)_{i\in \N_{0}}$ and also independent of $(\tau_i)_{i\in \N_{0}}$.
\end{remark}

With Definitions \ref{deftiv2} below we display an alternative construction in order to generate for a random walk of law 
${\bf P}_{\beta,\mu_\beta}(\cdot\, | \, L\in \xi)$. This construction will be used in Section \ref{step3}.  

\begin{definition}\label{deftiv2}
Let  $(\gep_i)_{i\in \N}$ be an  i.i.d. sequence of symmetric Bernoulli trials taking values $1$ and $-1$. 
Recall \eqref{defX} (and the definition of $v_L$ below) and  independently from $(\gep_i)_{i\in \N}$ sample a random set 
$$\mathfrak{X}=\{0,X_1,\dots,X_1+X_2+\dots+X_{v_L-1},L\}$$ under ${\bf P}_{\beta}(\cdot\, | \, L\in \xi)$. Independently from 
$(\gep_i)_{i\in \N}$ and $\mathfrak{X}$, 
sample 
 for every $j\in \{1,\dots,v_L\}$ the excursion  $(V_i^{j})_{i\in \{0,\dots, \tau^j\}}$
with law ${\bf P}_{\beta,\mu_\beta} (\cdot \,|\,X= X_j)$. Set  $\tau_0=0$ and   $\tau_j=\tau_{j-1}+\tau^j$ for $j\geq 1$. 
Finally, define $V$ as follows:
\begin{itemize}
\item $V_i=V_i^1,\quad i\in \{0,\dots,\tau_1-1\}$,
\item for $j\geq 2$ \, if\,  $V_0^j\neq 0$\,  then\,  $ V_i=-\text{sign}(V_{\tau_{j-1}-1}) \, \big| V^j_{i-\tau_{j-1}}\big|$ 
\, for every\,  $i\in \{\tau_{j-1},\dots,\tau_{j}-1\}$,
\item for $j\geq 2$\,  if\,  $V_{0}^j= 0$\,  then\,  $V_i=\gep_j \,  \big| V^j_{i-\tau_{j-1}}\big|$\,  for every\,  $i\in \{\tau_{j-1},\dots,\tau_{j}-1\}$.
\end{itemize}
The resulting stochastic process $V$ is a random walk of law ${\bf P}_{\beta,\mu_\beta}(\cdot\, |\, L\in \mathfrak{X})$.
\end{definition}

\subsection{Skorohod's representation Theorem  for cadlag random functions}\label{secpassunif}

Along the paper, we will often need to consider some piecewise constant cadlag processes defined either on $[0,\infty)$ or on $[0,1]$. We will also need to consider  the interpolated versions of such processes. To that aim we define two sets of functions, i.e.,  for $I\in \{[0,1],[0,\infty)\}$, we let $(\mathcal{C}_I,d)$ be the set of  continuous functions on $I$, endowed with 
$$d(f,g)=\sum_{k=1}^\infty \frac{1}{2^k} \frac{||f-g||_{I\cap [0,k],\infty}}{1+||f-g||_{I\cap [0,k],\infty}},$$
and similarly we let  $(D_I, d)$ be the set of cadlag functions defined on $I$ also endowed with the same distance.
We recall that $(\mathcal{C}_I, d)$ is a Polish space whereas $(D_I, d)$ is not. Therefore, one can a priori not apply directly the Skorohod's representation Theorem in $(D_I, d)$
(see \cite[Theorem 6.7]{Bill08}). Let us explain briefly below how this difficulty can be handled.


 For $n\in \N$, we will only consider functions $F$ that are piecewise constant and cadlag,  defined on $I\in \{[0,\infty),[0,1]\}$, and such that all jumps of $F$ 
 occur at times belonging to $I\cap \frac{\N}{n}$. For such functions, we denote by $F^{\text{int}}$ their interpolated version, that is,  
\begin{align}
\rm{F}^{\, \text{int}}: \ I &\mapsto \R\\
 \nonumber                        s &\mapsto (1-\{sn\})\, \rm{ F}\big(\tfrac{\lfloor sn\rfloor}{n}\big)+\{sn\} \, \rm{F}\big( \tfrac{\lfloor sn\rfloor +1}{n}\big).
\end{align} 
 
We recall (\ref{lawP}--\ref{taul}). All the cadlag processes considered in the rest of the paper are built with the increments $(U_i)_{i=1}^{\xi_L}$ (resp. $(U_i)_{i=1}^{\tau_1}$) of a 
random walk $V$ of law  ${\bf P}_{\beta,x}(\cdot |\,  L\in \mathfrak{X})$  (respectively ${\bf P}_{\beta,x}(\cdot |\,  X_1=L)$) with $x\in \{0,\mu_\beta\}$.  Since 
$\xi_L\leq L$ and $\tau_1\leq L$, it is useful to define for $\alpha>0$
\begin{equation}\label{fstu}
\mathfrak{A}_{L,\alpha}:=\{ \exists i\in \{1,\dots L\} \colon |U_i| \geq \alpha \log L\}.
\ee
With \eqref{lawP} and \eqref{tailest}
 we easily prove that there exists an $\alpha>0$ such that  for $x\in \{0,\mu_\beta\}$
\be{bincre}\lim_{L\to \infty} \mathbf{P}_{\beta,x}\big[ \mathfrak{A}_{L,\alpha}\  |\ L\in \mathfrak{X}\big]=0\quad \text{and}\quad \lim_{L\to \infty} \mathbf{P}_{\beta,x}\big[ \mathfrak{A}_{L,\alpha}\  |\ X_1=L\big]=0.
\ee
As a consequence, if we denote by $F_L$ (respectively $F_L^{\text{int}}$) a generic cadlag random process (and its interpolated version) build with the increments of $V$ 
and rescaled vertically  by $L^\alpha$ for some $\alpha>0$, we can deduce  from \eqref{bincre} that  for  $x\in \{0,\mu_\beta\}$ 
and for $\gep>0$
\be{equaint}
\lim_{L\to \infty}\mathbf{P}_{\beta,x}\big[ d(F_L,F_L^{\text{int}})>\gep\, |\, L\in \mathfrak{X}\big]=0 \quad 
\text{and} \quad \lim_{L\to \infty}\mathbf{P}_{\beta,x}\big[ d(F_L,F_L^{\text{int}})>\gep\, |\, X_1=L\big]=0.
\ee
Thus, for $F_\infty\in \cC_I$, the convergence in law of $(F_L)_{L\in \N}$ towards $F_\infty$ in $(D_I,d)$ is equivalent to the convergence in law of
$F_L^{\text{int}}$ towards $F_\infty$ in $(\cC_I,d)$. Moreover, the fact that $F_L$ only jumps at times belonging to  $I\cap \frac{\N}{L}$ allows us 
to reconstruct $F_L$ from $F_L^{\text{int}}$ in an easy way.  Therefore, Skohorod's representation Theorem can be applied in the present paper 
for convergence in $(D_I,d)$ as well.

\subsection{Proof of Theorem \ref{transf21}}\label{proofRWR}
We recall the stretch description of IPDSAW in (\ref{defLL}--\ref{hammm}) and we observe that the partition function can be rewritten under the form
 \be{partfunstr}
 Z_{L,\beta}=\sum_{N=1}^{L}\sum_{l\in\mathcal{L}_{N,L}} \hspace{-1mm}   e^{\beta\sum_{i=1}^{N-1}(l_i\;\tilde{\wedge}\;l_{i+1})}.
 \ee
We note that 
$\forall x,y\in\mathbb{Z}$ one can write  $x\;\tilde{\wedge}\;y=\frac12\left(|x|+|y|-|x+y|\right)$ and therefore  the partition function in \eqref{partfunstr}  becomes
\begin{align}\label{ls}
\nonumber Z_{L,\beta}
&=\sum_{N=1}^{L}\sum_{\substack{l\in\mathcal{L}_{N,L}\\l_0=l_{N+1}=0}}\exp{\Bigl(\beta\sum_{n=1}^N{|l_n|}-\tfrac{\beta}{2}\sum_{n=0}^N{|l_n+l_{n+1}|}\Bigr)}\\
&=c_\beta\, e^{\beta L} \sum_{N=1}^{L}\left(\tfrac{c_\beta}{e^\beta}\right)^N\sum_{\substack{l\in\mathcal{L}_{N,L}
\\ l_0=l_{N+1}=0}}\prod_{n=0}^{N}\frac{\exp{\Bigl(-\tfrac{\beta}{2}|l_n+l_{n+1}|\Bigr)}}{c_\beta}.
\end{align}

At this stage we recall the definition of the auxiliary random walk $V$ in  (\ref{lawP}--\ref{defG}) as well as the family of one to one correspondence $(T_N)_{N=1}^L$
between path configurations and random walk trajectories (see \ref{defTN}). Since for $l\in \cL_{N,L}$  the increments $(U_i)_{i=1}^{N+1}$ of $V=(T_N)^{-1}(l)$  in \eqref{defTN} necessarily satisfy $U_i:=(-1)^{i-1}(l_{i-1}+l_i)$, one can 
rewrite  \eqref{ls} as
\begin{align}\label{tgh}
Z_{L,\beta}&
=c_\beta e^{\beta L} 
\sum_{N=1}^{L} \Gamma_\beta^{\, N}\   \mathbf{P}_{\beta}(\cV_{N,L-N}) \quad \text{with}\quad \Gamma_\beta:=\frac{c_\beta}{e^\beta}.
\end{align}
The probabilistic representation of the partition function in \eqref{tgh} is a key tool when studying IPDSAW. It allows for instance to spot quickly 
the critical point of the model which turns out to be the solution in $\beta$ of $\Gamma_\beta=1$. 
The present paper being fully dedicated to the critical regime of IPDSAW, we will henceforth 
always work at $\beta=\beta_c$ and therefore we remove the term $\Gamma_\beta$ from the r.h.s. in \eqref{tgh}. 

Another useful consequence of formula \eqref{tgh} is that it provides us with a very strong link between the polymer law  $P_{L,\beta}$ and the random walk law ${\bf P}_\beta$
conditioned on a suitable event. We recall (\ref{taul}--\ref{taul}), the fact that $\Gamma_{\beta_c}=1$ and also that the term indexed by $N$ in the sum in \eqref{ls} corresponds to the contribution to the partition function of those 
path in $\cL_{N,L-N}$. Consequently, we can derive from (\ref{ls}--\ref{tgh}) that for every $N\in \{1,\dots,L\}$, 
\begin{align}\label{egalisucc}
P_{L,\beta}(N_l=N)&={\bf P}_\beta(\xi_L=N\, |\, V_{\xi_L+1}=0, K_{\xi_L}=L), \\
\nonumber P_{L,\beta}(l\in \cdot \, |\, N_l=N)&={\bf P}_\beta(T_N(V)\in \cdot\, |\, \xi_L=N,  V_{N+1}=0, K_{N}=L).
\end{align}
Theorem \ref{transf21} is a straightforward consequence of \eqref{egalisucc}.

%
%

\section{Proof of Theorem \ref{Theo-shape} subject to Theorem \ref{midlaux}}\label{pr:Theo-shape}

\subsection{Tightness of $(y_L)_{L\in \N}$ (recall Remark \ref{tyl})} \label{prtyl}
Let $k_0\in \N$, recall that ${\bf P}_{\beta}(U_1=0)=1/c_\beta$ and observe that 
\begin{align}\label{tightyl}
\mathbf{P}_{\beta_c}\bigl( y_L> k_0,\   V_{\xi_L+1}=0, \ K_{\xi_L}=L\bigr)
\nonumber &=\sum_{k=k_0+1}^L \big(\tfrac{1}{c_\beta}\big)^k \ \mathbf{P}_{\beta_c}\bigl(V_{\xi_{L-k+1}-1}\neq 0,\ V_{\xi_{L-k+1}}=0, \ K_{\xi_{L-k+1}}=L-k+1\bigr)\\
\nonumber &=\sum_{k=k_0+1}^L \big(\tfrac{1}{c_\beta}\big)^k \ \mathbf{P}_{\beta_c}\bigl(L-k+1\in \mathfrak{X}, \ V_{\xi_{L-k+1}}=0\bigr)\\
&=(1-\tfrac{1}{c_\beta}) \sum_{k=k_0+1}^L \big(\tfrac{1}{c_\beta}\big)^k \ \mathbf{P}_{\beta_c}\bigl(L-k+1\in \mathfrak{X}, \bigr).
\end{align}
Note that \eqref{tightyl} can also be written without the event $\{y_L> k_0\}$ in the left hand side in  \eqref{tightyl} and with a sum running from 
$k=0$ to $k=L$ in the r.h.s. Therefore, we recall \eqref{tailest} and we obtain that there exists a $d>0$ such that $\mathbf{P}_{\beta_c}\bigl(   V_{\xi_L+1}=0, \ K_{\xi_L}=L\bigr)=\frac{d}{L^{2/3}} (1+o(1))$.  As a consequence 
(since also $c_\beta>1$) we deduce from  \eqref{tightyl} 
that for every $\gep>0$ we can choose $k_0$ large enough such that $\mathbf{P}_{\beta_c}( y_L> k_0\, \mid\,   V_{\xi_L+1}=0, \ K_{\xi_L}=L)\leq \gep$ for every $L\in \N$.

\subsection{From the profile and the center-of mass walk to the occupied set, i.e., proof of Theorem \ref{Theo-shape} subject to Theorem \ref{midl}}\label{step1}


\begin{proof}
By the Skorohod's representation Theorem, we can already assert that there exists 
$(\mathrm h_L,\mathrm{m}_L)_{L\in \N}$ a sequence of cadlag processes 
 and $B$ and $D$ two independent Brownian motions of variance $\sigma_\beta^2$, all defined on the same probability space $(\Omega,\cA,P)$ so that 
 \begin{itemize}
\item  $P$-a.s. it holds that for all $K>0$,
\be{skoropp}
\lim_{L\to \infty} \sup_{s\in [0,K]} \big|(\mathrm h_L(s),\mathrm{m}_L(s))-\big(|B_{s\wedge a_1}|,D_{s\wedge a_1}\big)\big|=0,
\ee 

\item  for all $L\in \N$, $(\mathrm h_L, \mathrm{m}_L)$ has for law $R_{L,\beta_c}$.
\end{itemize}
Then, we recall the definition of $\cS_{\text{crit}}(B,D)$ in \eqref{defcS} and we note that 
the random set 
\be{rs}
S_{L}(\mathrm{h}_L,\mathrm{m}_L):=\bigg\{(s,y)\in [0,i_L]\times \R\colon\,  \mathrm{m}_L(s)-\frac{|\mathrm{h}_L(s)|}{2}\leq y \leq  \mathrm{m}_L(s)+\frac{|\mathrm{h}_L(s)|}{2}\bigg\},
\ee
has for law $Q_{L,\beta}^{\frac{2}{3},\frac{1}{3}}$ with 
$i_L=\sup\{s\geq 0\colon \frac{s}{L^{1/3}}+\int_{0}^s |h_L(s)| ds \leq
1\}$.

Then, it remains to note that $S_{L}(\mathrm{m}_L,\mathrm{h}_L)$ converges $P$-a.s. towards 
$\cS_{\text{crit}}(B,D)$ for the Hausdorff distance. The almost sure convergence of 
$(\mathrm{m}_L,\mathrm{h}_L)$ towards $(B_{\cdot\wedge a_1}, D_{\cdot\wedge a_1})$
implies that $i_L$ also converges towards $a_1$ almost  surely and this is sufficient to conclude.
\end{proof}

\begin{remark}
{\rm To be completely rigorous, we must note that, as it is defined in \eqref{rs}, the law of 
$S_{L}(\mathrm{h}_L,\mathrm{m}_L)$ is not exactly $Q_{L,\beta}^{\frac23,\frac13}$. However, we recall 
\eqref{defS} and \eqref{defscal} and we observe that by enlarging $S_{L}(\mathrm{h}_L,\mathrm{m}_L)$ of 
$1/L^{1/3}$ verticaly and by shifting it of $1/2L^{2/3}$ horizontally we retrieve a set 
of Law $Q_{L,\beta}^{\frac23,\frac13}$. Finally, since the Hausdorff distance between those two sets is bounded above by $1/L^{1/3}$,
working with $S_L(h_L,m_L)$ is sufficient to conclude.}
\end{remark}

\subsection{Time change, i.e., proof of Theorem \ref{midl} subject to Theorem \ref{midlaux}}\label{step2}

\bigskip

As explained in Remark \ref{tyl},   $\widehat V_L$ is 
considered under the conditioning 
$\{L\in \mathfrak{X}\}$
whereas $\widehat B$ is considered under the conditioning $\{B_{a_1}=0\}$. 
We  observe that 
 for $s\geq0$,
\be{rattach}
\widehat{V}_L(s)=\widetilde{V}_L\Big(\tfrac{K_{\lfloor sL^{2/3}\rfloor\wedge \xi_L}}{L}\big)
\quad \text{and}\quad  \widehat B_s= \widetilde B(A_{s\wedge a_1}),
\ee
and similarly 
\be{rattach2}
\widehat{M}_L(s)=\widetilde{M}_L\Big(\tfrac{K_{\lfloor sL^{2/3}\rfloor\wedge \xi_L}}{L}\big)
\quad \text{and}\quad  \widehat D_s= \widetilde D(A_{s\wedge a_1}).
\ee
Therefore, by  \cite[Lemma p. 151]{Bill08},  the proof of Theorem \ref{midl} will be complete once we show that the following convergence in law holds true
\be{cvl}
\lim_{L\to \infty} \bigg(|\widetilde V_L|,\,  \widetilde M_L,\,   \frac{K_{\lfloor \cdot L^{2/3}\rfloor\wedge \xi_L}}{L}\bigg)=_{\text{Law}}\Big(|\widetilde B|,\,  \widetilde D,\,   A_{\cdot\wedge a_1}\Big).
\ee
The following relations between $\xi$ and $\widetilde V$ on the one hand and between 
$a$ and $\widetilde B$ on the other hand will be of key importance to get \eqref{cvl}
\be{equa}
a_s=\int_0^s \frac{1}{|B_{a(u)}|} du=\int_0^s \frac{1}{|\widetilde B(u)|} du \quad \text{and} \quad 
\xi_{sL}=\int_0^s \frac{L}{1+|V_{\xi_{ uL}}|} du=\int_0^s \frac{L^{2/3}}{L^{-1/3}+|\widetilde V_L(u)|} du
\ee
where the first equality holds true for $s\in [0,1]$ and the second for 
$s$ in the set $J_L$ of hopping times of $s\to \xi_{sL}$.

\subsubsection{Outline of the proof of  \eqref{cvl}} 
We will follow the scheme below
\begin{enumerate}
\item With the help of Theorem \ref{midlaux}, we infer the Skorohod's representation Theorem 
and state that there exists a sequence of cadlag processes $|\widetilde V'_L|$, $\widetilde M'_L$  and $|\widetilde B'|$, $\widetilde D'$
defined on the same probability space $(\Omega,\cA,\mathbb P)$ such that for $\mathbb P$-a.e. $\omega\in \Omega$
\be{cvps}
\lim_{L\to \infty}  \big|\big| \, |\widetilde V'_L| - |\widetilde B'| \, \big|\big|_{\infty,[0,1]}=0 \quad \text{and} \quad \lim_{L\to \infty}  || \widetilde M'_L - \widetilde D'||_{\infty,[0,1]}=0
\ee
and such that $(|\widetilde V'_L|,\widetilde M'_L)$ has the same law as  $(|\widetilde V_L|,\widetilde M_L)$ under the conditioning $\{L\in \mathfrak{X}\}$ and $(|\widetilde B'|, \widetilde D')$ are two independent Brownian motions of variance $\sigma_\beta$ under the conditioning $\{B_{a_1}=0\}$.
\smallskip

\item We  define for $s\in [0,1]$, the quantity $a'(s)$ with the l.h.s. of formula \ref{equa} applied to $\widetilde B'$ and for 
$L\in \N$ the quantity  $\xi'_{sL}$ with the r.h.s. of  \eqref{equa} applied to $\widetilde V'_L$. Then we show 
\bl{cvpstau} For all $\gep>0$,
\be{supam}
\lim_{L\to \infty}\ \prob{  \sup_{s\in [0,1]} \bigg| \frac{\xi'_{sL}}{L^{2/3}}-a'(s)\bigg|\geq \gep}=0.
\ee
\el
\item  Subsequently, we define the quantity  $A'$ as the inverse of $a'$ and $K'$ as 
$$K'(j)=\max\{i\geq 1\colon \xi'_i=j\}, \quad j\leq \xi'_L.$$
and we show the following convergence in probability.
\bl{cvpsAAA} 
For all $\gep>0$,
\be{limpro}
\lim_{L\to \infty}\ \prob{\sup_{c\in [0,\infty[} \bigg| \frac{K'_{\lfloor cL^{2/3}\rfloor \wedge \xi'_L}}{L}-A'_{c\wedge
a'(1)}\bigg|\geq \gep}=0.
\ee
\el

\item At this stage,   \eqref{cvps} and Lemma \ref{cvpsAAA} allow us to state that 
\be{cvll}
\lim_{L\to \infty} \bigg(|\widetilde V'_L|,\,  \widetilde M'_L,\,   \frac{K'_{\lfloor \cdot L^{2/3}\rfloor\wedge \xi_L}}{L}\bigg)=_{\text{Law}}\Big(|\widetilde B'|,\, \widetilde D',\,   A'_{\cdot\wedge a'(1)}\Big)
\ee
and this implies \eqref{cvl} by a straightforward application of \cite[Lemma p. 151]{Bill08}.
\end{enumerate}

\begin{remark}\label{remtau}
{\rm We note that, since we defined $\xi'$ with the help of formula \eqref{equa}, 
it is a continuous process and therefore it does not have the same law as 
$\xi$ as defined in \eqref{taul}. However this difference is armless because 
the set of times $J'_L$ at which $\xi'$ takes integer values has 
the same law as the set $J_L$ containing the hopping times  of $\xi$ and moreover between two consecutive points 
of $J_L$ (respectively $J'_L$) $\xi$ (resp. $\xi'$) jumps by one unit exactly.}
\end{remark}

\noindent At this stage it remains to prove lemmas \ref{cvpstau} and \ref{cvpsAAA}
and we begin with the proof of \eqref{supam}.

\bpr

We recall \eqref{equa}. The proof of  \eqref{supam} will be complete once  once we show that for every $\gep>0$,
\be{convprob}
\lim_{L\to \infty} \mathbb{P}\Big(\textstyle \int_0^1  \big| \frac{1}{L^{-1/3}+|\widetilde V'_L(u)|}- \frac{1}{|\widetilde B'_{u}|}\big|  du\geq \gep\Big)=0.
\ee 
We define for  $L\in \N$ and $\eta>0$ the three quantities
\begin{align}
C_1&=\int_0^1  \Big|  \frac{1}{L^{-1/3}+|\widetilde V'_L(u)|}- \frac{1}{|\widetilde B'_{u}|}\Big| \,  \ind_{\big\{|\widetilde B'_{u}|>\eta\big\}}\, 
 \ind_{\big\{L^{-1/3}+|\widetilde V'_L(u)|>\eta\big\}} du,\\
C_2&=\int_0^1  \frac{1}{|\widetilde B'_{u}|} \,  \ind_{\big\{|\widetilde B'_{u}|\leq \eta\big\}}\, du,\\
C_3&=\int_0^1  \frac{1}{L^{-1/3}+|\widetilde V'_L(u)|}  
 \ind_{\big\{L^{-1/3}+|\widetilde V'_L(u)|\leq \eta\big\}} du.
\end{align}
We immediately observe that \eqref{cvps} and the dominated convergence theorem yield that for $\eta>0$ and for $\mathbb{P}$-a.e $\omega$
$\lim_{L\to \infty} C_1=0$. Moreover, \eqref{equa} combined with the fact that $\widetilde B$ and $\widetilde B'$ have the same law and with the fact that $a_1<\infty$
yields that  for $\mathbb{P}$-a.e $\omega$ the function $u\mapsto 1/|\widetilde
B'_{u}|$ is integrable on $[0,1]$. This is sufficient to conclude that  for  $\mathbb{P}$-a.e $\omega$, 
$\lim_{\eta \to 0} C_2=0$.

Thus it remains to consider $C_3$. Since  $\widetilde V'_L$ and  $\widetilde V_L$ are equaly  distributed,
it comes that  
\begin{align}
C_3&=_{\text{law}}\frac{1}{L}\int_0^L    \frac{L^{1/3}}{1+|V_{\xi{\lfloor x\rfloor}}|} \, 
 \ind_{\big\{(1+|V_{\xi{ x}}|)\leq \eta L^{1/3}\big\}} \, dx
 \leq\frac{1}{L^{2/3}} \big| \{ j\leq \xi_L\colon\, |V_j|\leq \eta L^{1/3}\}\big|. 
\end{align}
Therefore, the proof of \eqref{convprob} will be complete once we show that for all $\gep>0$
\be{tppr}
\lim_{\eta \to 0} \limsup_{L\to \infty}  \mathbf{P}_{\beta}\Big(  \big| \{ j\leq \xi_L\colon\, |V_j|\leq \eta L^{1/3}\}\big| \geq \gep L^{2/3}\ \big| \, L\in \mathfrak{X}\Big)=0.
\ee

To prove \eqref{tppr}, we use some results obtained in \cite{CP15} under the same conditioning. 
We recall (\ref{premm}--\ref{defX}) and we denote by $(X^i)_{i=1}^{v_L}$ the order statistics of $(X_i)_{i=1}^{v_L}$ and by $(\mathfrak{N}^i)_{i=1}^{v_L}$ the sequence of 
horizontal excursions  reordered according to  the sequence  $(X^i)_{i=1}^{v_L}$.
 Then,  we distinguish between the $k$ largest such excursions and their 
lengths, i.e.,  $(X^i)_{i=1}^{k}$ and   the others, i.e.,   
\be{ortu}
 \frac{1}{L^{2/3}}| \{ j\leq \xi_L\colon\, |V_j|\leq \eta L^{1/3}\}|\leq  A^\eta_{k,L}+ B_{k,L}
\ee
with
\begin{align}\label{supra}
A^\eta_{k,L}&=\frac{1}{L^{2/3}}\, \sum_{j=1}^k 
\big|\big\{i\in \big\{\tau_{r_j-1},\dots,\tau_{r_j}-1\big\}\colon\;  |V_i|\leq \eta L^{1/3}\big\}\big|
\quad \text{and}\quad B_{k,L}=\sum_{i= k+1}^{v_L} \frac{\mathfrak{N}^i}{L^{2/3}}.
\end{align}
In the second step of the proof of \cite[Proposition 4.7]{CP15}, it is shown that  for all $\gep>0$ 
\be{lims}
\lim_{k\to \infty} \limsup_{L\to \infty} \bP_\beta (B_{k,L} \geq \gep | L\in \mathfrak{X})=0.
\ee
Thus, the proof of  \eqref{tppr} will be complete once we show that  for all $k\in \N$
\be{limss}
\lim_{\eta\to 0} \limsup_{L\to \infty} \bP_\beta \bigg( \sum_{j=1}^k 
\big|\big\{i\in \big\{\tau_{r_j-1},\dots,\tau_{r_j}-1\big\}\colon\;  |V_i|\leq \eta L^{1/3}\big\}\big| \geq \gep L^{2/3}\   | \  L\in \mathfrak{X}\bigg)=0.
\ee
which again will be a consequence of the fact that for all $j\in \N$
\be{limss2}
\lim_{\eta\to 0} \limsup_{L\to \infty} \bP_\beta \Big(
\big|\big\{i\in \big\{\tau_{r_j-1},\dots,\tau_{r_j}-1\big\}\colon\;  |V_i|\leq \eta L^{1/3}\big\}\big| \geq \gep L^{2/3}\   | \  L\in \mathfrak{X}\Big)=0.
\ee

The proof of \eqref{limss2} goes as follows. For $i\in \{1,\dots,v_L\}$ we denote by $\mathfrak{E}_j:=(V_i)_{i\in \{\tau_{j-1},\dots,\tau_j-1\}}$ the $j$-th excursion of 
$V$ and we recall that  conditionally  on $(X_i)_{i\in \{1,\dots,v_L\}}=(x_i)_{i\in \{1,\dots,v_L\}}$,  the excursions 
$(\mathfrak{E}_j)_{j\in \{1,\dots, v_L\}}$ are independent and of  law $\mathbf{P}_{\beta}(\cdot \,|\, X_1=x_1)$ 
for $\mathfrak{E}_1$ and $\mathbf{P}_{\beta,\mu_\beta}(\cdot \,|\, X_1=x_j)$  for   $\mathfrak{E}_j$ with $j\geq 2$. Thus, for $j\geq 1$, 
\begin{align}\label{jko}
R_{L,\eta,\gep}(j)&:= \bP_\beta \Big(
\big|\big\{i\in \big\{\tau_{r_j-1},\dots,\tau_{r_j}-1\big\}\colon\;  |V_i|\leq \eta L^{1/3}\big\}\big| \geq \gep L^{2/3}\   | \  L\in \mathfrak{X}\Big)\\
\nonumber  &=\sum_{\ell=1}^{L-j} \bP_\beta(X_{r_j}=\ell \,|\,  L\in \mathfrak{X}) \max_{x\in \{0,\mu_\beta\}} 
\Big\{\bP_{\beta,x}\Big(\big|\big\{i\in \big\{0,\dots,\tau_1-1\big\}\colon\;  |V_i|\leq \eta L^{1/3}\big\}\big|\geq \gep L^{2/3}\,\big |\, X_1=\ell \Big)\Big\}.
\end{align}
We can argue here that $\frac{X_{r_j}}{L}$ under ${\bf P}_\beta(\cdot \mid L\in \mathfrak{X})$ converges in distribution towards the 
$j$-th largest inter-arrival  of an $1/3$-stable regenerative set on $[0,1]$ conditioned on $1$ being in the set. For this reason, and for 
every $\xi>0$ there exists an $m_1>0$ such that for $L$ large enough 
\be{equaborne}
{\bf P}_\beta\big(X_{r_j}\notin [m_1L,L] \mid L\in \mathfrak{X}\big)\leq \xi,
\ee
Thus,  we let 
$$C_{L,\eta,\gep}:=\Big\{V\colon \big|\big\{i\in \big\{0,\dots,\tau_1-1\big\}\colon\;  |V_i|\leq \eta L^{1/3}|\geq \gep L^{2/3} \Big\}$$
and \eqref{limss2} will be proven once we show that 
\be{contborne}
\lim_{\eta \to 0} \limsup_{L\to \infty} \sup_{\ell\in [m_1 L,L]}\,  \max_{x\in \{0,\mu_\beta\}} {\bf P}_{\beta,x} \big(C_{L,\eta,\gep} \mid X=\ell\big)=0.
\ee
From \cite[Proposition 3.4]{CarPet17b} we know that 
$\tau/\ell^{2/3}$ under $\probmubeta{\cdot  \mid X=\ell}$ (or under ${\bf P}_\beta(\cdot\mid X=\ell)$) converges in distribution towards $\mathfrak{R}$ the extension of a Brownian excursion normalized by its area.
Thus, for every $\xi>0$, there exists $[\alpha_1,\alpha_2]\subset (0,\infty)$ such that for $L$ large enough, we state that for every $\ell\in [m_1L,L]$ we have 
$$\max_{x\in \{0,\mu_\beta\}} {\bf P}_{\beta,x}\Big(\, \frac{\tau}{\ell^{2/3}}\in [\alpha_1,\alpha_2]\,  \big | \, X=\ell \Big) \leq \xi.$$ 
As a consequence,  \eqref{jko} will be proven once we show that 
\be{contborne2}
\lim_{\eta \to 0} \limsup_{L\to \infty} \sup_{\ell\in [m_1 L,L]}  \max_{x\in \{0,\mu_\beta\}}  {\bf P}_{\beta,x} \big(C_{L,\eta,\gep} \cap \big\{\tfrac{\tau}{\ell^{2/3}}\in [\alpha_1,\alpha_2] \big\}\mid X=\ell\big).
\ee
At this stage, we introduce the notation  
$(\hat V_s)_{s\in [0,1]}=(\frac{1}{\sqrt{\tau}} |V_{\lfloor s\tau\rfloor}|)_{s\in [0,1]}$ and we note that 
for $\ell\in [m_1L,L]$ and $\tau\in [\alpha_1 m_1^{2/3} L^{2/3}, \alpha_2 L^{2/3}]$ we have  
\be{upset}
C_{L,\eta,\gep} \cap \big\{\tfrac{\tau}{\ell^{2/3}}\in [\alpha_1,\alpha_2]\big\} \subset \bigg\{\int_0^1\  \ind_{\Big\{| \hat V_s|\leq \frac{\eta}{\sqrt{\alpha_1} m_1^{1/3}}\Big\}} 
ds\geq \frac{\gep}{\alpha_2} \bigg\},
\ee
where we have used that $L^{1/3}\leq \frac{\sqrt{\tau}}{\sqrt{\alpha_1} m_1^{1/3}}$
and $L^{2/3}\geq \frac{\tau}{\alpha_2}$.  
At this stage, we use  the convergence established in  \cite[Theorem A]{CarPet17b} and \eqref{upset} to assert that 
\be{grandlim}
\limsup_{L\to \infty} \sup_{\ell \in [m_1 L,L]}  \max_{x\in \{0,\mu_\beta\}}  {\bf P}_{\beta,x} \big(C_{L,\eta,\gep} \cap \big\{\tfrac{\tau}{\ell^{2/3}}
\in [\alpha_1,\alpha_2] \big\}\mid X=\ell\big)
\leq \mathbb{P} \bigg( \int_0^1\  \ind_{\Big\{\frac{1}{\sqrt{\mathfrak{R}}} \cE(s\mathfrak{R})\leq \frac{\eta}{\sqrt{\alpha_1} m_1^{1/3}}\Big\}} 
ds\geq \frac{\gep}{\alpha_2}\bigg)
\ee
and we conclude by observing that $\mathbb{P}$ almost surely $\cE$ is continuous on $[0,\mathfrak{R}_{\cE}]$ and equals $0$
at $0$ and $\mathfrak{R}_\gep$ only. Thus, the r.h.s. in \eqref{grandlim} vanishes as $\eta\to 0$ and this concludes the proof.

%
%
%
%

\epr
%

At this stage, it remains to prove \eqref{limpro}.
\bpr
Let us define 
 $$D_{L,\gep}:=\bigg\{\sup_{c\in [0,\infty[} \bigg| \frac{K_{\lfloor cL^{2/3}\rfloor \wedge \xi_L}}{L}-A_{c\wedge
a_1}\bigg|\geq \gep\bigg\}$$
and for $\eta>0$ we set 
$$E_{L,\eta}=\bigg\{\sup_{s\in [0,1]} \Big |\frac{\xi_{\lfloor sL\rfloor}}{L^{2/3}}-g(s)\Big|\leq \eta\bigg\}.$$
By lemma \ref{cvpstau}, the proof of \eqref{limpro} will be complete once we show that  there exists $\eta>0$ such that 
$\lim_{L\to \infty} P(D_{L,\gep} \cap E_{L,\eta})=0$. Under the event $D_{L,\gep} \cap E_{L,\eta}$
there exists a $c\in [0,a_1\wedge \xi_L/L^{2/3}]$ such that $\big| \frac{1}{L} K_{\lfloor cL^{2/3}\rfloor}-A_{c}\big|\geq \gep$. Assume that $K_{\lfloor cL^{2/3}\rfloor}\geq (A_{c}+\gep) L$ (the other case is treated similarly). Then 
$cL^{2/3} \geq \xi_{K_{\lfloor cL^{2/3}\rfloor}}\geq \xi_{(A_{c}+\gep)L}$ and under $E_{L,\eta}$ we can state that 
$\xi_{(A_{c}+\gep)L}\geq (a_{A_{c}+\gep}-\eta) L^{2/3}$ so that finally
$c\geq a_{A_{c}+\gep}-\eta$. But since $a_{A_{c}}=c$  we get 
$\eta\geq \min\{a_{u+\gep}-a_{u}\colon\, u\in [0,1-\gep]\}$. Since $a$ is $P$-almost surely continuous and strictly increasing on $[0,1]$ we complete the proof of the Lemma by claiming that 
$$\lim_{\eta \to 0} P(\min\{a_{u+\gep}-a_{u}\colon\, u\in [0,1-\gep]\}\leq \eta)=0.$$
\epr

\section{Proof of Theorem \ref{midlaux}} \label{thD}

\subsection{Truncated version of Theorem \ref{midlaux}, i.e., proof of  Proposition \ref{propded}}\label{step3} 
We recall Definition \ref{deftiv2} and for every $L\in \N$, we will generate a random walk path 
$(V_i)_{i=0}^{\xi_L}$ of law ${\bf P_\beta}(\cdot\, |\, L\in \mathfrak{X})$. To begin with, we use Theorem \ref{convhaus} in combination with Skorohod's representation theorem 
(the set of closed subsets in $[0,1]$ endowed with the Hausdorf distance being a Polish space) to assert that 
there exists a sequence of random sets $\mathfrak{X}_L$ and a random  set $\mathfrak{X}_\infty$ defined on the same probability space $(\Omega_1,\cA_1,\mathbb{P}_1)$ and 
such that 
\begin{enumerate}
\item for every  $L\in \N$, $\mathfrak{X}_L$ has the same law as $\mathfrak{X}\cap [0,L]$ with $\mathfrak{X}$ sampled from  ${\bf P}_{\beta}(\cdot\, |\, L\in \mathfrak{X})$. We will denote by 
$(X^L_j)_{j=1}^{v_L}$ the inter-arrivals of $\mathfrak{X}_L$ , i.e., 
$$\mathfrak{X}_L=\big\{0,X_1^L,\dots,X_1^L+\dots+X_{v_L-1}^L,L\big\},$$
\item $X_\infty$ is a $C_{1/3}$ regenerative set intersected with $[0,1]$ and conditioned on $1\in \mathfrak{X}$,
\item $ \lim_{L\to \infty} \frac{1}{L} \mathfrak{X}_L(\omega)=\mathfrak{X}_\infty(\omega)$  for $\mathbb P_1$-a.e. $\omega_1$.
\end{enumerate}

For $k\in \N$ and  every $\omega_1\in \Omega_1$ we denote by 
$(d_1^\infty,f_1^\infty ),\dots,(d_r^\infty,f_r^\infty)$ the positions of the maximal intervals of $[0,1]$ which are not intersecting $\mathfrak{X}_\infty$ and 
are larger than $1/k$. Note that $r:=r^\infty(\omega_1)$ is random and bounded above by $k$. Note also that $X_j^\infty:=f_j^\infty-d_j^\infty>1/k$ for every $j\in \{1,\dots,r\}$.  
We also denote by $r_L$ the number of intervals of $[0,L]$ larger than $L/k$ whose extremities are consecutive points of $\mathfrak{X}_L$. We let $(d_j^L,f_j^L)_{j=1}^{\,r_L}$ be the extremities of 
those intervals and $(X_{e_{j,L}}^L)_{j=1}^{\,r_L}$ their associated inter-arrivals (i.e., $1\leq e_{1,L}<\dots<e_{r_L,L}\leq v_L$).  Because of the almost sure convergence of $\frac{1}{L} \mathfrak{X}_L$ towards $\mathfrak{X}_\infty$ we can claim that for $L$ large enough $r_L$ equals $r$ and moreover that
for $\mathbb{P}_1$-a.e. $\omega_1$ and $j\in \{1,\dots,r\}$
\begin{align}\label{convsetre}
\lim_{L\to \infty} (d_j^L,f_j^L)=(d_j^\infty,f_j^\infty) \quad \text{and} \quad \lim_{L\to \infty} X_{e_{j,L}}^L=X_j^\infty.
\end{align}
We will also need the notations 
$$
\Lambda_\infty=\cup_{j=1}^r [d_j^\infty,f_j^\infty]\quad  \text{and}\quad \Lambda_L=\cup_{j=1}^{r_L} [d_j^L,f_j^L).
$$

At this stage, we sample a family of independent random variables $(Y^{j,N})_{(j,N)\in \N^2}$ on  a probability space 
$(\Omega_2, \cA_2, \mathbb{P}_2)$ as follows 
\begin{enumerate}
\item for every $(j,N)\in \N^2$ the random variable $Y_{j,N}$ is a Bernoulli with parameter ${\bf P}_{\beta,\mu_\beta}(V_0\neq 0 |\, X=N)$. 
\end{enumerate}
At this stage we use Theorem \ref{convexc} and the Skorohod's representation theorem to define on the same probability space 
$(\Omega_3,\cA_3,\mathbb{P}_3)$ 
a family of independent sequences of discrete random processes
$\big\{(\mathfrak{E}^{j,y}_N)_{N\in \N}, j\in \N, y\in \{0,1\}\}$ and a family of independent continuous random processes
$\big\{(\cE^{j,y}(s),D^{j,y} (s))_{\,s\in [0,a^j(1)]},\, j\in \N, y\in \{0,1\} \big\}$
 such that 
\smallskip

\begin{enumerate}
\item for every $(j,N)\in \N^2$ the random process 
$\mathfrak{E}_{N}^{j,0}=(V_i^{j,0,N})_{i=0}^{\tau_{j,0,N}}$ has for law ${\bf P}_{\beta,\mu_\beta}(\cdot\,| \, V_0=0,\, X=N)$, 
and the random process $\mathfrak{E}_{N}^{j,1}=(V_i^{j,1,N})_{i=0}^{\tau_{j,1,N}}$ has for law ${\bf P}_{\beta,\mu_\beta}(\cdot\,| \, V_0\neq 0,\, X=N)$,
\item for every $(j,y)\in \N\times \{0,1\}$,  $\cE^{j,y}$ is a Brownian excursion normalized by its area and $a_{j,y}$ is the inverse function of this 
area and $D^{j,y}$ is a standard brownian motion independent of $\cE^{j,y}$,
\item for every $(j,y)\in \N\times \{0,1\}$ and for $\mathbb{P}_3$-a.e. $\omega_3$ 
\begin{align}\label{convsko}
\lim_{N\to \infty} \sup_{s\in [0,1]} \Big| \frac{1}{N^{1/3}} \big| V^{j,y,N}_{\xi^{j,y,N}_{sN}}\big |-\cE^{j,y} (a_{j,y}(s))\Big|=0\quad \text{and} \quad \lim_{N\to \infty}  \sup_{s\in [0,1]} \Big| \frac{1}{N^{1/3}} M^{j,y,N}_{\xi^{j,y,N}_{sN}}-D^{j,y} (a_{j,y}(s))\Big|,
\end{align}
where $\xi^{j,y,N}$ is the pseudo-inverse of the geometric area (plus extension) of $\mathfrak{E}_N^{j,y}$ defined as in \eqref{taul} and where $M^{j,y,N}$ is the center-of-mass walk associated with $\mathfrak{E}^{j,y}_N$, i.e., 
\begin{align}\label{defcentexc}
\nonumber M^{j,y,N}_i&=\sum_{t=0}^{i-1} (-1)^{t-1} |V_{t}^{j,y,N}|+(-1)^{i-1} \frac{|V_{i}^{j,y,N}|}{2}, \quad i\in \{0,\dots,\tau_{j,y,N}-1\},\\
M^{j,y,N}_{\xi^{j,y,N}_N}&=M^{j,y,N}_{\tau_{j,y,N}}=\sum_{t=0}^{\tau_{j,y,N}-1} (-1)^{t-1} |V_{t}^{j,y,N}|.
\end{align}
\end{enumerate}
Let us note that for every $(j,N)\in \N^2$, the random process $\mathfrak{E}_{N}^{j,Y_{j,N}}$ is an excursion of law 
${\bf P}_{\beta,\mu_\beta}(\cdot \,|\, X=N)$.
To lighten the notations we will drop the $L$ dependency of $(X_{j}^L)$,  $(d_j^L,f_j^L)$ and $e_{j,L}$ when there is no risk of confusion.

With these tools in hand, we apply for every $L\in \N$  the construction of $V$ (given in Definition \ref{deftiv2}) with law ${\bf P}_{\beta,\mu_\beta}(\cdot\, |\, L\in \mathfrak{X})$ on a probability space 
$(\times_{i=1}^4 \Omega_i, \otimes_{i=1}^4 \cA_i, \otimes_{i=1}^4 \mathbb P_i)$.  Of course
$\mathfrak{X}_L$ plays the role of $\mathfrak{X}$, then for every $t\in \{1,\dots,v_L\}$ we sample on $(\Omega_4, \cA_4, \mathbb P_4)$ an excursion $V^t$ of law ${\bf P}_{\beta,\mu_\beta}(\cdot\, |\, X=X_t)$ and independently  an i.i.d. sequence of Bernoulli trials  $(\gep_i)_{i\in \N}$.
We do so except that for the indices $(e_{j})_{j=1}^{r_L}$ we replace each excursion $V^{e_j(\omega_1)}(\omega_4)$ by the excursion  $\mathfrak{E}^{j,Y_{j,X_{e_j}}(\omega_2)}_{X_{e_{j}}(\omega_1)}(\omega_3)$   defined above.  

\smallskip

 For every $s \in [0,1]$ we set $j_{s}:=\min \{t\leq r_L\colon\, f_t\geq  sL\}$ and for $j\in \{1,\dots,r_L\}$ we let $\alpha_{j,L}$ be the product of $(-1)^{\tau_{e_{j}-1}}$ with the sign of the 
 excursion of $V$ indexed by $e_j$, i.e., 
\begin{align}\label{defalpha}
\alpha_{j,L}&=(-1)^{\tau_{e_{j}-1}} \sign(V_{\tau_{e_j}-1})=(-1)^{\tau_{e_j-1}}\Big(\ind_{\big\{V_{\tau_{e_j-1}}=0\big\}}\,  
\gep_{e_j}- \ind_{\big\{V_{\tau_{e_j-1}}\neq 0\big\}} \sign(V_{\tau_{e_j-1}-1})\Big)\\
\nonumber &=(-1)^{\tau_{e_j-1}}\Big(\ind_{\big\{Y_{j,X_{e_j}}=0\big\}}\,  
\gep_{e_j}- \ind_{\big\{Y_{j,X_{e_j}}=1\big\}} \sign(V_{\tau_{e_j-1}-1})\Big).
\end{align}
For the ease of notations, for $t\leq k$ and $u\in [0,\infty)$,  we will use the following shortcuts  in the computations below: 
\begin{align}\label{shorter}
\widehat V_{u}^t:=V^{\, t,\,Y_{t,X_{e_{t}}},\, X_{e_{t}}}_{\xi^{t,\,Y_{t,X_{e_{t}}},\, X_{e_{t}}}_{u}} \quad \text{and} \quad \widehat M_{u}^t:=M^{\, t,\,Y_{t,X_{e_{t}}},\, X_{e_{t}}}_{\xi^{t,\,Y_{t,X_{e_{t}}},\, X_{e_{t}}}_{u}} 
\end{align}
We recall (\ref{defexcmid}--\ref{tronc}) and \eqref{defcentexc} and we observe that the truncated processes 
 $\widetilde M_{L,k}$ and $\widetilde V_{L,k}$ obtained  from $V$ can be written as  
\begin{align}\label{defvl}
\nonumber |\widetilde V_{L,k}|(s)&= \ind_{\Lambda_L} (sL) \  \frac{1}{L^{1/3}}\  |V_{\xi_{sL}}|=\ind_{\Lambda_L} (sL)\  \frac{1}{L^{1/3}}\     \big|\widehat V_{sL-d_{j_{s}}}^{\, j_s}\big| \\
&= \ind_{\Lambda_L} (sL)\ \bigg [\frac{X_{e_{j_{s}}}}{L}\bigg ]^{1/3}\,   \Bigg (  \bigg[\frac{1}{X_{e_{j_{s}}}}\bigg ]^{1/3}\  \big|\widehat V_{sL-d_{j_{s}}}^{\, j_s}\big|   \Bigg ),
\end{align}
and
\begin{align}\label{defml}
\widetilde M_{L,k}(s)&=\frac{1}{L^{1/3}} \Bigg[\sum_{t=1}^{j_{s}-1} M^{\text{exc}}(e_{t})+ \ind_{\Lambda_L} (sL)
\bigg(\sum_{i=\tau_{e_{j_{s}}-1}}^{\xi_{sL}-1} (-1)^{i-1} V_i+ (-1)^{\xi_{sL}-1} \frac{V_{\xi_{sL}}}{2}\bigg)\Bigg]\\
\nonumber & = \sum_{t=1}^{j_{s}-1} \Bigg[\frac{X_{e_t}}{L}\Bigg]^{1/3}  \alpha_{t,L}\  \Bigg(\Bigg[\frac{1}{X_{e_t}}\Bigg]^{1/3}  \widehat M^{\,t}_{X_{e_t}} \Bigg)+ \ind_{\Lambda_L} (sL)
 \Bigg[\frac{X_{e_{j_s}}}{L}\Bigg]^{1/3}  \alpha_{j_s,L}\  \Bigg(\Bigg[\frac{1}{X_{e_{j_s}}}\Bigg]^{1/3} \widehat  M^{\, j_s}_{sL-d_{j_s}}\Bigg).
\end{align}
To complete the proof of Proposition \ref{propded} it remains to identify the limiting distribution of  $(|\widetilde V_{L,k}|, \widetilde M_{L,k})$ (defined in in (\ref{defvl}--\ref{defml})). 
The tension of  $(|\widetilde V_{L,k}|, \widetilde M_{L,k})$  is ensured first by the fact that $r_L\leq k$ (for every $L\in \N$) and then by the fact that for every $(j,y)\in \{1,\dots,k\}\times \{0,1\}$
the convergences in \eqref{convsko}  ensure that for $N\geq N_{j,y}(\omega_3)$ the modulus of continuity of 
$\Big(\frac{1}{N^{1/3}} V^{j,y,N}_{\xi^{j,k,N}_{sN}},\frac{1}{N^{1/3}} M^{j,y,N}_{\xi^{j,y,N}_{sN}}\Big)_{s\in [0,1]}$ are arbitrarily small. We conclude by saying that with high probability $L/k$ is larger than $\max_{j\leq k, y\in\{0,1\}} N_{j,y}$.  Therefore, we need to obtain the limiting law of finite dimensional distributions of $(|\widetilde V_{L,k}|, \widetilde M_{L,k})$
as $L\to \infty$.
To that aim, 
we define two auxiliary processes $|\tilde B_{k,L}|$ and $\tilde D_{k,L}$ as 
\begin{align}\label{def:B}
|\tilde B_{k,L}|(s)&=\ind_{\Lambda_\infty}(s)\   (X_{j_s}^\infty)^{1/3} \   \cE^{j_s,Y_{j_s}}\Big(a_{j_s,Y_{j_s}}\Big(\tfrac{s-d_{j_s}^\infty}{X_{j_s}^\infty}\Big)\Big), \quad s\in [0,1]
\end{align}
where for every $t\leq r$, we set 
$Y_t:=Y_t(\omega_1,\omega_2)=Y_{t,\, X_{e_{t,L}}^L}$  (which actually explains  the $L$-dependency of  $|\widetilde B_{L,k}|$) and 
\begin{align}\label{def:D}
\tilde D_{k,L}(s)&=\sum_{t=1}^{j_s-1}  (X_t^\infty)^{1/3} \ \alpha_{t,L}\  D^{t,Y_t} \big(a_{t,Y_t}(1)\big)+ \ind_{\Lambda_\infty}(s)\   (X_{j_s}^\infty)^{1/3} \  \alpha_{j_s,L}  \ D^{j_s,Y_{j_s}} \Big(a_{j_s,Y_{j_s}}\Big(\tfrac{s-d_{j_s}^\infty}{X_{j_s}^\infty}\Big)\Big), \quad s\in [0,1].
\end{align}
We will complete the proof by first observing that (\ref{convsetre}--\ref{convsko}) implies that  for every $s\in [0,1]$  the following convergences occurs for $\otimes_{i=1}^4 \mathbb P_i$-a.e. $(\omega_i)_{i=1}^4\in \times_{i=1}^4 \Omega_i$, i.e., 
\be{convdiff}
\lim_{L\to \infty}   |\widetilde V_{L,k}(s)|-|\tilde B_{k,L}(s)|=0 \quad \text{and} \quad \lim_{L\to \infty}  \  \widetilde M_{L,k}(s)-\tilde D_{k,L}(s)=0
\ee
and then by showing that  for every $L\in \N$, the two dimensional process $(|\tilde B_{k,L}|, \tilde D_{k,L})$ has the same distribution as that of $(|\tilde B^k|, \tilde D^k)$
defined in the statement of Proposition \ref{propded}. To prove this last point we prove below that we do not change the law of $(|\tilde B_{k,L}|, \tilde D_{k,L})$ by removing  the terms $(\alpha_{t,L})_{t=1}^r$ in \eqref{def:D}. The resulting process $(|\tilde B_{k,L}|, \tilde D_{k,L})$  does not depend on $L$ anymore and a straightforward consequence of Theorem \ref{Bessel}
ensures that this process is distributed as  $(|\tilde B^k|, \tilde D^k)$.
 
 For $t\in \N$ we let 
$\cF_t$ be the sub-$\sigma$-algebra of $\cA_3$ defined as 
\be{defF}
\cF_t=\sigma(\mathfrak{E}_{N}^{j,y}, \mathcal{E}^{j,y},D^{j,y};\,  j\leq t,\, y\in \{0,1\},\, N\in \N).
\ee
The proof will be complete once we show that for every $t\leq r$ the law of the 
\be{firstlaw}
\Big(\cE^{t,Y_{t}}\Big(a_{t,Y_{t}}\Big(\tfrac{s-d_{t}^\infty}{X_{t}^\infty}\Big)\Big),
\alpha_{t,L}\,D^{t,Y_{t}}\Big(a_{t,Y_{t}}\Big(\tfrac{s-d_{t}^\infty}{X_{t}^\infty}\Big)\Big)\Big)_{s\in [0,d_t^\infty+X_t^\infty]}
\ee
conditioned on the sub-$\sigma$-algebra $\cA_1\otimes \cA_2 \otimes \cF_{t-1}\otimes \cA_4$  equals the law of 
\be{seclaw}
\Big(\cE\Big(a\Big(\tfrac{s-d_{t}^\infty}{X_{t}^\infty}\Big)\Big),
D\Big(a\Big(\tfrac{s-d_{t}^\infty}{X_{t}^\infty}\Big)\Big)\Big)_{s\in [0,d_t^\infty+X_t^\infty]}
\ee 
with $\cE$ a Brownian excursion normalized by its area, $a$ the inverse function of this 
area and $D$ is a standard brownian motion independent of $\cE$. Such an equality indeed allows us to compute the characteristic function of any 
finite dimensional distribution of $(|\tilde B_{k,L}|, \tilde D_{k,L})$ by conditioning successively on $\cF_{t-1}$ from $t=r$ up to $t=1$, getting rid at each step 
of the random variable $\alpha_{t,L}$. 

To prove this later  equality in law  we note first  that the law of $$\Big(\cE^{t,Y_{t}}\Big(a_{t,Y_{t}}\Big(\tfrac{s-d_{t}^\infty}{X_{t}^\infty}\Big)\Big),\,D^{t,Y_{t}}\Big(a_{t,Y_{t}}\Big(\tfrac{s-d_{t}^\infty}{X_{t}^\infty}\Big)\Big)\Big)_{s\in [0,d_t^\infty+X_t^\infty]}$$  conditioned on  $\cA_1\otimes \cA_2 \otimes \cF_{t-1}\otimes \cA_4$
does not depend on $Y_t(\omega_1,\omega_2)$ and equals the law of  \eqref{seclaw}, second that the random variable $\alpha_{t,L}$ is  $\cA_1\otimes \cA_2 \otimes \cF_{t-1}\otimes \cA_4$-measurable and takes values $-1$ and $1$ only, third  that for any $c\in \{-1,1\}$ the laws of 
$(\cE,D)$ and $(\cE, c D)$ are equal.

\subsection{The center-of-mass walk outside large excursions, i.e., proof of Proposition~\ref{proposui}}\label{step4}

Proposition \ref{proposui} contains two limits. We will display the proof of the second limit only since the proof of the first limit is way easier. 
To be more specific, the first limit gives some control on the fluctuations of the $V$ random walk sampled from 
${\bf P}_{\beta,\mu_\beta}(\cdot\, |\cdot L\in \mathfrak X)$ outside its largest excursions (in terms of geometric area swept). The second 
limit is much more involved, essentially because, despite the $V$ random walk, the center-of-mass walk does not come back close to the origin at the end of every excursion of $V$.
 
We recall \eqref{tronc3} and we observe that for every $s\in [0,1]$ we have 
\be{ecrit}
\widetilde M_L(s)-\widetilde M_{L,k}(s)=\frac{1}{L^{1/3}} \big(M_{\xi_{sL}}- M_{L/k}^{+}(\xi_{sL})\big)
\ee
Therefore, Proposition \ref{proposui} will be proven once we show that  for every  $\eta>0$

\be{difftruncb}
\lim_{k\to \infty} \limsup_{L\to \infty} 
\probmubeta{\max_{n \leq \xi_L} \big | M_{\xi_{sL}}- M_{L/k}^{+}(\xi_{sL}) \big | \ \geq \eta L^{1/3} \ \Big |\  L\in\mathfrak{X}}=0.
\ee

The proof of \eqref{difftruncb} is divided into $4$ steps. In the first step, we prove \eqref{difftruncb} subject to  Claims \ref{c1} and \ref{c2}. Claim \ref{c1}  provides  a control on the fluctuations of the process whose increments are the altitude differences of the center-of-mass walk between the endpoints of each small excursions (in terms of area swept). Claim \ref{c2}, in turn,  provides  a control on the  fluctuations of the center-of-mass walk inside each such small excursions. Those two claims are subsequently proven in Steps 2 and 3 respectively. Note that for the proof of Claims 
\ref{c1} and \ref{c2} we use the alternative  construction of the $V$ trajectory \emph{excursion by excursion} displayed in Definition \ref{deftiv} (see Section \ref{mren}). Note also that proving Claims 
\ref{c1} and \ref{c2} requires to use Lemma \ref{boundtauti}  which is proven in step 4 and provides an upper-bound on the expectation of an auxiliary stopping time.

\subsubsection*{Step 1: Proof of \eqref{difftruncb} subject to Claims \ref{c1} and \ref{c2}}
We recall \eqref{defexcmid} and for every $j\in \N$ we set 
\begin{align}\label{deftg}
T_{j,\frac{L}{k}}:=&  \sum_{r=1}^j M^{\text{exc}}(r)  \, \ind_{\big\{X_r\leq \frac{L}{k}\big\}} \\
\nonumber F_j:=& \max_{i\in \{\tau_{j-1},\dots,\tau_j-1\}}
\Big| \sum_{s=\tau_{j-1}}^{i-1} (-1)^{s-1} V_s+ (-1)^{i-1} \frac{V_i}{2}\Big|=\max_{i\in \{0,\dots,\tau_j-\tau_{j-1}-1\}} 
\Big| \sum_{s=0}^{i-1} (-1)^{s-1} V_s^j+ (-1)^{i-1} \frac{V_i^j}{2}\Big|.
\end{align}
In this step we prove  \eqref{difftruncb} subject to  Claims \ref{c1} and \ref{c2} and to Lemma \ref{boundvl} below.
\begin{claim}\label{c1}
For every $c>0$ and $\eta>0$, 
\be{diffale1}
\lim_{k\to \infty} \limsup_{L\to \infty} 
\probmubeta{ \max_{j \leq cL^{1/3}} \big |T_{j,\frac{L}{k}}\big| \ \geq \eta L^{1/3}}=0.
\ee
\end{claim}

\begin{claim}\label{c2}
For every $c>0$ and $\eta>0$,
\be{diffale2}
\lim_{k\to \infty} \limsup_{L\to \infty} 
\probmubeta{ \max_{j\leq cL^{1/3} \colon X_j\leq \frac{L}{k}} F_j\ \geq \eta L^{1/3}}=0.
\ee
\end{claim}

\begin{lemma}[Lemma 4.12 \cite{CP15}] \label{boundvl}
For $\beta>0$ and $\gep>0$ there exists a $c_\gep>0$ such that for $L\in \N$,
$$\probmubeta{v_L\geq c_\gep L^{1/3}}\leq \gep.$$ 
\end{lemma}
Lemma \ref{boundvl} is the same as Lemma 4.12 in \cite{CP15} except that it is stated there under $\mathbf{P}_\beta$ instead 
of $\mathbf{P}_{\beta,\mu_\beta}$ but the proof is literally the same and we will not repeat it here.  
\medskip

We first recall (\ref{deftg}) and we  observe that when $L\in \mathfrak{X}$ then $\xi_L=\tau_{v_L}-1$ and we can write
\be{gfyu}
\max_{n \leq \xi_L} \big | M_{n}- M_{L/k}^{+}(n)\big | =\max_{n \leq \tau_{v_L}-1} \big | M_{n}- M_{L/k}^{+}(n) \big | \leq \max_{j \leq v_L} \big |T_{j,\frac{L}{k}}\big|
+ \max_{j\leq v_L \colon X_j\leq \frac{L}{k}} F_j.
\ee
Therefore, \eqref{difftruncb} (and consequently Proposition~\ref{proposui}) will be proven once we show that for every $\eta>0$
\be{inegwcond}
\lim_{k\to \infty} \limsup_{L\to \infty} 
\probmubeta{ \max_{j \leq v_L} \big |T_{j,\frac{L}{k}}\big| \ \geq \eta L^{1/3} \ \Big |\  L\in\mathfrak{X}}=0,
\ee
and also
\be{ineggl}
\lim_{k\to \infty} \limsup_{L\to \infty} 
\probmubeta{ \max_{j\leq v_L \colon X_j\leq \frac{L}{k}} F_j\ \geq \eta L^{1/3} \ \Big |\  L\in\mathfrak{X}}=0.
\ee
We recall that 
\be{vl}
v_L=\max\{j\geq 1\colon\, X_1+\dots+X_j\leq L\},
\ee
and we set 
\be{vlp}
v'_{3L/4}:=v_L-\min\big\{j\geq 1\colon\, X_1+\dots+X_j\geq \tfrac{L}{4}\big\}.
\ee
We  observe that 
\begin{align}\label{T1}
\max_{j \leq v_L} \big |T_{j,\frac{L}{k}}\big|&\leq \max_{j \leq v_{3L/4}} \big |T_{j,\frac{L}{k}}\big|+ \max_{j \in \{v_{L/2}+1,\dots, v_{L}\}} \big |T_{j,\frac{L}{k}}\big|
\end{align}
but we can bound from above 
\begin{align}\label{T2}
 \max_{j \in \{v_{L/2}+1,\dots, v_{L}\}} \big |T_{j,\frac{L}{k}}\big|&\leq  \big |T_{v_{L/2},\frac{L}{k}}\big|+\max_{j \in \{v_{L/2}+1,\dots, v_{L}\}}
  \Big|\sum_{r=v_{L/2}+1}^j M^{\text{exc}}(r)  \, \ind_{\big\{X_r\leq \frac{L}{k}\big\}}\Big|\\
 \nonumber  &\leq  \big |T_{v_{L/2},\frac{L}{k}}\big|+   \Big|\sum_{r=v_{L/2}+1}^{v_L} M^{\text{exc}}(r)  \, \ind_{\big\{X_r\leq \frac{L}{k}\big\}}\Big|
  + \max_{j \in \{v_{L/2}+2,\dots, v_{L}\}}
  \Big|\sum_{r=j}^{v_L} M^{\text{exc}}(r)  \, \ind_{\big\{X_r\leq \frac{L}{k}\big\}}\Big|.
\end{align}
At this stage, we note that 
\be{egaltemp}
v_L+1-v'_{3L/4}=1+\min\{j\geq 1\colon\, X_1+\dots+X_j\geq \tfrac{L}{4}\}
\ee
and then either $L/4\in \mathfrak{X}$ and the r.h.s. in \eqref{egaltemp} equals $v_{L/4}+1$
or the r.h.s. in \eqref{egaltemp} equals $v_{L/4}+2$. In this last case, we note that 
$v_{L/4}+2\leq v_{L/2}+1$ except if $v_{L/4}=v_{L/2}$ but this means that $X_{v_{L/4}+1}=X_{v_{L/2}+1}>L/4$
and since the excursions associated with a geometric area larger than $L/k$ are not taken into account in the present computation 
it suffices to choose $k\geq 5$ to make sure that
\begin{align}\label{T3}
 \noindent  \max_{j \in \{v_{L/2}+1,\dots, v_{L}\}}
  \Big|\sum_{r=j}^{v_L} M^{\text{exc}}(r)  \, \ind_{\big\{X_r\leq \frac{L}{k}\big\}}\Big|& \leq  \max_{j \in \{1,\dots, v'_{3L/4}\}}
  \Big|\sum_{r=1}^{j} M^{\text{exc}}(v_L+1-r)  \, \ind_{\big\{X_{v_{L}+1-r}\leq \frac{L}{k}\big\}}\Big|.
\end{align}
We can finally use (\ref{T1}--\ref{T3}) to conclude that 
\be{T4}
\max_{j \leq v_L} \big |T_{j,\frac{L}{k}}\big| \leq  \max_{j \leq v_{3L/4}} \big |T_{j,\frac{L}{k}}\big| +2   \max_{j \in \{1,\dots, v'_{3L/4}\}}
  \Big|\sum_{r=1}^{j} M^{\text{exc}}(v_L+1-r)  \, \ind_{\big\{X_{v_L+1-r}\leq \frac{L}{k}\big\}}\Big|.
\ee
In the same spirit we bound from above
\be{boundgl}
 \max_{j\leq v_L \colon X_j\leq \frac{L}{k}} F_j\leq  \max_{j\leq v_{3L/4} \colon X_j\leq \frac{L}{k}} F_j + \max_{j\leq v'_{3L/4} \colon X_j\leq \frac{L}{k}} F_{v_L+1-j}.
\ee
By reversibility, we note that, under $\probmubeta{ \cdot \  |\  L\in\mathfrak{X}}$, the following equalities in distribution hold true 
\begin{align}\label{equala}
\nonumber \max_{j \leq v_{3L/4}} \big |T_{j,\frac{L}{k}}\big|&=_{\text{law}} \max_{j \in \{1,\dots, v'_{3L/4}\}}
\Big|\sum_{r=1}^{j} M^{\text{exc}}(v_L+1-r)  \, \ind_{\big\{X_{v_L+1-r}\leq \frac{L}{k}\big\}}\Big|\\
\max_{j\leq v_{3L/4} \colon X_j\leq \frac{L}{k}} F_j &=_{\text{Law}}  \max_{j\leq v'_{3L/4} \colon X_j\leq \frac{L}{k}} F_{v_L+1-j}.
\end{align}  
Thus,  we can conclude from (\ref{T4}--\ref{equala}) that \eqref{inegwcond} and \eqref{ineggl} will be proven once we show that for every  $\eta>0$ we have 
\begin{align}\label{limitsb}
\nonumber  & \lim_{k\to \infty} \limsup_{L\to \infty} \probmubeta{  \max_{j\leq v_{3L/4} \colon X_j\leq \frac{L}{k}} F_j  \geq \eta L^{1/3} \, \Big|\, L\in \mathfrak{X} }=0,   \\
&  \lim_{k\to \infty} \limsup_{L\to \infty} \probmubeta{ \max_{j\leq v_{3L/4}} \big|T_{j,\frac{L}{k}} \big | \geq \eta L^{1/3} \, \Big|\, L\in \mathfrak{X} }=0. 
\end{align}
At this stage, the  $2$ inequalities in 
\eqref{limitsb} are straightforward consequences of Lemma  \ref{boundvl} above, of Lemma \ref{rem}  and of Claims \ref{c1} and \ref{c2}. Lemmas \ref{rem} and \ref{boundvl}  indeed imply that it is sufficient to prove both inequalities in \eqref{limitsb} without the conditioning $\{L\in \mathfrak{X}\}$ and with 
$cL^{1/3}$ instead of $v_{3L/4}$ so that we are left with Claims \ref{c1} and \ref{c2}.

\bl{rem}
For every $\beta>0$, there exists a $M>0$ such that, for every function 
$G:\cup_{k=1}^\infty \Z^k \to \R^+$ and every $L\in \N$, we have
\be{bounddd}
\bE_{\beta,\mu_\beta}\Big[ G\Big(V_0,\dots,V_{\tau_{v_{3L/4}}-1}\Big)\,   \big|\, L\in \mathfrak{X}\Big] \leq M\   \bE_{\beta, \mu_\beta}\Big[ G\Big(V_0,\dots,V_{\tau_{v_{3L/4}}-1}\Big)\Big].
\ee
\el

\begin{proof}
We compute the Radon Nikodym density of the image measure of $\bP_{\beta,\mu_\beta}(\cdot | L\in \mathfrak{X})$ by $(V_0,\dots,V_{\tau_{v_{3L/4}}-1})$ 
w.r.t. its counterpart without conditioning. For $y\in \{1,\dots,\frac{3L}{4}\}$,  $t\in \{1,\dots,y\}$, $m\in \{0,\dots,t\}$ and $(z_0, z_1,\dots,z_{t-1})\in \Z^t$ satisfying 
$t+|z_1|+\dots+|z_{t-1}|=y$ we obtain
$$\frac{\bP_{\beta,\mu_\beta}(v_{3L/4}=m, S_{m}=y, \tau_{m}=t,  (V_0,\dots,V_{t-1})=(z_0,\dots,z_{t-1})  | L\in \mathfrak{X})}{\bP_{\beta,\mu_\beta}(v_{3L/4}=m, S_{m}=y, \tau_{m}=t,  (V_0,\dots,V_{t-1})=(z_0,\dots,z_{t-1}) )}:=G_L(y)+K_L(y),$$
with 
\begin{align}\label{deff}
\nonumber G_L(y)&= \frac{\sum_{n=0}^{L/8} \bP_{\beta,\mu_\beta}(n\in \mathfrak{X})\ \ \bP_{\beta,\mu_\beta}(X=L-n-y)}{\bP_{\beta,\mu_\beta}(L\in \mathfrak{X}) \ 
P_{\beta,\mu_\beta}(X\geq \frac{3L}{4}-y)},\\
K_L(y)&= \frac{\sum_{n=1+L/8}^{L/4} \bP_{\beta,\mu_\beta}(n\in \mathfrak{X})\ \ \bP_{\beta,\mu_\beta}(X=L-n-y)}{\bP_{\beta,\mu_\beta}(L\in \mathfrak{X}) \ 
\bP_{\beta,\mu_\beta}(X\geq \frac{3L}{4}-y)}.
\end{align}
The rest of the proof consists in showing that  $G_L(y)$ and $K_L(y)$ are   bounded above uniformly in $L\in \N$ and
$y\in \{0,\dots,3L/4\}$.
 We will focus on $G_L$ since $K_L$ can be treated similarly.
The constants $c_1,\dots,c_4$ below are positive and independent of $L,n,y$.
By recalling \eqref{tailest} and since $L-n-y\geq L/4$ when $n\in \{0,\dots,L/8\}$ we can claim that in the numerator of $G_L(y)$, the term
$P_{\beta,\mu_\beta}(X=L-n-y)$ is bounded above by $c_1/L^{4/3}$  independently of $n$ while  $\sum_{n=0}^{L/4} P_{\beta,\mu_\beta}(n\in \mathfrak{X})
\leq c_2 L^{1/3}$ for every $L\in \N$. For the denominator,  \eqref{tailest}  tells us that $\bP_{\beta,\mu_\beta}(L\in \mathfrak{X})\geq  c_3/L^{2/3}$ and that 
$$\bP_{\beta,\mu_\beta}(X\geq \tfrac{3L}{4}-y)\geq \bP_{\beta,\mu_\beta}(X\geq \tfrac{3L}{4})\geq \tfrac{c_4}{L^{1/3}},\quad L\in \N,\  y\in \{0,\dots,3L/4\}.$$  
This terminates the proof.


\end{proof}

\subsubsection*{Step 2: proof  of Claim \ref{c1}}
We recall \eqref{defexcmid} and in the  Definition \ref{deftiv}. For every $j\geq 1$ we set 
\be{defR}
R_j:=\sum_{i=0}^{\tau_{j}-\tau_{j-1}-1} (-1)^{i-1}\,  | V_{\tau_{j-1}+i}|=\sum_{i=0}^{\tau_1^{j}-1} (-1)^{i-1}\,  
|V^j_{i}|
\ee
so that $(R_j)_{j\geq 1}$ is an i.i.d. sequence of random variables satisfying $|R_j|=|M^{\text{exc}}(j)|$ for every $j\geq 1$.
We note that for $n\geq 2$, 
\be{redf}
T_{n,\frac{L}{k}}:=T_{n-1,\frac{L}{k}}+ (-1)^{\tau_{n-1}} \Big[ \ind_{\big\{V_{\tau_{n-1}}=0\big\}} \, \gep_n +  \ind_{\big\{V_{\tau_{n-1}}\neq 0\big\}}\,  (-\text{sign}(V_{\tau_{n-1}-1})) \Big] \,  R_n  \ind_{\big\{X_n\leq \frac{L}{k}\big\}}
\ee
and we define the filtration $(\mathcal{F}_n)_{n\geq 1}$ by
$$\mathcal{F}_n:=\sigma(\gep_1,\dots,\gep_n, (V_i)_{i\leq \tau_{n}-1}), \quad n\in \N.$$
The expectation of $T_{n,\frac{L}{k}}$ conditioned by $\cF_{n-1}$ is easily computed since 
$T_{n-1,\frac{L}{k}}$ is $\cF_{n-1}$ measurable. Therefore, for $n\geq 2$ we obtain
\be{condesp}
{\bf E}_{\beta,\mu_\beta} \big[T_{n,\frac{L}{k}}\, |\, \cF_{n-1}\big]= T_{n-1,\frac{L}{k}}
-(-1)^{\tau_{n-1}} \,  \text{sign}(V_{\tau_{n-1}-1})  \, {\bf E}_{\beta,\mu_\beta} \Big[ \ind_{\big\{V_0\neq 0\big\}} R_1\ind_{\big\{X_1\leq \frac{L}{k}\big\}}\Big],
\ee
and with the subsequent notation
\be{phi}
\varphi_1(x)={\bf E}_{\beta,\mu_\beta} \Big[ \ind_{\big\{V_0\neq 0\big\}} R_1\ind_{\big\{X_1\leq x\big\}}\Big],\quad x\in \N,
\ee
we can rewrite $T_{n,\frac{L}{k}}=Q_{n,\frac{L}{k}}-J_{n,\frac{L}{k}}$ such that 
\begin{align}\label{defQJ}
Q_{n,\frac{L}{k}}&:= M^{\text{exc}}(1)  \, \ind_{\big\{X_1\leq \frac{L}{k}\big\}}+\sum_{r=2}^n M^{\text{exc}}(r)  \, \ind_{\big\{X_r\leq \frac{L}{k}\big\}}+ (-1)^{\tau_{r-1}}  \, 
\text{sign}(V_{\tau_{r-1}-1}) \, \varphi_1(L/k)\\
J_{n,\frac{L}{k}}&:= \varphi_1(L/k)\,  \sum_{r=2}^{n}  (-1)^{\tau_{r-1}}  \, 
\text{sign}(V_{\tau_{r-1}-1}).
\end{align}
Equation \ref{redf} guaranties that $\big(Q_{n,\frac{L}{k}}\big)_{n\in \N}$ is an $L^2$ martingale.  The proof of Claim \ref{c1} will 
therefore be proven once we show that 
\be{diffale3}
\lim_{k\to \infty} \limsup_{L\to \infty} 
\probmubeta{ \max_{n \leq cL^{1/3}} \big |Q_{n,\frac{L}{k}}\big| \ \geq \eta L^{1/3}}=0,
\ee

\be{diffale4}
\lim_{k\to \infty}\limsup_{L\to \infty} 
\probmubeta{ \max_{n \leq cL^{1/3}} \big |J_{n,\frac{L}{k}}\big| \ \geq \eta L^{1/3}}=0.
\ee

We begin with \eqref{diffale3} and we apply Doob inequality with the fact that $\big(Q_{n,\frac{L}{k}}\big)_{n\in \N}$ is an $L^2$ martingale to assert that 
there exists $c_1>0$ such that 
\be{eqaborn}
\probmubeta{ \max_{n \leq cL^{1/3}} \big | Q_{n,\frac{L}{k}} \big| \ \geq \eta L^{1/3}}\leq 
\frac{4 {\bf E}_{\beta,\mu_\beta}\Big [Q_{cL^{1/3},\frac{L}{k}}^{\,2}\Big]}{\eta^2 L^{2/3}}
\ee
and  that 
\begin{align}\label{newmajmar}
{\bf E}_{\beta,\mu_\beta}\Big [Q_{cL^{1/3},\frac{L}{k}}^{\,2}\Big]={\bf E}_{\beta,\mu_\beta}\Big [(M^{\text{exc}}(1))^2  &\, \ind_{\big\{X_1\leq \frac{L}{k}\big\}}\Big]\\
\nonumber &+\sum_{r=2}^{cL^{1/3}} {\bf E}_{\beta,\mu_\beta}\Big [ \big(M^{\text{exc}}(r)  \, \ind_{\big\{X_r\leq \frac{L}{k}\big\}}+ (-1)^{\tau_{r-1}}  \, 
\text{sign}(V_{\tau_{r-1}-1}) \, \varphi_1(L/k)\big)^2\Big]
\end{align}
At this stage, we recall \eqref{defexcmid} and \eqref{defR} which yield that $(M^{\text{exc}}(r))^2=R_r^{\,2}$ for every $r\geq 1$. Moreover $(R_r)_{r\geq 1}$ is an i.i.d. sequence of random variables and therefore, we deduce from \eqref{newmajmar} that 
\be{vaeub}
{\bf E}_{\beta,\mu_\beta}\Big[ Q_{cL^{1/3},\frac{L}{k}}^{\,2}\Big]\leq 2 cL^{1/3} \Big( {\bf E}_{\beta,\mu_\beta}\Big[ R_1^{\,2} \, \ind_{\big\{X_1\leq \frac{L}{k}\big\}}\Big]+ \varphi_1(L/k)^2\Big)\leq 4 c L^{1/3} \,  {\bf E}_{\beta,\mu_\beta}\Big[ R_1^{\,2} \, \ind_{\big\{X_1\leq \frac{L}{k}\big\}}\Big]
\ee
where the second inequality above is the result of Jensen inequality.
As a result, we need to bound from above the quality ${\bf E}_{\beta,\mu_\beta}\Big[ R_1^{\,2} \, \ind_{\big\{X_1\leq \frac{L}{k}\big\}}\Big]$, i.e., 
\begin{align}\label{ster}
{\bf E}_{\beta,\mu_\beta}\Big[ R_1^2 \, \ind_{\big\{X_1\leq \frac{L}{k}\big\}}\Big]
&={\bf E}_{\beta,\mu_\beta}\Big[\Big(-\frac{V_0}{2}+\sum_{i=1}^{\tau_1-1}(-1)^{i-1} \frac{V_i-V_{i-1}}{2}  
+(-1)^{\tau_1-2} \frac{V_{\tau_1-1}}{2}\Big)^2 \ind_{\big\{X_1\leq \frac{L}{k}\big\}}\Big].
\end{align}
At this stage, we substitute an expectation with respect to ${\bf P}_\beta$ to that w.r.t. ${\bf P_{\beta,\mu_\beta}}$ in the r.h.s. 
of \eqref{ster}. We proceed as follow. We define, for $y\in \N$, the set of excursions   $\mathcal{D}_y=\sup_{s\geq 1}  \mathcal{D}_{s,y}$ defined with 
\be{defD}
\mathcal{D}_{s,y}:=\cup_{j=0}^{s-2}\{(v_i)_{i=0}^{s-1}\colon\, v_i=0\ \forall i\leq j-1 \ \text{and}\  v_i>0\ \forall i \geq j\  \ \text{and}\ 
s+\sum_{i=0}^{s-1} |v_i|\leq x \} 
\ee
so that \eqref{ster} becomes 
\begin{align}\label{steravecD}
{\bf E}_{\beta,\mu_\beta}\Big[ R_1^2 \, \ind_{\big\{X_1\leq \frac{L}{k}\big\}}\Big]
=2 \sum_{s\geq 1} \sum_{v\in \mathcal{D}_{s,y}} {\bf P}_{\beta,\mu_\beta}\big((V_i)_{i=0}^{s-1}=&(v_i)_{i=0}^{s-1}\ \text{and} \ V_s\leq 0 \big) \\
\nonumber &\times \Big(-\tfrac{v_0}{2}+\sum_{i=1}^{s-1}(-1)^{i-1} \tfrac{v_i-v_{i-1}}{2}  
+(-1)^{s-2} \tfrac{v_{s-1}}{2}\Big)^2,
\end{align}
where the factor $2$ in front of the r.h.s. comes from the fact that negative and positive excursions contribute the same when computing  ${\bf E}_{\beta,\mu_\beta}\Big[ R_1^2 \, \ind_{\big\{X_1\leq \frac{L}{k}\big\}}\Big]$. At this stage, we recall \eqref{defubeta} and we observe that for all $v\in \mathcal{D}_{s,y}$ we have 
\be{equaliexc}
{\bf P}_{\beta}\big((V_{i+1})_{i=0}^{s-1}=(v_i)_{i=0}^{s-1}\ \text{and}\   V_{s+1}=0 \big)=\tfrac{1}{c_\beta} \big(\ind_{\{v_0=0\}}+\tfrac{1}{2} \ind_{\{v_0\neq 0\}}\big)\   {\bf P}_{\beta,\mu_\beta}\big((V_i)_{i=0}^{s-1}=(v_i)_{i=0}^{s-1}\ \text{and} 
\ V_s\leq 0 \big).
\ee
It remains to combine \eqref{steravecD} with \eqref{equaliexc} to obtain that there exit $c_2>0$ such that 
\begin{align}\label{ster2}
\nonumber{\bf E}_{\beta,\mu_\beta}\Big[ R_1^2 \, \ind_{\big\{X_1\leq \frac{L}{k}\big\}}\Big]
&\leq c_2 {\bf E}_{\beta}\Big[\Big(-\frac{V_1}{2}+\sum_{i=2}^{\tau_1-1}(-1)^{i}\,  \frac{V_i-V_{i-1}}{2}  
+(-1)^{\tau_1-1} \frac{V_{\tau_1-1}}{2}\Big)^2 \, \ind_{\big\{V_{\tau_1-1}>0,\, V_{\tau_1}=0\big\}} \ind_{\big\{X_1\leq 1+\frac{L}{k}\big\}}\Big]\\
&=\frac{c_2}{4} {\bf E}_{\beta}\Big[\big(\sum_{i=1}^{\tau_1} (-1)^{i}\,  U_i \big)^2
\, \ind_{\big\{V_{\tau_1-1}>0,\, V_{\tau_1}=0\big\}}  \, \ind_{\big\{X_1\leq 1+ \frac{L}{k}\big\}}\Big]:= \frac{c_2}{4}\,  \varphi_2(1+\tfrac{L}{k}).
\end{align}
For $x\in \N$, we decompose $\varphi_2(x)$ with respect to the value taken by $\tau_1$, i.e., 
\be{varp}
\varphi_2(x)=\sum_{s\geq 2} {\bf E}_{\beta}\Big[\big(\sum_{i=1}^{s} (-1)^{i}\,  U_i \big)^2
\, \ind_{\{V_{s-1}>0,\, V_{s}=0\}}\, \ind_{\{\tau_1=s\}}  \, \ind_{\{X_1\leq x\}}\Big]:=\sum_{s\geq 2} \alpha_{s,x}
\ee
We define for $j\in \N$ and $V\in \{0\}\times \Z^{\N}$ the geometric area seen from the minimum of $V$ after $j$ steps, i.e.,
\be{areamin}
A_{\text{min},j}(V)= j+\sum_{i=1}^{j} |V_i-\min\{0,V_1,\dots, V_j\}|
\ee
and we define also 
\begin{align}\label{defWT}
\cW_{s,x}&:=\{ (U_i)_{i=1}^s\colon\, V_{s-1}>0, \tau_1=s, V_s=0,\,  X_1\leq x\}\\
\nonumber \cO_{s,x}&:=\{(U_i)_{i=1}^s\colon\, V_s=0,\, A_{\text{min},s}(V)\leq x\}
\end{align}
and we apply to $\cW_{s,x}$ the $s$ shifts $\theta^j$, $j\in \{0,\dots,s-1\}$ defined by 
$$\theta^j(u_1,\dots,u_s)=(u_{j+1},\dots,u_s,u_1,\dots,u_j)$$
The crucial point here is that for every $(u_1,\dots,u_s)\in \cW_{s,x}$ and every $0\leq j\leq s-1$, 
\begin{enumerate}
\item[(a)] $\theta^j(u_1,\dots,u_s)\in \cO_{s, x}$
\item[(b)] ${\bf P}_{\beta}((U_1,\dots,U_s)=(u_1,\dots,u_s))={\bf P}_{\beta}((U_1,\dots,U_s)=\theta^j(u_1,\dots,u_s)) $
\item[(c)] $A_{\text{min},s}(u_1,\dots,u_s)=A_{\text{min},s}(\theta^j(u_1,\dots,u_s))$
\item[(d)] If $s\in 2\N$ for $(v_1,\dots,v_s)=\theta^j(u_1,\dots,u_s)$ we have $\big(\sum_{i=1}^{s} (-1)^{i}\,  v_i \big)^2=
\big(\sum_{i=1}^{s} (-1)^{i}\,  u_i \big)^2$
\item[(e)] For every   $(v_1,\dots,v_s)\in \cup_{j=0}^{s-1} \theta^j(\cW_{s,x})$ there exists a unique $0\leq j\leq s-1$ and a unique \\
$(u_1,\dots,u_s)\in \cW_{s,x}$ such that $(v_1,\dots,v_s)=\theta^j(u_1,\dots,u_s)$.
\end{enumerate}
As a consequence, for every $s\in 2 \N$ and $x\in \N$ we have the upper bound 
\be{thu}
\alpha_{s,x}\leq  \frac{1}{s} {\bf E}_{\beta}\bigg[\Big(\sum_{i=1}^{s} (-1)^{i}\,  U_i \Big)^2
\, \ind_{\{V_{s}=0\}} \, \ind_{\{A_{\text{min},s}(V)\leq x\}}\bigg],
\ee
and moreover, one can show that there exists a $c_3>0$ such that  for $s\in 2\N+1$ and $x\in \N$ we have 
\be{updj}
\alpha_{s,x}\leq c_3 \, \alpha_{s+1,x+1},
\ee
since it suffices to add one increment equal to $0$ in front of a trajectory from $\cW_{s,x}$  to obtain a trajectory from $\cW_{s+1,x+1}$.
We recall \eqref{varp} and, as a consequence of \eqref{thu} and \eqref{updj}, we can claim that 
\be{psi}
\varphi_2(x)\leq \psi(x)+c_3 \psi(x+1),\quad x\in \N,
\ee
 with
\be{varp2aux}
\psi(x):=\sum_{p\geq 1} \frac{1}{2p} {\bf E}_{\beta}\Big[\Big(\sum_{i=1}^{2p} (-1)^{i}\,  U_i \Big)^2
\, \ind_{\{V_{2p}=0\}} \, \ind_{\{A_{\text{min},2p}(V)\leq x\}}\Big]:= \sum_{p\geq 1} \frac{1}{2p} \gamma_{2p,x}.
\ee
We easily conclude that 
\begin{align}\label{timinv}
\nonumber \gamma_{2p,x}&\leq 2 {\bf E}_{\beta}\Big[\big(\sum_{i=1}^{p} (-1)^{i}\,  U_i \big)^2
\, \ind_{\{V_{2p}=0\}} \, \ind_{\{A_{\text{min},2p}(V)\leq x\}}\Big]+ 2 {\bf E}_{\beta}\Big[\big(\sum_{i=p+1}^{2p} (-1)^{i} \, U_i \big)^2
\, \ind_{\{V_{2p}=0\}} \, \ind_{\{A_{\text{min},2p}(V)\leq x\}}\Big]\\
\nonumber &=4 {\bf E}_{\beta}\Big[\big(\sum_{i=1}^{p} (-1)^{i} \, U_i \big)^2
\, \ind_{\{V_{2p}=0\}} \, \ind_{\{A_{\text{min},2p}(V)\leq x\}}\Big]\leq 4 {\bf E}_{\beta}\Big[\big(\sum_{i=1}^{p} (-1)^{i} \, U_i \big)^2
\, \ind_{\{V_{2p}=0\}} \, \ind_{\{A_{\text{min},p}(V)\leq x\}}\Big]\\
&:= 4\,  A_{p,x},
\end{align}
where the equality between the r.h.s. in  the first and in the second line is obtained by  time inversion. 
We also observe by applying Markov property at time $p$ that 
\be{urgs}
 A_{p,x}=\sum_{y\in \Z} {\bf E}_{\beta}\Big[\big(\sum_{i=1}^{p} (-1)^{i}\, U_i \big)^2
\ind_{\{V_{p}=y\}} \, \ind_{\{A_{\text{min},p}(V)\leq x\}}\Big] {\bf P_\beta}(V_p=y) 
\ee
and it remains to use a local central limit theorem in \cite[Theorem 3.5.2]{RD05} to claim that there exists a $c_5>0$ such that for every $y\in \Z$ we have that 
${\bf P}_\beta(V_p=y)\leq \frac{c_5}{\sqrt{p}}$. Finally 
\be{boundvp}
\psi(x)\leq c_5 \sum_{p\geq 1} \frac{1}{p^{3/2}}  {\bf E}_{\beta}\Big[\big(\sum_{i=1}^{p} (-1)^{i}\,  U_i \big)^2 
\, \ind_{\{A_{\text{min},p}(V)\leq x\}}\Big].
\ee
At this stage, we let $(\mathcal{G}_n)_{n\geq 0}$ be the natural filtration associated with $(U_i)_{i\in \N}$ and we  set 
\be{deftauti}
\widetilde \tau_x:=\inf \{j\geq 1\colon A_{\text{min},j}(V)\geq x\},
\ee
which is a stopping time with respect to  $(\mathcal{G}_n)_{n\geq 1}$.
For every $p\in \N$, the inequality  $A_{\text{min},p}(V)\leq x$ implies that 
$p\leq \tilde \tau_x$ and therefore 
\be{rgh}
{\bf E}_{\beta}\bigg[\Big(\sum_{i=1}^{p} (-1)^{i}\,  U_i \Big)^2 
\, \ind_{\{A_{\text{min},p}(V)\leq x\}}\bigg]\leq {\bf E}_{\beta}\bigg[\Big(\sum_{i=1}^{p\wedge \tilde\tau_x} (-1)^{i}\,  U_i \Big)^2 
\bigg].
\ee
Using that $\big[\big(\sum_{i=1}^{n} (-1)^{i}\,  U_i \big)^2-n  {\bf E}_{\beta}\big(U_1^2\big)\big]$ is a $(\mathcal{G}_n)_{n\geq 1}$ martingale, we can assert that for every $p\in \N$,
\be{rgh2}
 {\bf E}_{\beta}\bigg[\Big(\sum_{i=1}^{p\wedge \tilde\tau_x} (-1)^{i}\,  U_i \Big)^2\bigg]= {\bf E}_{\beta}\big( \tilde \tau_x\wedge p)\   {\bf E}_{\beta}\big(U_1^2\big).
\ee
Thus, (\ref{boundvp}--\ref{rgh2}) allow us to assert that there exists a $c_6>0$ and $c_7>0$  such that  
\begin{align}\label{varphib}
\psi(x)&\leq c_6 \sum_{p=1}^\infty \frac{1}{p^{3/2}}  {\bf E}_{\beta}\big( \tilde \tau_x\wedge p) =c_6 \sum_{p=1}^\infty \frac{1}{p^{3/2}}  {\bf E}_{\beta}\big( p \ind_{\{\tilde \tau_x\geq  p\}}+ \tilde \tau_x \ind_{\{\tilde \tau_x<p\}}\big)\\
\nonumber &=c_6  {\bf E}_{\beta}\bigg[ \sum_{p=1}^{\tilde\tau_x}\frac{1}{\sqrt{p}}+ \tilde \tau_x \sum_{p=\tilde\tau_x+1}^\infty \frac{1}{p^{3/2}}\bigg]\leq c_7\, 
 {\bf E}_{\beta}\Big(\sqrt{\tilde \tau_x}\Big)\leq c_7 \sqrt{{\bf E}_{\beta}\Big(\tilde \tau_x \Big)}.
\end{align}
so that finally \eqref{psi} and \eqref{varphib} yield that there exists $c_8>0$ such that for every $x\in \N$,
\be{boundphif}
\varphi_2(x)\leq c_8 \sqrt{{\bf E}_{\beta}\Big(\tilde \tau_{1+x} \Big)},\quad x\in \N.
\ee
At this stage, we combine \eqref{eqaborn}, \eqref{vaeub}, \eqref{ster2} with \eqref{boundphif} (at $x=L/k$) and we obtain that there exists $c_9>0$ such that 
$$\probmubeta{ \max_{n \leq cL^{1/3}} \big | Q_{n,\frac{L}{k}} \big| \ \geq \eta L^{1/3}}\leq 
c_9 \frac{  \sqrt{{\bf E}_{\beta}\Big(\tilde \tau_{2+\frac{L}{k}} \Big)}}{\eta^2 L^{1/3}}.$$
Thus, we complete the proof of \eqref{diffale3} with a straightforward application of  Lemma \ref{boundtauti} (proven in Step 5).

\medskip

We continue with the proof of \eqref{diffale4}. We apply Cauchy Schwartz to  \eqref{phi} and we recall \eqref{ster2} to conclude that  there exists  a $c_1>0$ such that 
\be{vp12}
\varphi_1(x)\leq {\bf E}_{\beta,\mu_\beta} \Big[R_1^2\, \ind_{\big\{X_1\leq x\big\}}\Big]^{1/2}\leq  c_1  \varphi_2(x)^{1/2}.
\ee
Then, we use \eqref{boundphif} and Lemma \ref{boundtauti} to conclude that there exists a $c_2>0$ such that 
\be{bvp}
\varphi_1(x)\leq c x^{1/6}, \quad x\in \N.
\ee
We recall 
$$J_{n,\frac{L}{k}}:= \varphi_1(L/k)\,  \sum_{r=2}^{n}  (-1)^{\tau_{r-1}}  \, 
\text{sign}(V_{\tau_{r-1}-1})$$
and therefore, \eqref{diffale4} will be proven once we show that 
\be{diffale42}
\bigg(\frac{1}{L^{1/6}}  \max_{n \leq cL^{1/3}}  \Big | \sum_{r=2}^{n}  (-1)^{\tau_{r-1}} \text{sign}(V_{\tau_{r-1}-1})\Big|\bigg)_{L\in \N}
\ee
is a tight sequence of random variables.

The idea to perform this proof consists in rewriting the sum in \eqref{diffale42} as a sum of i.i.d. centered random variable with a finite second moment. To that aim we set $r_0=0$ and for every $x\geq 0$ we define $r_{1+x}:=\min\{j\geq r_{x}+1 \colon\; V_{\tau_j}=0\}$. 
Then, for every $x\in \N_0$ we define $Y_x$ as
\begin{align}\label{defYX}
Y_x:&=\sum_{j=r_x+1}^{r_{x+1}} (-1)^{\tau_j} \,  \sign(V_{\tau_j -1})=(-1)^{\tau_{r_x}} \sum_{j=1}^{r_{x+1}-r_x} (-1)^{\tau_{j+r_x}-\tau_{r_x}} \,  \sign(V_{\tau_{j+r_x} -1}).
\end{align}
We have implicitly divided the $V$ trajectory into groups of excursions indexed by $x$. Except for the very first group ($x=1$) every other group begins with 
an excursion starting at $0$ and the sign of this first excursion is given by $\gep_{r_x+1}$. Then, the sign of the other excursions 
in the group are simply alternating so that the sign of the $(r_x+j)$-th excursion is
$$\sign(V_{\tau_{r_x+j}-1})=\gep_{r_x+1}\,  (-1)^{j-1}.$$
As a consequence, we may rewrite, for $x\geq 2$ 
\begin{align}\label{defYXbis}
Y_x=\gep_{r_x+1}\,  (-1)^{\tau_{r_x}} Z_x \quad \text{with}\quad Z_x:= \sum_{j=1}^{r_{x+1}-r_x} (-1)^{\tau_{j+r_x}-\tau_{r_x}+j-1}.
\end{align}
At this stage, we denote by $G_{\text{exc}}(x)$ the part of the $V$ trajectory (in modulus) made of the $r_{x+1}-r_{x}$ excursions
contained in the  group indexed by $x$, i.e.,
$$G_{\text{exc}}(x):=(|V_{\tau_{r_x}}|,|V_{\tau_{r_x}+1}|,\dots,|V_{\tau_{r_{x+1}}-1}|).$$
We easily observe that $(G_{\text{exc}}(x))_{x\geq 1}$ is i.i.d. We also observe that $Z_x$ is a function of $G_{\text{exc}}(x)$
only and that $r_{x+1}-r_{x}$ follows a geometric law with parameter $1-e^{-\beta/2}$ (that is ${\bf P}_\beta(V_1=0\, |\, V_1\geq 0)$). 
As a consequence, $(Z_x)_{x\geq 1}$ is an i.i.d. sequence of random variables with a finite second moment. We recall Remark \ref{remindep} which tells us that $(\gep_i)_{i\geq 0}$ is independent of 
$(G_{\text{exc}}(x))_{x\geq 0}$. Since for every $x\geq 0$ the random variable $(-1)^{\tau_{r_x}}$ is 
$\sigma(G_{\text{exc}}(j), j\geq 0)$ measurable and takes values $-1$ and $1$ only, the fact that $(\gep_i)_{i\geq 0}$ is an i.i.d. sequence 
of symmetric  Bernoulli trials  implies that 
$(\gep_{r_x+1}\,  (-1)^{\tau_{r_x}})_{x\geq 0}$ is also an i.i.d. sequence 
of symmetric Bernoulli trials independent of $\sigma(G_{\text{exc}}(j), j\geq 0)$. As a result,
$(Y_x)_{x\geq 1}$ is an i.i.d. sequence of centered random variables with a finite second moment. Thus, the tightness of the sequence of random variables in 
\eqref{diffale42} is a straightforward consequence of Donsker invariance principle.

\subsubsection*{Step 3: proof  of Claim \ref{c2}}


We set 
\be{compl}
B_{L,\eta}:=\Big\{\max_{j\leq cL^{1/3} \colon X_j\leq \frac{L}{k}} F_j < \eta L^{1/3}\Big\}=\bigcap_{j=1}^{cL^{1/3}} \big\{X_j >\tfrac{L}{k}\big\}\cup \Big\{ F_j<\eta L^{1/3},\,  X_j\leq \tfrac{L}{k}\Big\},
\ee
and we use \eqref{deftg} to recall that the sequence $(F_j)_{j\in \N}$ is i.i.d. so that 
\be{majoB}
{\bf P}_{\beta,\mu_\beta}(B_{L,\eta})=e^{cL^{1/3} \, \log \big(1- {\bf P}_{\beta,\mu_\beta}\big(F_1\geq \eta L^{1/3},\   X_1\leq \tfrac{L}{k}\big)\big)}.
\ee
Thus, \eqref{majoB} guarantees that the proof of Claim \ref{c2} will be complete once we show that for every $\eta>0$,
\be{limsupaux}
\lim_{k\to \infty} \limsup_{L\to \infty}\  L^{1/3}\   {\bf P}_{\beta,\mu_\beta}\Big(F_1\geq \eta L^{1/3},\   X_1\leq \tfrac{L}{k}\Big)=0
\ee
By using the mapping of trajectories introduced in (\ref{ster}--\ref{ster2}) we again substitute the law ${\bf P}_\beta$ to ${\bf P}_{\beta,\mu_\beta}$ 
in the r.h.s. of \eqref{limsupaux}. We indeed obtain that there exists a $c>0$ such that 
\be{boundbtoub}
 {\bf P}_{\beta,\mu_\beta}\Big(F_1\geq \eta L^{1/3},\   X_1\leq \tfrac{L}{k}\Big)\leq c\,   {\bf P}_{\beta}\big(C_{L,\eta}\big)
\ee
where
\be{CLeta}
C_{L,\eta}:=\Big\{F_1\geq \eta L^{1/3},\,  V_{\tau_1-1}>0, \, V_{\tau_1}=0,\, X_1\leq 1+\tfrac{L}{k}\Big\}
\ee
and with an alternative description of $F_1$, i.e., 
\be{altG1}
F_1:=\frac{1}{2} \max_{i\in \{1,\dots,\tau_1\}} \big| \sum_{s=1}^{i} (-1)^{s-1} \, U_s\big|.
\ee
We slightly modify the notations in \eqref{defWT}, i.e.,  for $x\in \N$,
\begin{align}\label{defWT2}
\tilde \cW_{s,x}&:=\{ (U_i)_{i=1}^s\colon\, V_{s-1}>0, \tau_1=s, V_s=0,\,  X_1\leq x, \, F_1\geq \eta L^{1/3}\}\\
\nonumber \tilde \cO_{s,x}&:=\{(U_i)_{i=1}^s\colon\, V_s=0,\, A_{\text{min},s}(V)\leq x, \, F_1\geq \tfrac{\eta}{3} L^{1/3}\}
\end{align}
 and we note that 
\be{tyu2}
C_{L,\eta}:=\cup_{s \geq 2 } \widetilde \cW_{s,\tfrac{L}{k}}.
\ee
We apply to $\cW_{s,x}$ the $s-1$ shifts $\theta^j$, $j\in \{0,\dots,s-1\}$ defined by 
$$\theta^j(u_1,\dots,u_s)=(u_{j+1},\dots,u_s,u_1,\dots,u_j)$$
The crucial point here is that for every $(u_1,\dots,u_s)\in \tilde \cW_{s,x}$ and every $0\leq j\leq s-1$, the properties
(a--c) and (e) stated below \eqref{defWT} are still satisfied here with $\tilde \cW_{s,x}$ and $\tilde \cO_{s,x}$ instead of $\cW_{s,x}$ and  $\cO_{s,x}$
whereas the (d) property  is replaced by 
\be{comparai}
\max_{j\in \{1,\dots,s\}} \big| \sum_{i=1}^{j} (-1)^{i-1} v_i \big|\geq \frac{1}{3}\max_{j\in \{1,\dots,s\}} \big| \sum_{i=1}^{j} (-1)^{i-1} u_i \big|,
 \ee
 with $(v_1,\dots,v_s)=\theta^j(u_1,\dots,u_s)$.
As a consequence, we obtain the following upper bound, 
\be{majoC}
{\bf P}_{\beta}(C_{L,\eta})\leq \sum_{s\geq 2} \frac{1}{s} {\bf P}_{\beta} \Big(F_1>\frac{\eta L^{1/3}}{3},\, V_s=0,\  A_{\text{min},s}\leq \frac{L}{k}\Big).
\ee 
At this stage, we consider a sequence of $s$ increments $(U_i)_{i=1}^s$ such that the associated $V$ trajectory satisfies $V_s=0$ and
$A_{\text{min},s}(V)\leq \frac{L}{k}$. Then, the $\tilde V$ trajectory defined by $\tilde V_i=V_{s-i}$ for $i\in \{0,\dots,s\}$ has increments 
$(-U_{s+1-i})_{i=1}^s$ also satisfies $\tilde V_s=0$ and   $A_{\text{min},s}(\tilde V)\leq \frac{L}{k}$. We can use this auxiliary trajectory and check easily that
$F_1\leq \max\{F_{1,1},F_{1,2}\}$ where 
\begin{align}\label{supG}
F_{1,1}&:=\max_{j\in \big\{1,\dots,\lfloor\frac{s}{2}\rfloor +1\big\}} \Big| \sum_{i=1}^j (-1)^{i-1} U_i\Big| \quad \text{and} \quad F_{1,2}:=\max_{j\in \big\{\lfloor \frac{s}{2}\rfloor+1,\dots,s\big\}} \Big| \sum_{i=1}^j (-1)^{i-1} U_i\Big|.
\end{align}
Moreover, 
\begin{align}\label{majoG2} 
\nonumber F_{1,2}&\leq \max_{j\in \big\{\lfloor \frac{s}{2}\rfloor +1,\dots,s\big\}} \Big| \sum_{i=1}^s (-1)^{i-1} U_i+ (-1)^{s+1} \sum_{i=1}^{s-j} (-1)^{i-1} (-U_{s+1-i}) \Big|\\
&\leq  \Big| \sum_{i=1}^s (-1)^{i-1} U_i\Big|+ \max_{j\in \big\{1,\dots,\lfloor \frac{s}{2}\rfloor +1\big\}} \Big|\sum_{i=1}^{j-1} (-1)^{i-1} (-U_{s+1-i}) \Big|:=\Big| \sum_{i=1}^s (-1)^{i-1} U_i\Big|+\tilde F_{1,1},
\end{align}
and a straightforward computations gives us that $| \sum_{i=1}^s (-1)^{i-1} U_i |\leq F_{1,1}+\tilde F_{1,1}$. Thus, 
\be{domG}
F_1\leq 3\max\{F_{1,1},\tilde F_{1,1}\}
\ee
and we note that, conditioned on $\tilde V_s=0$ and   $A_{\text{min},s}(\tilde V)\leq \frac{L}{k}$, the two random variables $F_{1,1}$ and $\tilde F_{1,1}$
have the same law. As a consequence, $F_1$ can be replaced by $F_{1,1}$ in the r.h.s of \eqref{majoC} and the proof will be complete once we prove that 
for every $\eta>0$,  
\be{lastthingto}
\lim_{k\to \infty} \limsup_{L\to \infty}\  L^{1/3}\  \sum_{s\geq 2} \frac{1}{s} \alpha_{\eta,\,s,L,k} =0
\ee
where 
$$\alpha_{\eta,\, s,L,k}:={\bf P}_{\beta} \Big(F_{1,1}>\eta L^{1/3},\, V_s=0,\  A_{\text{min},s}\leq \frac{L}{k}\Big).$$
By Markov inequality applied at time $t_s:=\lfloor \frac{s}{2}\rfloor+1$ we can write that for every $s\geq 2$, 
\be{supdeG}
\alpha_{\eta,\,s,L,k}\leq \sum_{y\in \Z} {\bf P}_{\beta} \Big(F_{1,1}>\eta L^{1/3},\, V_{t_s}=y,\  A_{\text{min},t_s}\leq \frac{L}{k}\Big)  {\bf P_\beta}(V_{s-t_s}=y) 
\ee
so that it  remains to use the local central limit Theorem in \cite[Theorem 3.5.2]{RD05} to claim that there exists a $c>0$ such that for every $y\in \Z$ we have that 
${\bf P}_\beta(V_{t_s}=y)\leq \frac{c}{\sqrt{s}}$ and to sum over $y$ to obtain 
\begin{align}\label{supdeG2}
\nonumber \sum_{s\geq 2} \frac{1}{s} \alpha_{\eta,s,L}&\leq c \sum_{s\geq 2}  \frac{1}{s^{3/2}}   {\bf P}_{\beta} \Big(F_{1,1}>\eta L^{1/3},\    A_{\text{min},t_s}\leq \frac{L}{k}\Big)\\
&\leq  2 c \sum_{p\geq 1}  \frac{1}{p^{3/2}}   {\bf P}_{\beta} \Big(\max_{j\in \{1,\dots, p+1\}} \Big| \sum_{i=1}^j (-1)^{i-1} U_i\Big| >\eta L^{1/3},\    A_{\text{min},p+1}\leq \frac{L}{k}\Big)\\
&\leq   c_2 \sum_{p\geq 1}  \frac{1}{p^{3/2}}   {\bf P}_{\beta} \Big(\max_{j\in \{1,\dots, p\}} \Big| \sum_{i=1}^j (-1)^{i-1} U_i\Big| >\eta L^{1/3},\    A_{\text{min},p}\leq \frac{L}{k}\Big)
\end{align}
where the second inequality in \eqref{supdeG2} is obtained by noting that $t_{2p}=t_{2p+1}$ for every $p\geq 1$. At this stage, we recall the definition of 
$\tilde\tau$ in \eqref{deftauti} and we recall also that 
for every $p\in \N$, the inequality  $A_{\text{min},p}(V)\leq \frac{L}{k}$ implies that 
$p\leq \tilde \tau_{L/k}$ and therefore 
\be{insidesub}
\Big\{\max_{j\in \{1,\dots, p\}} \Big| \sum_{i=1}^j (-1)^{i-1} U_i\Big| >\eta L^{1/3},\    A_{\text{min},p}\leq \frac{L}{k}\Big\}\subset
\Big\{\max_{j\in \{1,\dots, p\}} \Big| \sum_{i=1}^{j\wedge \tilde\tau_{L/k}} (-1)^{i-1} U_i\Big| >\eta L^{1/3}\Big\}.
\ee
Moreover $\tilde \tau_{L/k}$ is a stopping time and  $\big(\sum_{i=1}^{n} (-1)^{i-1} U_i \big)_{n\geq 1}$ is a martingale so that by Doob inequality
we can claim that 
\be{probinside}
{\bf P}_\beta\Big(\max_{j\in \{1,\dots, p\}} \Big| \sum_{i=1}^{j\wedge \tilde\tau_{L/k}} (-1)^{i-1} U_i\Big| >\eta L^{1/3} \Big)\leq \frac{{\bf E}_\beta\Big( \Big| \sum_{i=1}^{p\wedge \tilde\tau_{L/k}} (-1)^{i-1} U_i\Big|^2  \Big)}{\eta^2 L^{2/3}}= \frac{{\bf E}_{\beta}\big( \tilde \tau_{L/k}\wedge p)\   {\bf E}_{\beta}\big(U_1^2\big)}{\eta^2 L^{2/3}}
\ee
where we have used that $\big[\sum_{i=1}^{n} (-1)^{i-1} U_i \big)^2-n  {\bf E}_{\beta}\big(U_1^2\big)\big]_{n\geq 1}$  is a martingale.
At this stage, \eqref{lastthingto} becomes 
\be{verylast}
\lim_{k\to \infty} \limsup_{L\to \infty}\  L^{1/3}\  \sum_{p\geq 1} \frac{1}{p^{3/2}} \frac{{\bf E}_{\beta}\big( \tilde \tau_{L/k}\wedge p)\   {\bf E}_{\beta}\big(U_1^2\big)}{\eta^2 L^{2/3}}=0,
\ee
so that, by mimicking \eqref{varphib}, it remains to prove that 
\be{veryverylast}
\lim_{k\to \infty} \limsup_{L\to \infty}\  L^{1/3}\  \sqrt{ \tilde \tau_{L/k}}=0,
\ee 
which is a straightforward consequence of Lemma \ref{boundtauti} proven in Step 5 below.

\subsubsection*{Step 4: Lemma \ref{boundtauti}}
In this section, we state and prove a lemma that allows us to control the growth of $\widetilde \tau_x$ as $x\to \infty$.

\begin{lemma}\label{boundtauti}
For every $\beta>0$, there exists a $c>0$ such that 
${\bf E}_\beta(\tilde \tau_x)\leq c x^{2/3}$ for every $x\in \N$.
\end{lemma}

To prove the lemma we need to divide every $V$ trajectory into pseudo-excursions. To that aim, we define two sequences of random times, i.e.,  
$\eta_0=0$ and  for  every $i\in \N$,
\begin{align}\label{hattau}
\tilde \eta_i:=&\inf\{j\geq\eta_{i}\colon\, V_{j+1}> V_{j}\},\\
\nonumber \eta_{i}:=&\inf\{j>\tilde \eta_{i}\colon\, V_{j}\leq  V_{\tilde \eta_i}\}.
 \end{align}
 The pseudo-excursion indexed by $j\in \N$ is given by $(V_{\eta_{j-1}+1}, V_{\eta_{j-1}+2},\dots, V_{\eta_j})$ and we 
 associate it with the quantity 
\be{deftiX}
\tilde X_j=\eta_j-\eta_{j-1}+\sum_{i=\tilde \eta_j+1}^{\eta_j -1} V_i-V_{\tilde \eta_j}.
\ee
 At this stage, we observe that every pseudo-excursion starts with a non-increasing part of length $\tilde \eta$
followed by a real positive excursion (seen from $V_{\tilde \eta}$) of length $\eta-\tilde \eta$. The quantity 
$\tilde X$ corresponds to the total length of the pseudo-excursion plus the area swept by 
its real excursion. Henceforth, we will abusively call it the area of the pseudo-excursion.

For $n\in \N$, we denote by $m_n$ the number of pseudo-excursions that have been completed before time $n$ and by 
$a_n$ the number of pseudo-excursions that have been completed before their cumulated area reaches $n$, i.e., 
 \begin{align}\label{deftii}
 m_n&:=\max\{j\geq 0\colon\, \eta_j \leq n\},\\
 a_n&:=\max\{j\geq 0\colon \tilde X_1+\dots+\tilde X_j\leq n\}.
 \end{align}
At this stage, we define an increasing functional of the trajectory, i.e., 
$$R_n=\tilde X_1+\dots+\tilde X_{m_n}+(n-\eta_{m_n})+\sum_{i=\tilde \eta_{m_n +1}}^{n} V_i-V_{\tilde \eta_{m_n}}.$$
It is easy to see that $R_n$ is bounded above by $A_{\text{min},n}$ for every $n\in \N$ and every trajectory $V$ such that $V_0=0$. 
Therefore, by recalling the definition of $\widetilde \tau_x$ in \eqref{deftauti} and by defining 
\be{deftautihat}
\widehat \tau_x:=\inf \{j\geq 1\colon R_j(V)\geq x\},
\ee
we can claim that $\widehat \tau_x\geq \widetilde \tau_x$ for every $x\in \N$ and every $V$. For this reason, the proof of lemma \ref{boundtauti}
will be complete once we show that  there exists a $c>0$ such that 
\be{target}
{\bf E}_\beta(\widehat \tau_x)\leq c x^{2/3}, \quad  x\in \N. 
\ee 
To prove \eqref{target}, we use the equality 
$\hat \tau_x=\sum_{j=1}^{a_x} \eta_j-\eta_{j-1}+\hat \tau_x-\eta_{a_x}$ to rewrite the l.h.s. in \eqref{target} under the form  
\be{target2}
\frac{1}{x^{2/3}} {\bf E}_\beta(\widehat \tau_x)= {\bf E}_\beta\Bigg[\sum_{j=1}^{a_x}\frac{\eta_j-\eta_{j-1}}{\tilde X_j^{\, 2/3}}\,  \frac{\tilde X_j^{\, 2/3}}{x^{2/3}}\Bigg]
+\frac{1}{x^{2/3}} {\bf E}_\beta\big[\widehat \tau_x-\eta_{a_x}\big].
\ee 
The following claim shed lights on the fact that pseudo-excursions are almost i.i.d.
 \begin{claim}\label{subproc}
 Under ${\bf P}_\beta$ the pseudo-excursions, i.e., 
 $\big\{(V_{\eta_{j-1}+i}-V_{\eta_{j-1}})_{i=0}^{\eta_{j}-\eta_{j-1}}\colon\, j\geq 1 \big\} $ are i.i.d. Moreover, for every $j\geq 1$ the sequences 
$(V_{\eta_{j-1}+i}-V_{\eta_{j-1}})_{i=0}^{\tilde \eta_{j}-\eta_{j-1}}$ and $(V_{\tilde\eta_{j}+i}-V_{\tilde \eta_{j}})_{i=0}^{\eta_{j}-\tilde \eta_{j}}$ are independent as well. 
 \end{claim}
 \smallskip
 \begin{proof}
 The proof of the claim is a straightforward consequence of strong  Markov property combined with the fact that 
 $1+\tilde \eta_j$ and  $\eta_j$ are stopping times for every $j\geq 1$.
 \end{proof}

Because of Claim \ref{subproc} above,  $(\eta_j-\eta_{j-1})_{j\geq 1}$ and  $(\tilde X_j)_{j\geq 1}$ are sequences of i.i.d. random variables since, for every $j\geq 1$, both 
$\eta_j-\eta_{j-1}$ and $\tilde X_j$ are functions the $j$-th pseudo excursion only. Thus, by taking the conditional expectation with respect to $\{a_x,\tilde X_1,\dots,\tilde X_{a_x}\}$ inside both terms 
in the r.h.s. of \eqref{target2} we obtain that 
\begin{align}\label{target3}
\frac{1}{x^{2/3}} {\bf E}_\beta(\widehat \tau_x)&= {\bf E}_\beta\Bigg[\sum_{j=1}^{a_x}\,  {\bf E}_\beta \bigg(\frac{\eta_j-\eta_{j-1}}{\tilde X_j^{\, 2/3}}\ \Big| \ \tilde X_j\bigg) \ \frac{\tilde X_j^{\, 2/3}}{x^{2/3}}\Bigg]
+\frac{1}{x^{2/3}} {\bf E}_\beta\big[ {\bf E}_\beta\big (\widehat \tau_x-\eta_{a_x}\, |\,  R_{a_x}\big)\big]\\
\nonumber &= {\bf E}_\beta\Bigg[\sum_{j=1}^{a_x}\,  {\bf E}_\beta \bigg(\frac{\eta_1}{\tilde X_1^{\, 2/3}}\ \Big| \ \tilde X_1\bigg) \          \frac{\tilde X_j^{\, 2/3}}{x^{2/3}}\bigg]
+\frac{1}{x^{2/3}} {\bf E}_\beta\Big[ {\bf E}_\beta\big( \widehat \tau_ {x-R_{a_x}}\,  | \, \tilde X_1>x-R_{a_x}\big)\Big].
\end{align}
Thus, the proof of \eqref{target} is a consequence of the 3 inequalities displayed in Claims \ref{control},  that are also proven 
below. 
 
\begin{claim}\label{control}
For every  $\beta>0$ there exists a $C>0$ such that 
\be{up3}
 {\bf E}_\beta \bigg[\frac{\eta_1}{\tilde X_1^{\, 2/3}}\ \Big| \ \tilde X_1=x\bigg]\leq C,\quad x\in \N,
 \ee
  \be{up31}
 {\bf E}_\beta\Big[\frac{ \widehat \tau_ {x}}{x^{2/3}}\,  | \, \tilde X_1>x\Big]\leq C, \quad x\in \N,
 \ee
 \be{up32}
 {\bf E}_\beta\Bigg[\sum_{j=1}^{a_x}\,        \frac{\tilde X_j^{\, 2/3}}{x^{2/3}}\bigg]\leq C, \quad x\in \N.
 \ee
\end{claim}
\begin{proof}
We need to introduce some more notations to prove those three inequalities.
We recall \eqref{deuxx} and we note that, 
under ${\bf P}_{\beta}$ the first excursion $(V_0,\dots,V_{\mathfrak{N}_1-1})$ can also be  divided into  two independent processes, i.e., 
$(V_0,\dots,V_{\tilde{\mathfrak{N}}_1})$ and $(V_{\tilde{\mathfrak{N}}_1},\dots, V_{\mathfrak{N}_1-1})$ 
with
\be{deftii2}
\tilde{\mathfrak{N}}_1:=\max\{i\geq 1\colon\, V_1=V_2=\dots=V_i=0\}.
\ee
We can therefore rewrite $X_1=\tilde{\mathfrak{N}}_1+Z_1$ and $\tilde X_1=\tilde \eta_1+\tilde Z_1$ with
\begin{align}
Z_1&:=\mathfrak{N}_1-\tilde{\mathfrak{N}}_1+\sum_{i=\tilde{\mathfrak{N}}_1+1}^{\mathfrak{N}_1-1} |V_i| \quad \text{and} \quad \tilde Z_1:=\eta_1-\tilde \eta_1+\sum_{i=\tilde\eta_1+1}^{\eta_1-1} 
V_i-V_{\tilde \eta_1}
\end{align}
We observe  that $\tilde{\mathfrak{N}}_1$ and  $\tilde \eta_1$  both follow a geometric  law on $\N_0$ with parameter ${\bf P}_\beta(U_1=0)$ and 
${\bf P}_\beta(U_1 \leq 0)$ respectively.
Moreover, $(\mathfrak{N}_1-\tilde{\mathfrak{N}}_1,Z_1)$ and $(\eta_1-\tilde{\eta}_1,\tilde Z_1)$ have exactly the same law since $(|V_{\tilde{\mathfrak{N}}_1}|,\dots,|V_{\mathfrak{N}_1-1}|)$
and  $(V_{\tilde{\eta}_1}-V_{\tilde{\eta}_1},\dots,V_{\eta_1-1}-V_{\tilde{\eta}_1})$ themselves have the same law which can be seen also as the law of $(V_0,\dots,V_{\mathfrak{N}_1-1})$ under 
${\bf P}_\beta(\cdot\, |\, V_1>0)$.

We will also need the fact  that  there exists a $c>0$ such that  
\be{recall46}
{\bf P}_\beta(Y=x)=\frac{c}{x^{1/3}}(1+o(1)), \quad \text{for}\ \  Y\in  \{ X_1,\tilde X_1, Z_1,\tilde Z_1\}\ \ \text{and}\ \  x\in \N.  
\ee 
To prove \eqref{recall46}, we first observe that the local central limit theorem obtained in \cite[Lemma 4.6]{CP15}  (which remains true under ${\bf P}_{\beta}$)
is exactly  \eqref{recall46} for $X_1$.  From this,  we deduce that  \eqref{recall46} is true with $Z_1$ as well by using 
$${\bf P}_\beta(Y=x)=\frac{1}{2\pi} \int_{0}^{2\pi} \phi_{Y}(t) \, e^{-ixt} dt, \quad x\in \N, \quad Y\in \{X_1,Z_1\}$$
 ($\phi_{Y}$ being the characteristic function of $Y$) in combination with the fact that $X_1=\tilde{\mathfrak{N}}_1+Z_1$
and that $\tilde{\mathfrak{N}}_1$ has a geometric law and is independent of $Z_1$. As a consequence, \eqref{recall46} holds true for $\tilde Z_1$
also since $Z_1$ and $\tilde Z_1$ have the same law.   Finally, the fact that $\tilde X_1=\tilde \eta_1+\tilde Z_1$, that $\tilde \eta_1$ has a geometric law and is independent of $\tilde Z_1$
allows us to conclude that \eqref{recall46} is satisfied for $\tilde X_1$ as well.
\smallskip

For the sake of conciseness, we will not display the details of the proof of  \eqref{up3}. The reason is that the same inequality is proven in \cite[Lemma 4.11]{CP15}  with $\mathfrak N_1,X_1$ instead of $\eta_1,\tilde X_1$. Then, the equalities 
$X_1=\tilde{\mathfrak{N}}_1+Z_1$ and $\tilde X_1=\tilde \eta_1+\tilde Z_1$, the equality in law of  $(\mathfrak{N}_1-\tilde{\mathfrak{N}}_1,Z_1)$ and $(\eta_1-\tilde{\eta}_1,\tilde Z_1)$, the fact that $\tilde\eta_1$ and $\tilde{\mathfrak{N}}_1$
have geometric laws (and therefore light tails) and \eqref{recall46} are sufficient deduce \eqref{up3} from \cite[Lemma 4.11]{CP15}.
\smallskip

We continue with \eqref{up31}. We will assume for simplicity and in this proof only that $x^{1/3}\in 2\N$.  The case where 
$x^{1/3}\notin 2\N$ is taken care of similarly. We use \eqref{recall46} to claim that there exists a $c>0$ such that 
${\bf P}_\beta(\tilde X_1>x)=\frac{c}{x^{1/3}}(1+o(1))$
therefore, the proof of \eqref{up31} will be complete once we show that there exists a $C>0$ such that
\be{tbpl}
{\bf E}_\beta(\hat \tau_x \,  \ind_{\{\tilde X_1>x\}})\leq C x^{1/3},\quad x\in \N.
\ee
We introduce for every $y\in \N$ the stopping time 
\be{tautche}
\overline \tau_y:=\inf\Big\{j\geq 1\colon\, j+\sum_{i=1}^j V_{\tilde \eta_1+i}-V_{\tilde \eta_1}\geq y \Big\}
\ee
so that under the event  $\{\tilde X_1>x\}$ we have  $\hat \tau_x\leq \tilde \eta_1 \,  \ind_{\{\tilde \eta_1\geq x\}}
+(\tilde \eta_1+\overline \tau_{x-\tilde\eta_1}) \,  \ind_{\{\tilde \eta_1< x\}}$ and it allows us to rewrite \eqref{tbpl} as
\be{tbpl1}
{\bf E}_\beta\big(\hat \tau_x \,  \ind_{\{\tilde X_1>x\}}\big)\leq   {\bf E}_\beta\big(\tilde \eta_1  \big)
+{\bf E}_\beta\big(\overline \tau_{x-\tilde\eta_1} \,  \ind_{\{\tilde \eta_1< x\}}\,   \ind_{\{\tilde X_1>x\}}\big):=M_{1,x}+M_{2,x}.
\ee
The first term $M_{1,x}$ in the r.h.s. of \eqref{tbpl1} is simply bounded above by  $ {\bf E}_\beta\big(\tilde \eta_1)$ which is finite since 
$\tilde \eta_1$ has a geometric law.  Therefore, it remains to control $M_{2,x}$ and since $\tilde \eta_1$ is independent of $(V_{\tilde{\eta}_1}-V_{\tilde{\eta}_1},\dots,V_{\eta_1-1}-V_{\tilde{\eta}_1})$ 
which (as explained above) has the same law as $(V_i)_{i=0}^{\mathfrak{N}_1-1}$ under ${\bf P}_\beta(\cdot\,  |\, V_1>0)$ we 
simply rewrite (recall \eqref{taul})
\begin{align}\label{wrtpn}
\nonumber M_{2,x}&=\sum_{k=0}^{x-1} {\bf P}_\beta(\tilde \eta_1=k)\  {\bf E}_\beta\big(\overline \tau_{x-k}\,  \ind_{\{\tilde Z_1>x-k\}}\big) \\
&=\sum_{k=0}^{x-1} {\bf P}_\beta(\tilde \eta_1=k)\  {\bf E}_\beta\big(\xi_{x-k}\,  \ind_{\{X_1>x-k\}}\, |\, V_1>0\big).
\end{align}
From \eqref{wrtpn}, we deduce that the proof of \eqref{up31} will be complete once we show that there exists a $c>0$ such that 
\be{bornexcdep}
{\bf E}_\beta(\xi_x  \ind_{\{X_1>x\}}\, \ind_{\{U_1>0\}})\leq c x^{1/3},\quad x\in \N.\\
\ee

\noindent To prove \eqref{bornexcdep}, we write
${\bf E}_\beta(\xi_x  \ind_{\{X_1>x\}} \, \ind_{\{U_1>0\}})\leq \mathcal T_{1,x}+\mathcal T_{2,x}$
with 
\begin{align}
\mathcal T_{1,x}&={\bf E}_\beta\big[\xi_x \,  \ind_{\{V_{\xi_x}\leq 2 x^{1/3}\}}\,  \ind_{\{U_1>0\}} \,  \ind_{\{X_1>x\}}\big],\\
\nonumber \mathcal T_{2,x}&={\bf E}_\beta\big[\xi_x \,  \ind_{\{V_{\xi_x}> \frac{3}{2} x^{1/3}\}}\, \ind_{\{U_1>0\}} \, \ind_{\{X_1>x\}}\big].
\end{align}

The terms $\mathcal T_{1,x}$ and $\mathcal T_{2,x}$ are taken care of in a similar manner, but proving that  there exists a $c>0$ such that   $\mathcal T_{2,x}\leq c x^{1/3}$ for every $x\in \N$ is harder than proving the same inequality for $\mathcal T_{1,x}$. For this reason and for conciseness, we will only deal with  $\mathcal T_{2,x}$  here. 

  We let $L_{x^{1/3}}$ and $\tilde L_{\frac{3}{2}x^{1/3}}$ be the last time at which $V$
crosses $x^{1/3}$ before time $\xi_x$ and the first time (after $L_{x^{1/3}}$) at which $V$ crosses $\frac{3}{2}x^{1/3}$, i.e., 
\begin{align}\label{defTx}
L_{x^{1/3}}&:=\max\{j\in \{0,\dots,\xi_x\}\colon\, V_j\leq x^{1/3}\}\\
\nonumber \tilde L_{\frac{3}{2} x^{1/3}}&:=\min\{j\geq T_{x^{1/3}}+1\colon\, V_j\leq \tfrac{3}{2} x^{1/3}, V_{j+1}> \tfrac{3}{2} x^{1/3} \}.
\end{align}
We observe that $\xi_x-L_{x^{1/3}}\leq x^{2/3} $ since after $L_{x^{1/3}}$ the trajectory remains above $x^{1/3}$ up to time $\xi_x$. Therefore,
there exists a $C>0$ such that 
${\bf E}_\beta\big[(\xi_x-L_{x^{1/3}}) \,  \ind_{\{U_1>0\}} \, \ind_{\{X_1>x\}}\big]\leq C x^{1/3}$ for every $x\in \N$ and we can safely substitute $L_{x^{1/3}}$ to 
$\xi_x$ in the definition of $\mathcal{T}_{2,x}$. We define also 
\begin{align}\label{defOO}
\nonumber \cO_x:=\big\{(j_1,j_2,y_1,y_2,z_1,z_2)\in (\N_0)^6\colon\, & j_1< j_2\leq x,\  \ z_1\leq z_2\leq x\ \  \text{and}\ \   y_1,y_2\in \big(x^{1/3},\tfrac{3}{2} x^{1/3}\big]\big\} 
\end{align}
and we split the expectation defining $\mathcal{T}_{2,x}$ depending on the values $(j_1,j_2)$ taken by $(T_{x^{1/3}},\tilde  T_{\frac{3}{2} x^{1/3}})$ and $(y_1,y_2)$ taken by 
$(V_{j_1+1},V_{j_2})$ and also $(z_1,z_2)$ taken by $(K_{j_1},K_{j_2})$. Thus, 
\be{eqmast}
\mathcal T_{2,x}=\sum_{(j_1,j_2,y_1,y_2,z_1,z_2)\in \cO_x}\   \sum_{y_3>\frac{3}{2} x^{1/3}} j_1 \, 
{\bf P_\beta}(W_{j_1,y_1,z_1} \cap I_{j_1,j_2,y_2,z_2} \cap J_{j_2,y_3}),
\ee
with 
\begin{align}
W_{j_1,y_1,z_1}&=\big\{V_i>0, \forall i \in \{1,\dots,j_1\}, \   V_{j_1}\leq x^{1/3},\, 
K_{j_1}=z_1,\  V_{j_1+1}=y_1\big\},\\
\nonumber I_{j_1,j_2,y_2,z_2}&=\big\{ x^{1/3}<V_i\leq  \tfrac{3}{2} x^{1/3},\,  \forall i\in \{j_1+1,\dots,j_2\},\,  V_{j_2}=y_2,\,  K_{j_2}=z_2\big\},\\
\nonumber J_{j_2,y_3}&=\big\{ V_{j_2+1}=y_3,  V_i> x^{1/3},\,  \forall i\in \{j_2+1,\dots,\xi_x\},\,  V_{\xi_x}>\tfrac{3}{2} x^{1/3}\big\}.
\end{align}
Provided we change the equality into an inequality,  we can safely restrict the event $J_{j_2}$ in the r.h.s. of \eqref{eqmast} to $\{V_{j_2+1}=y_3\}$. Then, we apply Markov property at time $j_1+1$
and $j_2$
and we obtain 
\begin{align}\label{eqmast2}
\mathcal T_{2,x}\leq \sum_{(j_1,j_2,y_1,y_2,z_1,z_2)\in \cO_x}\   \sum_{y_3>\frac{3}{2} x^{1/3}} j_1 \, 
&{\bf P_\beta}(W_{j_1,y_1,z_1})\  {\bf P}_{\beta,y_1} \big(\tilde I_{j_1,j_2,y_2,z_1,z_2}\big)\ 
{\bf P}_{\beta,y_2} (V_1=y_3),
\end{align}
with 
\be{defiti}
\tilde I_{j_1,j_2,y_2,z_1z_2}=\big\{x^{1/3}<V_i\leq  \tfrac{3}{2} x^{1/3},\,  \forall i\leq j_2-j_1-1,\,  V_{j_2-j_1-1}=y_2,\,  K_{j_2-j_1-1}=z_2-z_1-1-y_1 \big\}.
\ee
At this stage, we map every trajectory $(V_i)_{i=0}^{j_1+1}$ taken account in $W_{j_1,y_1,z_1}$ onto an associated path which  equals 
$V$ up to time $j_1$, then touches 
$x^{1/3}$ at time  $j_1+1$ and is  equal to $2 x^{1/3}-y_1$ at time $j_1+2$ (i.e. we reflect $V_{j_1+1}$ with respect to $x^{1/3}$ to obtain the position of the image of 
$V$ at time $j_1+2$). Thus, we obtain a new set 
\be{altWti}
\tilde W_{j_1,y_1,z_1}
=\big\{V_i>0, \forall i \in \{1,\dots,j_1\}, \   V_{j_1}\leq x^{1/3},\, 
K_{j_1}=z_1,\ V_{j_1+1}=x^{1/3},  V_{j_1+2}=2 x^{1/3}-y_1\big\},
\ee
so that ${\bf P_\beta}(W_{j_1,y_1,z_1})=c_\beta {\bf P_\beta}(\tilde W_{j_1,y_1,z_1})$. We also reflect every
piece of trajectory in $\tilde I_{j_1,j_2,y_2,z_1,z_2}$ with respect to $x^{1/3}$ and we denote by $\widehat I_{j_1,j_2,y_2,z_1,z_2}$
the set containing the resulting paths, thus 
\begin{align}\label{unemajotra}
 {\bf P}_{\beta,y_1} \big(\tilde I_{j_1,j_2,y_2,z_1,z_2}\big)=  {\bf P}_{\beta,\, 2 x^{1/3}-y_1} \big(\widehat I_{j_1,j_2,y_2,z_1,z_2} \big)
\end{align}
and
$$\widehat I_{j_1,j_2,y_2,z_1,z_2}\subset \big\{ \tfrac{1}{2} x^{1/3}\leq   V_i<x^{1/3}, \  \forall i \leq j_2-j_1-1,\ 
 V_{j_2-j_1-1}=2x^{1/3}-y_2, \, 
  K_{j_2-j_1-1}\leq z_2-z_1\big\}.$$
With the help of (\ref{altWti}--\ref{unemajotra}) and by  summing over $y_3$ in \eqref{eqmast2} we obtain
\begin{align}\label{eqmast3}
\mathcal T_{2,x}\leq c_\beta \sum_{(j_1,j_2,y_1,y_2,z_1,z_2)\in \cO_x}\  j_1 \, 
&{\bf P_\beta}(\tilde W_{j_1,y_1,z_1})\  {\bf P}_{\beta,2x^{1/3}-y_1} \big(\widehat I_{j_1,j_2,y_2,z_1,z_2}\big)\ 
{\bf P}_{\beta} (U_1> \tfrac{3}{2} x^{1/3}-y_2).
\end{align}
Since $U_1$ has a geometric law, there exists a $c>0$ such that
${\bf  P}_{\beta} (U_1> y)=c {\bf  P}_{\beta} (U_1=y)$  for every $y\geq 0$. Therefore, there exists a constant $c>0$ such that we can substitute 
${\bf P}_{\beta} (U_1= y_2-\tfrac{3}{2} x^{1/3})$
to ${\bf P}_{\beta} (U_1> \tfrac{3}{2} x^{1/3}-y_2)$ in the r.h.s. in \eqref{eqmast3} and write
\begin{align}\label{eqmast4}
\mathcal T_{2,x}\leq c \sum_{(j_1,j_2,y_1,y_2,z_1,z_2)\in \cO_x}\  j_1 \, 
&{\bf P_\beta}(\tilde W_{j_1,y_1,z_1})\  {\bf P}_{\beta,2x^{1/3}-y_1} \big(\widehat I_{j_1,j_2,y_2,z_1,z_2} \cap \{V_{j_2-j_1}=\tfrac{1}{2} x^{1/3}\} ).
\end{align}
At this stage, we let $C_{j_1,j_2,y_1,y_2,z_1,z_2}$ be the set of paths obtained by concatenating trajectories in $\tilde W_{j_1,y_1,z_1}$ and in $\widehat I_{j_1,j_2,y_2,z_1,z_2} \cap \{V_{j_2-j_1}=\tfrac{1}{2} x^{1/3}\}$, i.e., 
\begin{align}\label{defCij}
\nonumber C_{j_1,j_2,y_1,y_2,z_1,z_2}\subset\{V_i>0, &\forall i \in \{1,\dots,j_1\}, V_{j_1}\leq x^{1/3}, K_{j_1}=z_1, V_{1+j_1}=x^{1/3}, V_{2+j_1}=2x^{1/3}-y_1    \\
\nonumber & \tfrac{1}{2} x^{1/3}\leq V_i <x^{1/3}, \forall i\in \{2+j_1,\dots,j_2\}, V_{j_2+1}=2 x^{1/3}-y_2, \\
&K_{1+j_2}\leq z_2+ x^{1/3}, V_{j_2+2}=\tfrac{1}{2} x^{1/3}\},
\end{align} 
and it is fundamental to note the $C_{j_1,j_2,y_1,y_2,z_1,z_2}$ with $(j_1,j_2,y_1,y_2,z_1,z_2)\in \cO_x$ are disjoint. The final step of this proof consists 
in attaching at the end of every path in a given $C_{j_1,j_2,y_1,y_2,z_1,z_2}$ another path which will reach the lower half-plane in such a way that 
the area swept by the whole excursion belongs to $[\frac{x}{2}, 2x]$.
We continue this computation by noticing again that by Donsker Theorem 
\be{applido}
\lim_{x\to \infty} {\bf P}_{\beta,\frac12 x^{1/3}}(V_i<x^{1/3}, \forall i\leq \tau, K_\tau\in [\tfrac{x}{2},x])={\bf P}_{B_0=\frac{1}{2}}(B_s<1,\, \forall s\in [0,\tau_B], \,A_{\tau_B}\in [\tfrac{1}{2},1])>0,
\ee
and therefore we can assert that there exists a $c>0$ such that for every $x\in \N$, it holds that
\begin{align}\label{eqmast5}
\mathcal T_{2,x}&\leq c \sum_{(j_1,j_2,y_1,y_2,z_1,z_2)\in \cO_x}\  j_1 \, 
{\bf P_\beta}(C_{j_1,j_2,y_1,y_2,z_1,z_2})\, \,  {\bf P}_{\beta,\frac12 x^{1/3}}\big(V_i<x^{1/3}, \forall i\leq \tau,\,  K_\tau\in \big[\tfrac{x}{2},x\big]\big)\\
\nonumber &\leq c  \, {\bf E}_{\beta}\Big [ T_{\text{max},x^{1/3}} \, \ind_{\{U_1>0\}} \,  \ind_{\big\{X_1\in \big[\tfrac{x}{2}, 3x\big]\big\}}\Big ]\leq c\,   {\bf E}_{\beta}\Big[\mathfrak{N}_1 \,  \ind_{\big\{X_1\in \big[\tfrac{x}{2}, 3x\big]\big\}}\Big]
\end{align}
with $T_{\text{max},x^{1/3}}:=\max\{i\leq \mathfrak{N}_1 \colon V_i\geq x^{1/3}\}$. This provides us the expected upper bound on $\mathcal{T}_{2,x}$ and the proof of \eqref{up31} is complete.

\end{proof}

We conclude this section with the proof of  \eqref{up32}.
To begin with we recall the definition of $a_x$ in \eqref{deftii} and  we note that by definition of $a_x$  
\be{bornesimp}
{\bf E}_\beta\Bigg[\sum_{j=1}^{a_x}\,        \tilde X_j^{\, 2/3}\bigg]={\bf E}_\beta\Bigg[\sum_{j=1}^{a_x}\,        \tilde X_j^{\, 2/3}  \, 
\ind_{\{\tilde X_j\leq x\}}\bigg]\leq {\bf E}_\beta\Bigg[\sum_{j=1}^{1+a_x}\,        \tilde X_j^{\, 2/3}  \, 
\ind_{\{\tilde X_j\leq x\}}\bigg]. 
\ee
Since $1+a_x$ is a bounded stopping time with respect to the filtration $(\sigma(\tilde X_1,\dots,\tilde X_n))_{n\geq 1}$ and since 
$\big(\sum_{i=1}^n \tilde X_i \, \ind_{\{\tilde X_j\leq x\}}-n {\bf E}_{\beta,\mu_\beta}\big(\tilde X_1 \,  \ind_{\{\tilde X_1\leq x\}}\big)\big)_{n\geq 1}$ is a martingale, 
we can rewrite  \eqref{bornesimp} as
\be{membdroi}
{\bf E}_\beta\Bigg[\sum_{j=1}^{a_x}\,        \tilde X_j^{\, 2/3} \,  \ind_{\{\tilde X_j\leq x\}}\bigg]\leq {\bf E}_{\beta}(1+a_x)\  {\bf E}_{\beta} \Big(\tilde X_1^{2/3} \,  \ind_{\{\tilde X_1\leq x\}}\Big). 
\ee
A straightforward computation with the help of \eqref{recall46} guaranties that there exists a $c>0$ such that  
$ {\bf E}_{\beta} \Big(\tilde X_1^{2/3}  \ind_{\{\tilde X_1\leq x\}}\Big)\leq c x^{1/3}$ for every $x\geq 1$. Then, it remains to 
use \eqref{recall46} for $\tilde X_1$ to conclude that  exists a $c>0$ such that  ${\bf E}_{\beta}(a_x)\leq c x^{1/3}$ for every $x\in \N$.

\subsection{Proof of Proposition ~\ref{propfin}} \label{prpropfin}

We shall prove first the second limit.
Thanks to the independence of $D$ and $B$ the process
$$ N(s):=D(s) -\tilde{D}^k(s) = \int_0^{a_s} \un{u\notin \Gamma_k}\,
dD_u$$
is an $L^2$ martingale, and by Doob maximal inequality
\begin{align*}
  \prob{\sup_{s\in\etc{0,1}}\valabs{N(s)} \ge \epsilon}& \le
  \frac{C}{\epsilon^2} \esp{N(1)^2} \\
&= \frac{C}{\epsilon^2} \int_0^\infty \prob{u \notin \Gamma_k, u <
  a_1}\, du
\end{align*}

Since the excursion process of Brownian motion is sigma-finite, we
have forall $u$ : $\prob{u \notin \Gamma_k} \to 0$. Therefore we can
  conclude by dominated convergence if we can prove that $\esp{a_1}<
  +\infty$. This is indeed true since $Y_t = \valabs{B}_{a(t)}$
  conditioned by $Y(1)=0$ is distributed as 
  $\etp{\frac32 \rho_t}^{2/3}$ with $(\rho_t)_{0\le t\le 1}$ a Bessel
    bridge of dimension $\delta=4/3$. Therefore,
    \begin{align*}
      \esp{a_1} &= \esp{\int_0^1 \frac{ds}{Y_s}}= C \int_0^1 \esp{\rho_t^{-2/3}}\, dt
    \end{align*}
Since $\rho_t$ has density
$$ C_\delta  (t(1-t))^{-\delta/2} \exp(- \frac{x^2}{2}
(\unsur{t} + \unsur{1-t}))\,,$$
we see that, by symmetry, 
\begin{align*}
  \esp{a_1} &\le C  \int_0^{\undemi}  t^{-\delta/2}\etp{ \int_0^\infty
               x^{-2/3} e^{-\frac{x^2}{2t}}}\, dt  <+\infty \,.
\end{align*}

We now prove the first limit of Proposition ~\ref{propfin}. Given $\eta>0$, since $\esp{a_1} <
  +\infty$ there exists $A=A(\eta)$ such that 
$$ \prob{a_1 \ge A} \le \undemi \eta\,.$$
By Kolmogorov's Continuity Criterion given $0< \gamma < \undemi$ the
random variable
$$ K := \sup_{0\le s<t<A} \frac{\valabs{{B}_t
    -{B}_s}}{\valabs{t-s}^\gamma}$$
 has a small moment : there exists $\delta>0$ such that $\esp{K} <
 +\infty$.

If $s\le 1$ then, on $\ens{a_1 < A}$,
\begin{align*}
 \valabs{\tilde{B}(s) - \tilde{B}^k(s)} &=
\valabs{{B}(a_s)} \un{a_s \notin\Gamma_k}\\
&= \valabs{{B}(a_s-{B}(g(a_s)} \un{a_s
  \notin\Gamma_k}\\
&\le \valabs{a_s - g(a_s)}^\gamma K \un{a_s \notin\Gamma_k}\\
&\le \unsur{k^\gamma} K
\end{align*}
Therefore

\begin{align*}
  \prob{ \sup_{s\in [0,1]} \big| \widetilde B(s)-\widetilde
    B^k(s)\big|\geq \gep }&
\le \prob{a_1 \ge A} + \prob{\sup_{s\in [0,1]} \big| \widetilde B(s)-\widetilde
    B^k(s)\big|\geq \gep; a_1 < A} 
\\
&\le \undemi \eta + \prob{K \ge k^\gamma
    \epsilon} \le \undemi \eta + \unsur{(k^\gamma \epsilon)^\delta}
  \esp{K^\delta} < \eta\,,
\end{align*}
for $k$ large enough.

  \bibliographystyle{imsart-nameyear}
  \bibliography{cnp}

\end{document}